\documentclass{article}

\usepackage{amssymb}
\usepackage{amsmath}
\usepackage{amsthm}
\usepackage[T1]{fontenc}
\usepackage{algpseudocode}
\usepackage{algorithm}
\usepackage{mathtools}
\usepackage{xcolor}
\usepackage{booktabs}
\usepackage{longtable}
\usepackage{array}
\usepackage{float}
\usepackage{placeins}

\usepackage{authblk}
\makeatletter
\let\maketitle\AB@maketitle
\makeatother

\usepackage[margin=1in]{geometry}
\usepackage{graphicx}
\usepackage{caption}
\usepackage{subcaption}
\usepackage[hidelinks]{hyperref}

\newtheorem{theorem}{Theorem}[section]
\newtheorem{lemma}[theorem]{Lemma}
\newtheorem{proposition}[theorem]{Proposition}
\newtheorem{corollary}[theorem]{Corollary}

\theoremstyle{definition}
\newtheorem{definition}[theorem]{Definition}

\theoremstyle{remark}
\newtheorem{remark}[theorem]{Remark}
\newtheorem*{remark*}{Remark}

\title{The Dynamics of RMSProp Beyond the Stability Threshold}
\author[1]{Kyle Loh\thanks{\raggedright JDMJ was partially supported by NSF Grant DMS-2307987. ORCID iDs: Kyle Loh, \url{https://orcid.org/0009-0000-6625-3770}; Jason Mireles-James, \url{https://orcid.org/0000-0001-8449-9408}; Zhenghao Xu, \url{https://orcid.org/0000-0001-8076-5166}; Molei Tao, \url{https://orcid.org/0000-0002-3308-6176}.}}
\author[1]{Jason Mireles-James}
\author[2]{Zhenghao Xu}
\author[3]{Molei Tao}

\affil[1]{\small Department of Mathematics and Statistics, Florida Atlantic University}
\affil[2]{\small NVIDIA}
\affil[3]{\small School of Mathematics, Georgia Institute of Technology}
\date{}

\begin{document}

\maketitle

\begin{abstract}
We study what happens to RMSProp when the learning rate is too large for convergence to the minimizer. For positive-definite quadratic loss functions, we show that the minimizer is Lyapunov unstable above an explicit critical learning rate. Nevertheless, every trajectory eventually enters a compact trapping region. We obtain explicit bounds proportional to \(h\) in the spatial variables and to \(h^2\) in the memory variables, where \(h\) is the learning rate. We also prove that nontrivial period-two orbits exist precisely above the critical learning rate, and under additional spectral assumptions, an attracting branch of these orbits emerges from the minimizer. For sufficiently large learning rates, we construct branches near admissible corners of a cube. We then explore the dynamics inside the trapping region using the parameterization method and numerical continuation, exploring invariant circles and tori, coexistence of attracting sets, circle-doubling sequences, and chaotic motions. We examine how the dynamics depend on the loss Hessian, memory decay, and dimension. Finally, we test a learning-rate reduction schedule and continue selected invariant circles into an autonomous variant of the Adam optimizer.
\end{abstract}

{\raggedright\noindent\textbf{Keywords.} RMSProp, global trapping region, periodic orbits, invariant circles, \mbox{Neimark--Sacker bifurcations}.\par}

\section{Introduction} \label{sec:intro}

Consider the problem of minimizing a smooth function 
$L\colon \mathbb{R}^d \to \mathbb{R}$.
The classic gradient descent algorithm is to
choose $x_0 \in \mathbb{R}^d$, $h > 0$, and iterate 
the sequence
\begin{equation} \label{eq:gradD}
x_{n+1} = x_n - h \nabla L(x_n), \quad \quad \quad n \geq 0.
\end{equation}
Since Equation \eqref{eq:gradD} is an Euler iteration
for the ordinary differential equation $x' = -\nabla L(x)$, 
whose solution curves have
\[
\frac{d}{dt} L(x(t)) = - \| \nabla L(x(t)) \|^2,
\]
a Taylor expansion shows that, at any noncritical point, there exists a 
sufficiently small $h$ so that an iteration of the algorithm decreases $L$. 
Moreover, if $\nabla L$ is uniformly Lipschitz on a neighborhood of 
$\left\{ x : L(x) \leq L(x_0)\right\}$, 
then there is an $h > 0$ small enough so that $L$ decreases
along the sequence $\{x_n\}_{n \geq 0}$.  We adopt the convention that 
$L$ is referred to as the loss function and 
$h$ as the learning rate (see Section \ref{rem:intro_AI} below).

The algorithm is simple, involving only first partial derivatives 
and elementary vector operations.  It scales well with dimension as, 
apart from the cost of evaluating the gradient,
memory usage and arithmetic operations are linear in the number of variables.  
Moreover, convergence of the algorithm in the small $h$ limit is particularly 
transparent, as already indicated above.
These facts make gradient descent an invaluable tool in numerical optimization,
where it is not only useful in its own right but 
serves as the foundation for many other, more sophisticated methods.

This being said, gradient descent has well-known shortcomings.  
Applying a single learning rate to all components of the gradient  
leads to poor performance in problems where $L$ exhibits anisotropic 
curvature.  The rate of convergence is only linear in the neighborhood of 
a local minimum, and is especially slow for ill-conditioned Hessians. 
Similarly, flat regions, where the gradient remains small,
 degrade algorithm performance and are difficult to escape.

These issues lead naturally to interest in adaptive step modifications
 like AdaGrad~\cite{duchi2011adaptive}, 
 RMSProp~\cite{tieleman2012lecture},
Adam~\cite{kingma2015adam}, and Muon \cite{jordan2024muon}.
 RMSProp is a commonly used modification of gradient descent
 and is the main focus of the present work.  
 The algorithm is defined formally in Section \ref{sec:algdef}, but
 the idea is to adapt to the geometry of $L$ by 
dividing each component of the gradient by the square root of an exponentially weighted average of 
its square. Introducing an additional ``memory'' state variable to track these averages 
gives a new discrete time dynamical system in twice as many variables. 
We refer to the original $x \in \mathbb{R}^d$ variable as the spatial 
variable, and to the new variable as the memory variable.  
A small numerical stability constant $\epsilon$ is added to all denominators to avoid division by zero. 
Despite its relative simplicity, RMSProp is the basis of more elaborate methods like Adam 
(discussed formally in Section~\ref{sec:Adam}).
See Section \ref{rem:dynamicsRMSProp} for some additional context.

From a dynamical systems perspective, 
the starting point for the discussion of RMSProp is the convergence result of 
 Dimitrieski et al.~\cite[Theorem~1]{dimitrieski2026rmsprop}. The result is restated  
 precisely as Theorem \ref{thm:dimitrieski_global_stability} below but, briefly,
it says that if $L$ is strongly convex, then there exists
 a critical learning rate $0 < h_*$ so that for $0 < h < h_*$  
the RMSProp algorithm converges to the global minimizer for $L$
when initialized from any initial point.  
The critical value $h_* > 0$ is determined explicitly by the Lipschitz constant of the gradient
 of $L$.

The present work studies the dynamics of 
RMSProp for the quadratic loss functional $L(x)=\tfrac12x^\top Ax$, with $A$ a symmetric 
and positive definite matrix (SPD).  The motivation for this simplification is discussed in 
Section \ref{rem:quadraticLoss}. However, we note that the stability theorem 
of \cite{dimitrieski2026rmsprop} gives explicitly a critical learning rate of 
\begin{equation} \label{eq:h_star}
h_* = \frac{2 \epsilon}{\lambda_{\max} (A)}, 
\end{equation}
for the quadratic case. Here $\lambda_{\max} (A)$ is the largest 
eigenvalue of $A$, and $\epsilon$ is the numerical stability constant. (Again, we refer to 
Section \ref{sec:algdef} for the precise definition of the map).
For quadratic loss functions, we establish two global stability/instability results 
which complement the existing small $h$ results of \cite{{dimitrieski2026rmsprop}}. 
 
 First, we show that for all $h > h_*$, quadratic RMSProp is Lyapunov unstable at the origin. 
 There are therefore initial states arbitrarily close to the minimizer whose iterates leave a fixed neighborhood of it.
The proof of this result is delicate because, as seen clearly in Section \ref{sec:algdef}, the 
version of RMSProp we work with is not $C^1$ at the origin.
Because of this, standard linear stability analysis is not available, and arguments
specially adapted to RMSProp are needed.

Nevertheless, we show that for every $h > h_*$, there is a compact trapping region 
which absorbs all iterates of the algorithm, regardless of initial condition.
 Moreover, we obtain explicit bounds on a box in phase space, 
 whose size depends explicitly on $h$, and which contains the trapping region.
 We find that the sides of the bounding box are proportional to $h$ in the 
 spatial variables, and to $h^2$ in the 
 memory variables.  
 
Taken together, these results lead to a practical observation about 
quadratic RMSProp, and constitute one of the main results of the present work.
That is, \textit{while for $h > h_*$, quadratic RMSProp need not minimize 
$L$, spatial variables of its iterates nevertheless enter and remain in an
order $h$ neighborhood of the minimizer for all time}.
 This has practical applications, for example in 
scheduling adjustments to the size of the learning rate $h$ during the minimization 
process.  This is discussed further in Section \ref{sec:adaptive-learning-rate}.

A natural next question is ``what kinds of dynamics actually occur
in the compact trapping region?'' 
We prove that, for generic $A$ and $h > h_*$ close enough to $h_*$, 
there exists a stable branch of period-two orbits 
which bifurcate from the minimizer when $h = h_*$. Again, the lack of a $C^1$ neighborhood of the origin deprives us of a direct application of standard bifurcation theory, and problem-specific arguments are needed.  We also establish the existence of a large learning rate family 
 of up to $2^{d-1}$ symmetric period-two orbits, close to the vertices of a certain cube 
 in the phase space.  The relationship between these large and small learning rate 
 families is studied in examples via numerical continuation.

The results described above hold for a generic symmetric positive definite matrix $A$, and 
any number of variables $d$. We then explore the dynamics in the 
trapping region for a number of explicit matrices $A$ of various sizes ($2 \times 2$ through $5 \times 5$).
Here we find that the system exhibits rich dynamics.
 The results are discussed in detail in Sections \ref{sec:numerical_methods}
and \ref{sec:dynamics_trapping_region}, but a brief summary is as follows:
we often find
Neimark--Sacker bifurcations of the family of $h > h_*$ periodic orbits discussed above, 
leading to attracting invariant circles.  Computation and continuation of these invariant circles
using the parameterization method (see Section \ref{rem:parm}) leads sometimes to further 
bifurcations into invariant tori, and sometimes to period doubling and the appearance of 
strange chaotic attractors. We also perform a limited parameter space search, classifying parameters by the kind of attractor 
they lead to.

Finally, the relevance of this study is not limited to RMSProp. For example, with its time-dependent bias correction removed, Adam reduces to RMSProp when $\beta_1=0$. Motivated by this relation and by the persistence of normally hyperbolic invariant manifolds under small perturbations~\cite{hps1977invariant}, we continue some of our results into Adam, suggesting that it may be amenable to the same analysis. We discuss this continuation further in Section~\ref{sec:Adam}.

The remainder of the paper is organized as follows. Section~\ref{sec:related-work} discusses related work and motivation. Section~\ref{sec:rmsprop-map} defines the RMSProp map, Section~\ref{sec:stability-threshold} contains the stability and trapping results, and Section~\ref{sec:period_two_orbits} treats period-two orbits. Sections~\ref{sec:numerical_methods} and~\ref{sec:bifurcation-analysis} describe the numerical methods and the numerical bifurcation analysis, Section~\ref{sec:survey} surveys the other parameters, and Section~\ref{sec:Adam} treats autonomous Adam.

\section{Related work and motivation}\label{sec:related-work}

\subsection{Machine learning context}\label{rem:intro_AI}

Gradient descent and its variants are
 enjoying a period of substantial renewed interest, due
at least in part to applications in machine learning (ML) like 
supervised and unsupervised training of neural networks and  
large language models.  A thorough review of this literature is a task far beyond
the scope of the present work, and we refer the interested reader to the review 
article of LeCun, Bengio, and Hinton
\cite{lecun2015deep}, and to the notes of Ruder~\cite{ruder2016overview} 
for much more exhaustive discussion. We note that referring to
$L$ as the ``loss function,'' and $h$ as the ``learning rate'' is a
convention that comes from the
machine learning literature.

\subsection{Dynamics of discretized gradient descent}\label{ref:intro_gradDynamics}

Treating the gradient descent algorithm as a dynamical system 
with the learning rate as a bifurcation parameter is a classic topic in numerical 
analysis and dynamics of algorithms.   
That increasing the learning rate to accelerate convergence leads to oscillations
is well known (see for example the 
book of Hairer and Wanner~\cite{hairer1996solving}).  It is also common to study gradient descent for quadratic loss 
functions, in which case the algorithm involves iterating a linear system, 
and all dynamics are understood by studying the eigenvalues of the resulting 
matrix.  In this case it's easy to see that convergence to the minimum, oscillations, 
and unbounded orbits are all possible. 
Beyond quadratic minimization, the gradient descent algorithm admits nonlinearities and the situation is much more complicated. Kong and Tao have proved that, for multiscale objectives, the algorithm exhibits chaotic dynamics when the learning rate is not small enough, providing a mechanism for escaping local minima as an alternative to noisy escape due to stochastic gradients~\cite{kong2020stochasticity}. We refer also to the work of Liang and Mont\'{u}far~\cite{liang2025largesteps}, who prove the existence of chaotic behavior and fractal convergence boundaries when gradient descent is applied to certain matrix factorization problems.

\subsection{Dynamics of RMSProp and the adaptive edge of stability}\label{rem:dynamicsRMSProp}

Violent but bounded oscillations occur in adaptive methods as well as in gradient descent. Their behavior, however, depends on more than the learning rate and the local curvature. In RMSProp, we also divide by a quantity determined by the memory variables. When gradients become large, the memory can grow and reduce subsequent updates. Thus, fixing $h$ and $\epsilon$ does not fix the local curvature or the coordinate rescaling along a training trajectory. Cohen et al.~\cite{cohen2025centralflows, cohen2022adaptiveeos} study this behavior in neural-network training. They account for the optimizer's coordinate rescaling when measuring the largest eigenvalue of the Hessian, obtaining what they call the preconditioned sharpness. In their experiments, this quantity often stays near the stability boundary of the update with the rescaling held fixed. They call this the adaptive edge of stability. Moreover, oscillations at the edge of stability can be chaotic~\cite{cohen2025centralflows}.

A natural next question is ``what kinds of dynamics actually occur when the iterates remain bounded?'' For our quadratic loss, the Hessian is constant, so we focus on the effects of adaptation without also having to account for changes in the loss curvature. We note that the adaptive edge of stability concerns the changing preconditioned curvature along a training trajectory, while $h_*$ in~\eqref{eq:h_star} corresponds to a fixed stability boundary of the quadratic minimizer. Bock and Wei\ss~\cite{bock2021local} give local stability conditions for several versions of Adam and RMSProp, while Dimitrieski et al.~\cite{dimitrieski2026rmsprop} prove global convergence for epsilon-outside RMSProp below $2\epsilon/L$ and boundedness under bounded step-size rules for strongly convex, $L$-smooth losses. Here we study what happens above the quadratic threshold, where bounded iterates need not converge to the minimizer.

\subsection{Quadratic loss functionals}\label{rem:quadraticLoss}

Restricting to a quadratic loss corresponds to isolating the quadratic term in the Taylor expansion of $L$
near a nondegenerate local minimum. 
Moreover, when $L$ is quadratic the origin is the unique minimizer of $L$.
So the variational structure of the loss function is completely understood,  
and we focus entirely on properties of the algorithm, like
its ability to detect the minimum.

Beyond this, Meterez et al.~\cite{meterez2026quadratic} have argued in defense of the quadratic model, 
showing that local quadratic models predict well the training results of a language model
with millions of parameters.  
The case of quadratic loss also illuminates the boundary between gradient descent and RMSProp.  This is 
because gradient descent for a quadratic functional results in a linear dynamical system, 
while RMSProp remains nonlinear due to the fact that 
the step sizes depend on past gradients.  Complicated behavior of RMSProp in a quadratic example
is then a feature of the optimizer, and not the chosen problem. 
Beyond this, the quadratic case admits detailed analytical study, serving as a base for further analysis 
of nonconvex, noisy, or multiscale losses.

\subsection{The parameterization method}\label{rem:parm}

The parameterization method is a functional analytic framework for studying 
invariant manifolds, first proposed by Cabr\'{e}, Fontich, and de la Llave for studying 
spectral submanifolds associated with fixed points of maps between Banach spaces
\cite{cabre2003parameterization1,cabre2003parameterization2,cabre2005parameterization3}, 
and later extended to quasi-periodic orbits of maps, and their attached invariant 
stable/unstable manifolds by Haro and de la Llave 
\cite{haro2006parameterization,haro2007explorations,fontich2009whiskered}.
By now the parameterization method supports a thriving sub-industry, and a thorough review 
of the literature is a task beyond the scope of the present work.  We refer the interested 
reader to the recent book of Haro et al.~\cite{haro2016parameterization}. We describe our computation and continuation of invariant circles and tori in Section~\ref{sec:parametrization_methods_stability}.

\section{RMSProp as a dynamical system}\label{sec:rmsprop-map}

Throughout, $d$ denotes the spatial dimension. We write $\oslash$ and $\odot$ to denote element-wise division and multiplication respectively, and $x^{\odot r}$ and $\sqrt[\odot]{x}$ for the element-wise power and square root operations on the vector $x$. Let $I_d$ denote the $d \times d$ identity matrix and $\mathbf{1} \in \mathbb{R}^d$ the $d$-dimensional column vector of ones. If $A$ is a $d \times d$ matrix with entries $a_{ij}$, then we write $|A|$ to denote the matrix with entries $|a_{ij}|$. For vectors, $|x|=(|x_1|,\ldots,|x_d|)^\top$ denotes the componentwise absolute value.

\subsection{Algorithm definition}\label{sec:algdef}
Given a learning rate $h > 0$ and a memory decay rate $\beta \in [0,1)$, we are interested in
the epsilon-outside RMSProp update of Algorithm~\ref{alg:rmsprop}, as used, for example,
in the PyTorch implementation~\cite{pytorchRMSpropAPI2026}.

\begin{algorithm}[!htbp]
\caption{RMSProp Update Rule (Epsilon-Outside Variant):  Tieleman and Hinton~\cite{tieleman2012lecture}}
\label{alg:rmsprop}
\begin{algorithmic}[1]
\Require Initial spatial vector $x_0 \in \mathbb{R}^d$, Initial memory $v_0\in\mathbb{R}_{\ge0}^d$ (usually $v_0=0$), Numerical stability constant $\epsilon \in \mathbb{R}_{>0}$
\Require Objective function $L: \mathbb{R}^d \to \mathbb{R}$, learning rate $h > 0$, decay rate $\beta \in [0,1)$
\Require Number of iterations $N$
\Statex
\For{$k = 0, 1, \dots, N-1$}
\State $g_k \gets \nabla L(x_k)$ \Comment{Compute gradient}
\State $v_{k+1} \gets \beta v_{k} + (1 - \beta) (g_k)^{\odot 2}$ \Comment{Update gradient memory}
\State $x_{k+1} \gets x_k - h (g_k) \oslash \left( \sqrt[\odot]{v_{k+1}} + \epsilon\mathbf 1 \right)$ \Comment{Update spatial variable}
\EndFor
\Statex
\State \Return $x_N, v_N$
\end{algorithmic}
\end{algorithm}

We note that the definition is not universal, as
different optimizer implementations place the numerical stability constant $\epsilon$ differently. 
PyTorch RMSprop adds $\epsilon$ outside the square root, as in Algorithm~\ref{alg:rmsprop} \cite{pytorchRMSpropAPI2026}.
The current Keras/TensorFlow RMSprop implementation instead adds $\epsilon$ inside the square root, 
giving $\sqrt[\odot]{v_{k+1}+\epsilon\mathbf 1}$ \cite{kerasRMSpropSource2026}. Optax, an optimization 
library for JAX, supports both conventions and defaults to epsilon inside the square root \cite{optaxRMSpropAPI2026}.

Both epsilon conventions preserve the phase space $\mathbb R^d\times\mathbb R_{\geq0}^d$, but differ in their denominators and stability thresholds.
Note that the origin is the unique fixed point of the RMSProp update map $F$, defined in~\eqref{eq:RMSProp} below, in the phase space $\mathbb R^d\times\mathbb R_{\geq0}^d$.
In the epsilon-outside case, the map is differentiable at the origin relative to $v\geq0$, as shown in Lemma~\ref{lem:outside_relative_derivative}, but it is not $C^1$ in a neighborhood of the origin. 

The epsilon-inside map additionally extends smoothly to an open neighborhood of the origin that includes negative memory coordinates. We focus on epsilon outside because the loss of a $C^1$ neighborhood prevents a direct application of the usual linearization theorem and is thus the more challenging case. The epsilon-inside case is smoother and its origin stability/instability follows from linearization.

As already mentioned in the introduction, we focus on the performance and properties
of RMSProp with $L$ a quadratic loss function
\begin{equation} \label{eq:quadratic}
L(x)=\frac{1}{2} x^\top Ax,
\end{equation}
where $A=A^\top\succ0$. We assume $h,\epsilon>0$, $0<\beta<1$, and $v_0\geq0$.

\subsection{The map}
Combining the memory and spatial updates defines a map on $\mathbb R^d\times\mathbb R^d_{\geq0}$. 
For the quadratic objective of Equation \eqref{eq:quadratic}, set
\begin{equation} \label{eq:memory}
F_V(x,v)=\beta v+(1-\beta)(Ax)^{\odot2}.
\end{equation}
The epsilon-outside RMSProp map $F:\mathbb R^d\times\mathbb R_{\geq0}^d\to\mathbb R^d\times\mathbb R_{\geq0}^d$ is
\begin{equation} \label{eq:RMSProp}
F(x,v) = \begin{bmatrix}
x-hAx\oslash\left(\sqrt[\odot]{\beta v+(1-\beta)(Ax)^{\odot2}}+\epsilon\mathbf 1\right) \\
\beta v+(1-\beta)(Ax)^{\odot2}
\end{bmatrix}
= \begin{bmatrix}
x-hAx\oslash\left(\sqrt[\odot]{F_V(x,v)}+\epsilon\mathbf 1\right) \\
F_V(x,v)
\end{bmatrix} 
\end{equation}
For fixed parameters, write $F(x,v)=(F_1(x,v),F_2(x,v))$, where $F_1$ is the spatial component and $F_2=F_V$ is the memory component. We omit the fixed parameters from the notation. 
Note for reference that the epsilon-inside map $F_I:\mathbb R^d\times\mathbb R_{\geq0}^d\to\mathbb R^d\times\mathbb R_{\geq0}^d$ 
is written
\begin{equation*}
F_I(x,v) = \begin{bmatrix}
x-hAx\oslash\sqrt[\odot]{F_V(x,v)+\epsilon\mathbf 1} \\
\beta v+(1-\beta)(Ax)^{\odot2}
\end{bmatrix}
\end{equation*}

\section{Origin stability and global trapping}\label{sec:stability-threshold}

Dimitrieski et al.~\cite{dimitrieski2026rmsprop} establish global asymptotic stability for epsilon-outside RMSProp, as defined in~\eqref{eq:RMSProp}, below $2\epsilon/L$ for strongly convex, $L$-smooth losses, together with boundedness under bounded step-size rules.

\begin{definition}[Lyapunov stability and instability \cite{elaydi2005difference}]\label{def:lyapunov_stability}
Let $F:\mathbb R^d\times\mathbb R_{\geq0}^d\to\mathbb R^d\times\mathbb R_{\geq0}^d$ be a discrete dynamical system and let $z_*$ be a fixed point. The point $z_*$ is \emph{stable in the sense of Lyapunov} if, for every $\rho>0$, there exists $\delta>0$ such that
\begin{equation*}\|z_0-z_*\|<\delta\quad\Longrightarrow\quad\|F^k(z_0)-z_*\|<\rho\quad\text{for every }k\geq0\end{equation*}
for all $z_0\in\mathbb R^d\times\mathbb R_{\geq0}^d$. The fixed point is \emph{Lyapunov unstable} if it is not Lyapunov stable; equivalently, there exists $\rho_0>0$ such that, for every $\delta>0$, there exist $z_0\in\mathbb R^d\times\mathbb R_{\geq0}^d$ and $k\geq0$ satisfying
\begin{equation*}\|z_0-z_*\|<\delta,\qquad\|F^k(z_0)-z_*\|\geq\rho_0.\end{equation*}
\end{definition}

\subsection{Global asymptotic stability below the threshold}

\begin{definition}[$\mu$-convexity]\label{def:mu_convexity}
A differentiable function $f:\mathbb R^d\to\mathbb R$ is \emph{$\mu$-strongly convex}, or \emph{$\mu$-convex}, if $\mu>0$ and
\begin{equation*}f(y)\geq f(x)+\nabla f(x)^\top(y-x)+\frac{\mu}{2}\|y-x\|_2^2\qquad\text{for all }x,y\in\mathbb R^d.\end{equation*}
\end{definition}

\begin{definition}[$L$-smoothness]\label{def:L_smoothness}
A differentiable function $f:\mathbb R^d\to\mathbb R$ is \emph{$L$-smooth} if $L>0$ and its gradient is globally $L$-Lipschitz:
\begin{equation*}\|\nabla f(y)-\nabla f(x)\|_2\leq L\|y-x\|_2\qquad\text{for all }x,y\in\mathbb R^d.\end{equation*}
\end{definition}

We recall, for reference, the constant-step stability result, whose statement and proof are found in \cite[Thm.~1]{dimitrieski2026rmsprop}.

\begin{theorem}[Global stability of epsilon-outside RMSProp]\label{thm:dimitrieski_global_stability}
Let $f:\mathbb R^d\to\mathbb R$ be continuously differentiable, $\mu$-strongly convex for some $\mu>0$, and globally $L$-smooth for some $L>0$, and let $x_*$ be its unique minimizer. Consider
\begin{equation*}v_{k+1}=\beta v_k+(1-\beta)(\nabla f(x_k))^{\odot2},\qquad x_{k+1}=x_k-h\nabla f(x_k)\oslash\left(\sqrt[\odot]{v_{k+1}}+\epsilon\mathbf 1\right),\end{equation*}
with $x_0\in\mathbb R^d$, $v_0\in\mathbb R_{\geq0}^d$, $0<\beta<1$, $\epsilon>0$, and constant step size
\begin{equation*}0<h<\frac{2\epsilon}{L}.\end{equation*}
Then $(x_*,0)$ is globally asymptotically stable in $\mathbb R^d\times\mathbb R_{\geq0}^d$.
\end{theorem}

For the quadratic objective $L(x)=\frac12x^\top Ax$, the unique minimizer is $x_*=0$, while the strong-convexity and smoothness constants are $\mu=\lambda_{\min}(A)$ and $L=L_A$, where $L_A=\lambda_{\max}(A)=\|A\|_2$. The cited theorem therefore gives the following specialization.

\begin{corollary}[Global asymptotic stability]\label{thm:outside_stability_below_threshold}
If $0<h<h_*$, then under $L(x)=\frac12x^\top Ax$, the origin is globally asymptotically stable in $\mathbb R^d\times\mathbb R_{\geq0}^d$ under $F$.
\end{corollary}

In the instability proof below, we use the component maps $F_1(x,v)$ and $F_2(x,v)$ in the original coordinates.

\subsection{Instability above the threshold}

To prove Lyapunov instability for $h>h_*$, we write the map as a leading linear term plus a higher-order remainder. We construct a cone about the most strongly expanding eigenspace and show that an orbit remains in this cone while it stays near the origin. The expanding component then forces orbits starting arbitrarily close to the origin to leave a fixed neighborhood.

We begin by observing that the RMSProp mapping is Fr\'{e}chet differentiable at the origin, relative to the phase space $v\geq0$, as defined below.

\begin{definition}[Relative derivative]\label{def:relative_derivative}
Let $z_*\in \mathbb R^d\times\mathbb R_{\geq0}^d$. A map $F:\mathbb R^d\times\mathbb R_{\geq0}^d\to\mathbb R^{2d}$ is \emph{differentiable at $z_*$ relative to the cone $\mathbb R^d\times\mathbb R_{\geq0}^d$} if there is a linear map $J:\mathbb R^{2d}\to\mathbb R^{2d}$ such that
\begin{equation*}
\lim_{\substack{z\to z_*,\ z\ne z_*\\z\in \mathbb R^d\times\mathbb R_{\geq0}^d}}
\frac{\|F(z)-F(z_*)-J(z-z_*)\|_2}{\|z-z_*\|_2}=0.
\end{equation*}
We call $J$ the \emph{relative derivative} and write $D_{\mathbb R^d\times\mathbb R_{\geq0}^d}F(z_*)=J$.
\end{definition}

\begin{lemma}[Relative derivative at the origin]\label{lem:outside_relative_derivative}
For every $h>0$, the epsilon-outside map $F$ is differentiable at the origin relative to $\mathbb R^d\times\mathbb R_{\geq0}^d$, with
\begin{equation}\label{eq:origin-relative-derivative}
D_{\mathbb R^d\times\mathbb R_{\geq0}^d}F(0)=J_0=\begin{bmatrix}I_d-\frac{h}{\epsilon}A&0\\0&\beta I_d\end{bmatrix}.
\end{equation}
More precisely, $F(z)=J_0z+R(z)$, where $\|R_x(z)\|_2=O(\|z\|_2^{3/2})$ and $\|R_v(z)\|_2=O(\|z\|_2^2)$ as $z\to0$ in $\mathbb R^d\times\mathbb R_{\geq0}^d$.
\end{lemma}

\begin{proof}
Write $z=(x,v)\in\mathbb R^d\times\mathbb R_{\geq0}^d$. Use the memory component $F_2=F_V$ from~\eqref{eq:memory}.
Adding and subtracting $(h/\epsilon)Ax$ in the spatial update gives the decomposition $F(z)=J_0z+R(z)$, with $J_0$ from~\eqref{eq:origin-relative-derivative} and
\begin{equation*}
R(z)=\begin{bmatrix}
hAx\odot\left(\frac1\epsilon\mathbf1-\frac1{\sqrt[\odot]{F_2(x,v)}+\epsilon\mathbf1}\right)\\[2pt]
(1-\beta)(Ax)^{\odot2}
\end{bmatrix}.
\end{equation*}
Here and below the fractions in the first component are componentwise.  Thus the two blocks of the remainder are $R_x(z)=hAx\odot\sqrt[\odot]{F_2(x,v)}\oslash\bigl(\epsilon(\sqrt[\odot]{F_2(x,v)}+\epsilon\mathbf1)\bigr)$ and $R_v(z)=(1-\beta)(Ax)^{\odot2}$.

To make the remainder estimate explicit, use the Euclidean norm and $L_A=\|A\|_2$.  Since $\sqrt{[F_2(x,v)]_i}+\epsilon\geq\epsilon$, $|(Ax)_i|\leq L_A\|x\|_2$, $v_i\leq\|v\|_2$, and $(Ax)_i^2\leq L_A^2\|x\|_2^2$, every spatial component satisfies $|(R_x(z))_i|\leq(h/\epsilon^2)|(Ax)_i|\sqrt{[F_2(x,v)]_i}$. Consequently,
\begin{equation}
|(R_x(z))_i|\leq\frac{hL_A}{\epsilon^2}\|x\|_2\sqrt{\beta\|v\|_2+(1-\beta)L_A^2\|x\|_2^2},\qquad
\|R_v(z)\|_2\leq(1-\beta)\|Ax\|_2^2\leq(1-\beta)L_A^2\|x\|_2^2.
\label{eq:outside_remainder_bounds}
\end{equation}
If $\varrho=\|z\|_2\leq1$, then $\|x\|_2,\|v\|_2\leq \varrho$ and $\varrho^2\leq \varrho$, so \eqref{eq:outside_remainder_bounds} gives
\[
\|R_x(z)\|_2\leq
\frac{h\sqrt dL_A}{\epsilon^2}\sqrt{\beta+(1-\beta)L_A^2}\,\varrho^{3/2},
\qquad
\|R_v(z)\|_2\leq(1-\beta)L_A^2\varrho^2.
\]
Consequently,
\[
\frac{\|R(z)\|_2}{\|z\|_2}
\leq\frac{h\sqrt dL_A}{\epsilon^2}\sqrt{\beta+(1-\beta)L_A^2}\,\varrho^{1/2}
+(1-\beta)L_A^2\varrho\longrightarrow0,
\]
as $z\to0$ in $\mathbb R^d\times\mathbb R_{\geq0}^d$, and hence $R(z)=o(\|z\|_2)$.

\end{proof}

Note that the classic Hartman--Grobman/stable manifold theory cannot be applied off the shelf, as the map is not differentiable in a neighborhood of the origin.  This said, we find that a Lipschitz extension of the RMSProp map permits application of the Lipschitz regularity unstable manifold construction of \cite{MR271991} (see also \cite{hps1977invariant}), whose proof is based on the Lipschitz Hadamard--Perron theorem to obtain the instability of the origin, and even the existence of a low regularity unstable manifold.

We do not present this argument here, giving instead a direct instability proof based on the construction of a simple invariant cone. The argument is fairly short, makes the present paper more self-contained, and provides some insight into the geometry of the instability which will be useful later.  The construction and preparation of the Lipschitz extension, on the other hand, are somewhat technical, and more opaque, and not genuinely shorter than the explicit construction given below.

\begin{theorem}[Local Lyapunov instability at the origin]
    \label{thm:outside_instability_above_threshold}
    If
    \[
        h>h_*:=\frac{2\epsilon}{L_A},
    \]
    then the origin is Lyapunov unstable under \(F\).
\end{theorem}

\begin{proof} By Lemma~\ref{lem:outside_relative_derivative}, \(F(z)=J_0z+R(z), \quad R(z)=o(\|z\|_2),\) as \(z\to0, \quad z\in \mathbb R^d\times\mathbb R_{\geq0}^d.\) Let \(\mathcal U:=E_{L_A}\times\{0\},\) where \(E_{L_A}\) is the eigenspace of \(A\) associated with its largest eigenvalue \(L_A\). On \(\mathcal U\), \(J_0|_{\mathcal U}=-aI, \quad a:=\frac{hL_A}{\epsilon}-1>1.\)

Since \(A\) is symmetric and positive semidefinite, \(\mathcal U^\perp\) is invariant under \(J_0\). Moreover, \(b:= \left\| J_0|_{\mathcal U^\perp} \right\|_2 <a.\) Indeed, the eigenvalues of the \(x\)-block of \(J_0\) are \(1-\frac{h}{\epsilon}\lambda_j, \quad 0\leq\lambda_j\leq L_A,\) while the eigenvalues in the \(v\)-directions are equal to \(\beta\). Since \(\frac{hL_A}{\epsilon}>2,\) all eigenvalues corresponding to \(\lambda_j<L_A\), as well as \(\beta\), have modulus strictly smaller than \(a\).

Choose \(\eta>0\) sufficiently small that \begin{equation} a-2\eta > \max\{1,b+2\eta\}. \label{eq:outside-cone-gap} \end{equation} Since \(R(z)=o(\|z\|_2)\), there exists \(\varrho_0>0\) such that \(\|R(z)\|_2 \leq \eta\|z\|_2\) whenever \(z\in B_{\varrho_0}(0) \cap \left( \mathbb R^d\times\mathbb R_{\geq0}^d \right).\)

Let \(\Pi_{\mathcal U}\) and \(\Pi_{\mathcal U^\perp}\) denote the orthogonal projections onto \(\mathcal U\) and \(\mathcal U^\perp\), respectively, and define the cone \[ \mathcal C := \left\{ z=u+w: u\in\mathcal U,\; w\in\mathcal U^\perp,\; \|w\|_2\leq\|u\|_2 \right\}. \] If \(z=u+w \in \mathcal C \cap B_{\varrho_0}(0) \cap \left( \mathbb R^d\times\mathbb R_{\geq0}^d \right),\) then \(\|z\|_2 \leq \|u\|_2+\|w\|_2 \leq 2\|u\|_2.\) Consequently, \begin{align*} \|\Pi_{\mathcal U}F(z)\|_2 &\geq a\|u\|_2-\|R(z)\|_2 \geq (a-2\eta)\|u\|_2, \end{align*} while \begin{align*} \|\Pi_{\mathcal U^\perp}F(z)\|_2 &\leq b\|w\|_2+\|R(z)\|_2 \leq (b+2\eta)\|u\|_2. \end{align*} By~\eqref{eq:outside-cone-gap}, \(\|\Pi_{\mathcal U^\perp}F(z)\|_2 < \|\Pi_{\mathcal U}F(z)\|_2,\) and \(\|\Pi_{\mathcal U}F(z)\|_2 \geq (a-2\eta)\|u\|_2 > \|u\|_2.\) Hence the cone is forward invariant as long as the orbit remains in \(B_{\varrho_0}(0)\), and the \(\mathcal U\)-component expands by the uniform factor \(q:=a-2\eta>1.\)

Let \(e_{\max}\in E_{L_A}\) be a unit eigenvector. For any \(\delta>0\), choose \(0<t<\min\{\delta,\varrho_0\}\) and initialize \(z_0=(te_{\max},0).\) Then \(z_0\in \mathcal C \cap \left( \mathbb R^d\times\mathbb R_{\geq0}^d \right).\) If the forward orbit remained in \(B_{\varrho_0}(0)\) for all \(k\geq0\), cone invariance would give \(\|\Pi_{\mathcal U}F^k(z_0)\|_2 \geq q^k t,\) which is impossible since \(q>1\). Therefore the orbit of an arbitrarily small initial condition leaves \(B_{\varrho_0}(0)\). Thus the origin is Lyapunov unstable. \end{proof}

\subsection{Global trapping}
We now show that every RMSProp orbit eventually enters a compact forward-invariant set. The argument uses a Lyapunov or energy-like function that decreases outside a bounded region. Indeed, the argument is similar to that found in the proof of Theorem \ref{thm:dimitrieski_global_stability}, but adapted to the quadratic loss function and taking into account the loss of stability of the origin. We first state a general trapping criterion, then establish the required decrease for normalized RMSProp and obtain bounds in the original coordinates.

\begin{theorem}
\label{thm:exterior-energy-trapping}
Let $F:\mathbb R^d\times\mathbb R_{\ge0}^d\to\mathbb R^d\times\mathbb R_{\ge0}^d$ and $V:\mathbb R^d\times\mathbb R_{\ge0}^d\to\mathbb R$ be continuous. Suppose $V$ is radially unbounded, namely
\begin{equation}
V(z)\longrightarrow+\infty\quad\text{as }\|z\|\longrightarrow\infty,
\qquad z\in\mathbb R^d\times\mathbb R_{\ge0}^d.
\label{eq:trapping-radial-unboundedness}
\end{equation}
Let $R\subset\mathbb R^d\times\mathbb R_{\ge0}^d$ be bounded, and assume
\begin{equation}
V(F(z))<V(z)\qquad\text{for every }z\in(\mathbb R^d\times\mathbb R_{\ge0}^d)\setminus R.
\label{eq:trapping-exterior-decrease}
\end{equation}
Then there exists $C_0>0$ such that, for every $C\ge C_0$, the sublevel set $\Omega_C=\{z\in\mathbb R^d\times\mathbb R_{\ge0}^d:V(z)\le C\}$ has the following properties:
\begin{enumerate}
\renewcommand{\labelenumi}{(\roman{enumi})}
\setlength{\itemsep}{2pt}
\item \textbf{Compactness.} $\Omega_C$ is compact.
\item \textbf{Containment of the region.} $R\subset\operatorname{int}_{\mathbb R^d\times\mathbb R_{\ge0}^d}(\Omega_C)$.
\item \textbf{Strict trapping.} $F(\Omega_C)\subset\operatorname{int}_{\mathbb R^d\times\mathbb R_{\ge0}^d}(\Omega_C)$. In particular, an orbit cannot leave $\Omega_C$ after entering.
\item \textbf{Finite-time entry.} For every $z_0\in\mathbb R^d\times\mathbb R_{\ge0}^d$, there exists an integer $N=N(z_0,C)\ge0$ such that $F^n(z_0)\in\Omega_C$ for all $n\ge N$.
\end{enumerate}
The interiors in parts (ii) and (iii) are relative to $\mathbb R^d\times\mathbb R_{\ge0}^d$, so neighborhoods are taken within this state space.
\end{theorem}

The theorem is closely related to standard results on Lyapunov functions for discrete time dynamical systems.  The paper of Kalman and Bertram \cite{kalman1960control} as well as the work of LaSalle \cite{lasalle1966invariance} are excellent resources.  A proof of the theorem allowing application on the cone $\mathbb R^d\times\mathbb R_{\ge0}^d$ requires small adjustments, and is given in Appendix~\ref{app:exterior-energy-trapping-proof}.

The rest of this section is devoted to showing that the epsilon-outside RMSProp map admits an energy satisfying conditions~\eqref{eq:trapping-radial-unboundedness} and~\eqref{eq:trapping-exterior-decrease}, which imply that every orbit eventually enters and remains in a compact trapping set.

\subsubsection{Normalization and energy estimates}
\label{sec:trapping-normalization-energy}
First, we introduce the notation and normalization used below. Let $A=A^\top\succ0$ be $d\times d$, $0<\beta<1$, and $\epsilon>0$. Write $m=\lambda_{\min}(A)$, $L_A=\lambda_{\max}(A)$, and $h_*=2\epsilon/L_A$. For $h>0$, set 
\begin{equation}\label{eq:rmsprop-normalization}
X=x/h,\qquad W=\sqrt[\odot]{v}/h,\qquad \text{and}\quad c=\epsilon/h.
\end{equation}
 The normalized map $G=(G_1,G_2)$ has memory and spatial components
\begin{equation}
\begin{aligned}
G_2(X,W)&=\sqrt[\odot]{\beta W^{\odot2}+(1-\beta)(AX)^{\odot2}},\\
G_1(X,W)&=X-AX\oslash\bigl(G_2(X,W)+c\mathbf1\bigr).
\end{aligned}
\label{eq:normalized-rmsprop-drift-lemma}
\end{equation}
For each fixed $h>0$, this map is continuous on $\mathbb R^d\times\mathbb R_{\ge0}^d$ because $c>0$.

For the following energy estimates, assume $h\ge h_*$. Throughout the remainder of this section, we make extensive use of the energy function
\begin{equation}
\mathcal V(X,W)=\frac12X^\top AX+\alpha\|W\|_1,
\label{eq:rmsprop_lyapunov_energy}
\end{equation}
where $\alpha>0$ is chosen in Lemma~\ref{lem:rmsprop-normalized-increment}. This energy is continuous and satisfies condition~\eqref{eq:trapping-radial-unboundedness}, since
\[
\mathcal V(X,W)\ge\frac m2\|X\|_2^2+\alpha\|W\|_1\longrightarrow+\infty
\quad\text{as }\|(X,W)\|\longrightarrow\infty,
\]
with $(X,W)\in\mathbb R^d\times\mathbb R_{\ge0}^d$.

To establish condition~\eqref{eq:trapping-exterior-decrease}, we first bound the changes in the spatial and memory terms separately. We then estimate their combined contribution coordinate by coordinate, distinguishing whether the memory or the gradient is large. This shows that the energy decreases outside a bounded set. The trapping criterion then applies, and an explicit sublevel bound gives the corresponding size estimates in the original coordinates.

\begin{remark}[Choice of energy for the scaling bounds]\label{rem:energy-scaling}
We choose the energy~\eqref{eq:rmsprop_lyapunov_energy} with the bounds on the size of the trapping region in mind. After the normalization~\eqref{eq:rmsprop-normalization}, the learning rate appears in the map only through $c=\epsilon/h$, which is bounded above for $h\ge h_*$. With this energy, we choose the constants in the decrease estimate~\eqref{eq:master_drift} and the trapping sublevel threshold independently of $h$. Returning to the original variables then gives bounds proportional to $h$ for $x$ and to $h^2$ for $v$. Dimitrieski et al.~\cite{dimitrieski2026rmsprop} use a Lyapunov function to prove global convergence and boundedness under bounded time-varying step sizes. Their function includes a $3/2$ power of the loss and a logarithmic memory term, and their estimate does not directly give these two scaling bounds. Our energy is therefore better suited to the particular size estimate we seek here.
\end{remark}

\begin{lemma}[Normalized Lyapunov increments]
\label{lem:rmsprop-normalized-increment}
Under the preceding hypotheses and notation, set $u=X-G_1(X,W)$ and $c_*=L_A/2$, so $c\le c_*$. Define
\[
K=\max\left\{1,\frac{4\sqrt{1-\beta}}{1-\sqrt\beta}\right\},\qquad
M=K\sqrt\beta+\sqrt{1-\beta},\qquad \alpha=\frac1{8M\sqrt{1-\beta}},
\]
and use $\mathcal V$ from~\eqref{eq:rmsprop_lyapunov_energy}. Then
\begin{equation*}
\mathcal V(G(X,W))-\mathcal V(X,W)
\le\sum_{j=1}^d B_j+\frac{L_A d}{2(1-\beta)},
\end{equation*}
where
\begin{equation}
B_j=-\frac{|(AX)_j|^2}{\sqrt\beta W_j+\sqrt{1-\beta}|(AX)_j|+c_*}
+\alpha\sqrt{1-\beta}|(AX)_j|-\alpha(1-\sqrt\beta)W_j.
\label{eq:coordinate-drift-definition}
\end{equation}
Moreover, $\|u\|_2^2\le d/(1-\beta)$.
\end{lemma}

\begin{proof}
Fix $(X,W)\in\mathbb R^d\times\mathbb R_{\geq0}^d$. Then
\[
\mathcal V(G(X,W))-\mathcal V(X,W)
=\tfrac12\bigl(G_1(X,W)^\top AG_1(X,W)-X^\top AX\bigr)
+\alpha\bigl(\|G_2(X,W)\|_1-\|W\|_1\bigr).
\]
We bound the first term and the second term individually, then add their upper bounds. To obtain the estimates needed for both terms, the memory update gives
\[
[G_2(X,W)]_j^2=\beta W_j^2+(1-\beta)(AX)_j^2
\ge(1-\beta)(AX)_j^2.
\]
Taking nonnegative square roots yields the lower bound below. For the upper bound, apply $\sqrt{a+b}\le\sqrt a+\sqrt b$ for $a,b\ge0$, using $W_j\ge0$:
\begin{equation}
\sqrt{1-\beta}|(AX)_j|\le [G_2(X,W)]_j
\le\sqrt\beta W_j+\sqrt{1-\beta}|(AX)_j|.
\label{eq:memory-upper-bound-drift-lemma}
\end{equation}

\emph{Spatial term.} Substituting $G_1(X,W)=X-u$ and using $A=A^\top$ to combine the cross terms gives
\[
\tfrac12(G_1(X,W)^\top AG_1(X,W)-X^\top AX)
=-u^\top AX+\tfrac12u^\top Au
=-\sum_j\frac{(AX)_j^2}{[G_2(X,W)]_j+c}+\tfrac12u^\top Au.
\]
The lower memory bound implies
\[
|u_j|=\frac{|(AX)_j|}{[G_2(X,W)]_j+c}
\le\frac{|(AX)_j|}{\sqrt{1-\beta}|(AX)_j|+c}
\le\frac1{\sqrt{1-\beta}}.
\]
This also holds when $(AX)_j=0$, since $c>0$. Summing squares gives $\|u\|_2^2\le d/(1-\beta)$. Since $u^\top Au\le L_A\|u\|_2^2$, we also obtain $u^\top Au/2\le L_A d/[2(1-\beta)]$. On the other hand, the upper memory bound and $c\le c_*$ give $[G_2(X,W)]_j+c\le\sqrt\beta W_j+\sqrt{1-\beta}|(AX)_j|+c_*$. Since $t\mapsto-g^2/t$ is increasing for $t>0$, we conclude that
\begin{equation}
\tfrac12G_1(X,W)^\top AG_1(X,W)-\tfrac12X^\top AX
\le-\sum_j\frac{|(AX)_j|^2}{\sqrt\beta W_j+\sqrt{1-\beta}|(AX)_j|+c_*}
+\frac{L_A d}{2(1-\beta)}.
\label{eq:quadratic-increment-bound}
\end{equation}

\emph{Memory term.} Nonnegativity of $W$ and $G_2(X,W)$ lets us write their $1$-norms as sums. Subtracting $W_j$ from the upper bound in~\eqref{eq:memory-upper-bound-drift-lemma}, multiplying by $\alpha$, and summing yields
\begin{equation}
\alpha(\|G_2(X,W)\|_1-\|W\|_1)
=\alpha\sum_j([G_2(X,W)]_j-W_j)
\le\alpha\sum_j\bigl[\sqrt{1-\beta}|(AX)_j|-(1-\sqrt\beta)W_j\bigr].
\label{eq:memory-increment-bound}
\end{equation}

\emph{Sum of the two terms.} Adding~\eqref{eq:quadratic-increment-bound} and~\eqref{eq:memory-increment-bound} gives
\[
\mathcal V(G(X,W))-\mathcal V(X,W)
\le\sum_{j=1}^d B_j+\frac{L_A d}{2(1-\beta)},
\]
where $B_j$ is defined in~\eqref{eq:coordinate-drift-definition}.
\end{proof}

\begin{lemma}[Uniform bound on the energy change]
\label{lem:rmsprop-uniform-drift}
Under the preceding hypotheses and notation, define
\[
\gamma=\min\left\{\frac{\alpha(1-\sqrt\beta)}8,\frac1{8M(1+K)}\right\},\qquad
H_0=\frac{c_*}{M}\bigl(\alpha\sqrt{1-\beta}+\gamma(1+K)\bigr),\qquad
H=\frac{L_A d}{2(1-\beta)}+dH_0.
\]
These constants are positive and depend only on $(A,\beta,d)$. For every state $(X,W)$ and coordinate $j$,
\begin{align}
B_j&\le-\gamma(|(AX)_j|+W_j)+H_0,\label{eq:uniform-coordinate-drift}\\
\mathcal V(G(X,W))-\mathcal V(X,W)
&\le-\gamma(\|AX\|_1+\|W\|_1)+H.\label{eq:master_drift}
\end{align}
\end{lemma}

\begin{proof}
Fix $(X,W)$ and $j$, write $g=|(AX)_j|$ and $w=W_j$, and estimate $B=B_j$ in three cases.

\emph{Case 1: $w\ge Kg$.} Dropping the nonpositive fractional term and using $K\ge4\sqrt{1-\beta}/(1-\sqrt\beta)$ gives
\[
B\le\alpha\sqrt{1-\beta}g-\alpha(1-\sqrt\beta)w
\le-\tfrac34\alpha(1-\sqrt\beta)w\le-\gamma(g+w),
\]
since $K\ge1$ implies $g+w\le2w$ and $\gamma\le\alpha(1-\sqrt\beta)/8$.

\emph{Case 2: $w<Kg$ and $g\ge c_*/M$.} The denominator in $B$ is at most $Mg+c_*\le2Mg$. Dropping the nonpositive memory term and using $\alpha\sqrt{1-\beta}=1/(8M)$ gives
\[
B\le-\frac g{2M}+\alpha\sqrt{1-\beta}g=-\frac{3g}{8M}\le-\gamma(g+w),
\]
because $g+w<(1+K)g$ and $\gamma\le1/[8M(1+K)]$.

\emph{Case 3: $w<Kg$ and $g<c_*/M$.} Dropping both nonpositive terms gives
\[
B+\gamma(g+w)\le\alpha\sqrt{1-\beta}g+\gamma(g+w)
\le\frac{c_*}{M}\bigl(\alpha\sqrt{1-\beta}+\gamma(1+K)\bigr)=H_0.
\]
Since $H_0>0$, all three cases imply~\eqref{eq:uniform-coordinate-drift}. Summing over $j$ and applying Lemma~\ref{lem:rmsprop-normalized-increment} gives~\eqref{eq:master_drift}.
\end{proof}

Define the bounded set
\begin{equation}
R_*=\frac{2H}{\gamma},\qquad
\mathcal B=\{(X,W)\in \mathbb R^d\times\mathbb R_{\ge0}^d:\|AX\|_1+\|W\|_1\le R_*\}.
\label{eq:trapping-core}
\end{equation}
On $\mathcal B$ we have $m\|X\|_2\le\|AX\|_2\le R_*$ and $\|W\|_1\le R_*$. Outside $\mathcal B$, \eqref{eq:master_drift} gives
\[
\mathcal V(G(X,W))-\mathcal V(X,W)<-\gamma R_*+H=-H<0.
\]
Thus the hypotheses of Theorem~\ref{thm:exterior-energy-trapping} hold. We now obtain a trapping set of the form $\{\mathcal V\le C\}$, with $C$ independent of $h$, and translate its bounds into the original coordinates.

\begin{theorem}[RMSProp trapping region and scaling bounds]
\label{thm:rmsprop-trapping-bound}
Under the preceding hypotheses and notation, define
\begin{equation*}
B_0=\frac{L_A R_*^2}{2m^2}+\alpha R_*,\qquad
C=B_0+2H,\qquad C_x=\sqrt{2C/m},\qquad C_v=(C/\alpha)^2.
\end{equation*}
These positive constants depend only on $(A,\beta,d)$. For every $h\ge h_*$, the set $\mathcal S=\{(X,W)\in \mathbb R^d\times\mathbb R_{\ge0}^d:\mathcal V(X,W)\le C\}$ is compact and satisfies
\[
\mathcal B\subset\operatorname{int}_{\mathbb R^d\times\mathbb R_{\ge0}^d}(\mathcal S),\qquad
G(\mathcal S)\subset\operatorname{int}_{\mathbb R^d\times\mathbb R_{\ge0}^d}(\mathcal S).
\]
Every normalized orbit enters $\mathcal S$ in finite time. In the original coordinates, every epsilon-outside RMSProp orbit with $v_0\ge0$ enters and permanently remains in the compact trapping region
\[
\begin{aligned}\mathcal K_h&=\left\{(x,v)\in \mathbb R^d\times\mathbb R_{\ge0}^d:\mathcal V\left(x/h,\sqrt[\odot]{v}/h\right)\le C\right\}
\\
&\subset[-C_xh,C_xh]^d\times[0,C_vh^2]^d,\end{aligned}
\]
with $F(\mathcal K_h)\subset\operatorname{int}_{\mathbb R^d\times\mathbb R_{\ge0}^d}(\mathcal K_h)$. Consequently, for $h\ge h_*$,
\[
\limsup_{n\to\infty}\|x_n\|_\infty \leq C_x h
\quad \quad \mbox{and} \quad \quad
\limsup_{n\to\infty}\|v_n\|_\infty\leq C_v h^2.
\]
\end{theorem}

\begin{proof}
On $\mathcal B$, the bounds following~\eqref{eq:trapping-core} imply $\mathcal V\le B_0$. Since \eqref{eq:master_drift} also gives $\mathcal V\circ G\le\mathcal V+H$ everywhere on $\mathbb R^d\times\mathbb R_{\ge0}^d$, we have
\[
\mathcal V(z)\le B_0<C,\qquad
\mathcal V(G(z))\le B_0+H<C\qquad(z\in\mathcal B).
\]
These bounds show that the level $C$ satisfies the threshold construction in the proof of Theorem~\ref{thm:exterior-energy-trapping}, given in Appendix~\ref{app:exterior-energy-trapping-proof}.

On $\mathcal S$, $(m/2)\|X\|_\infty^2\le\mathcal V\le C$ and $\alpha\|W\|_\infty\le\mathcal V\le C$. Hence $\|X\|_\infty\le C_x$ and $\|W\|_\infty^2\le C_v$. The homeomorphism $(X,W)\mapsto(hX,h^2W^{\odot2})$ conjugates $G$ to $F$ and maps $\mathcal S$ onto $\mathcal K_h$. It preserves relative interiors, compactness, and finite-time entry, so $\mathcal K_h$ has the asserted trapping properties and the bounds $\|x\|_\infty\le C_xh$ and $0\le v_j\le C_vh^2$. 
\end{proof}

\subsection{Normalized coordinates}\label{sec:normalized_coordinates}
Theorem~\ref{thm:rmsprop-trapping-bound} gives an eventual bound of order $O(h)$ in $x$ and $O(h^2)$ in $v$ when $h>h_*$. Moreover, the attractor radii experiments in Figure~\ref{fig:high-h-scaling} exhibit this same scaling. And for period-two orbits near $h_*$, Theorem~\ref{thm:period_two_collision} gives the sharper bounds $x=O(h-h_*)$ and $v=O((h-h_*)^2)$. These different scales in $x$ and $v$ coordinates make numerical analysis considerably more difficult, motivating the normalization outlined in this section.

For the numerical experiments, we use the normalized coordinates $X=x/h$ and $W=\sqrt[\odot]{v}/h$. In these coordinates, the trapping estimate bounds $X$ and $W$ independently of $h$ for $h\ge h_*$. The normalized epsilon-outside map $G:\mathbb R^d\times\mathbb R_{\geq0}^d\to\mathbb R^d\times\mathbb R_{\geq0}^d$ is given by~\eqref{eq:normalized-rmsprop-drift-lemma}, with $G(X_k,W_k)=(X_{k+1},W_{k+1})$.
Applying the same transformation to the epsilon-inside map $F_I$ yields the normalized map $G_I:\mathbb R^d\times\mathbb R_{\geq0}^d\to\mathbb R^d\times\mathbb R_{\geq0}^d$, with $G_I(X_k,W_k)=(X_{k+1},W_{k+1})$ and
\begin{equation*}
W_{k+1}=G_2(X_k,W_k),\qquad
X_{k+1}=X_k-AX_k\oslash\sqrt[\odot]{W_{k+1}^{\odot2}+\frac{\epsilon}{h^2}\mathbf1}.
\end{equation*}

The normalization is a change of coordinates for both epsilon conventions, where the phase space is $\mathbb R^d\times\mathbb R_{\geq0}^d$ in both the original and normalized coordinates.

\begin{proposition}[Conjugacy under normalization]
Fix $h>0$. The map
\[
H_h^W:\mathbb R^d\times\mathbb R_{\geq0}^d\to\mathbb R^d\times\mathbb R_{\geq0}^d,
\qquad
H_h^W(x,v)
    =\left(\frac{x}{h},
        \frac{\sqrt[\odot]{v}}{h}\right),
\]
is a homeomorphism with inverse $ (H_h^W)^{-1}(X,W) =\left(hX,h^2W^{\odot 2}\right).$ Moreover,
\[
H_h^W\circ F=G\circ H_h^W,
\qquad
H_h^W\circ F_I=G_I\circ H_h^W.
\]
Equivalently, the epsilon-outside maps satisfy
\[
G=H_h^W\circ F\circ(H_h^W)^{-1}.
\]
Consequently, each original-coordinate map and its normalized representation are topologically conjugate for every fixed $h>0$.
\end{proposition}

\begin{proof}
The componentwise square-root map is a homeomorphism of $\mathbb{R}_{\geq0}^d$, with inverse $W\mapsto W^{\odot2}$. The conjugacy identity follows directly by substituting $x=hX$ and $v=h^2W^{\odot2}$ into the corresponding RMSProp update.
\end{proof}

Thus the normalization preserves invariant sets, periodic orbits, Lyapunov stability, and asymptotic stability. The same conclusion holds for the mapping:
\[
H_h^w(x,v)
    =\left(\frac{x}{h},\frac{v}{h^2}\right),
\qquad
(H_h^w)^{-1}(X,w)=(hX,h^2w),
\]
for which the conjugacy is an invertible linear scaling.

Additionally, on compact invariant sets separated from $v_i=0$, this coordinate change and its inverse are smooth with bounded derivatives. Periodic-orbit Floquet multipliers and asymptotic Lyapunov exponents are then also preserved.

We also compare the size of the numerically observed attractors in the original and normalized coordinates. For each sampled learning rate, we discard an initial transient and estimate a radius about the sample mean from the remaining iterates, as described in the caption. Figure~\ref{fig:high-h-scaling} shows that, for large $h$, the normalized radii remain approximately constant, while the spatial and memory radii in the original coordinates grow approximately as $h$ and $h^2$. These measurements describe the size of the observed post-transient motion, while the trapping theorem gives an upper bound that applies to every orbit.

\begin{figure}[H]
    \centering
    \includegraphics[width=0.95\linewidth]{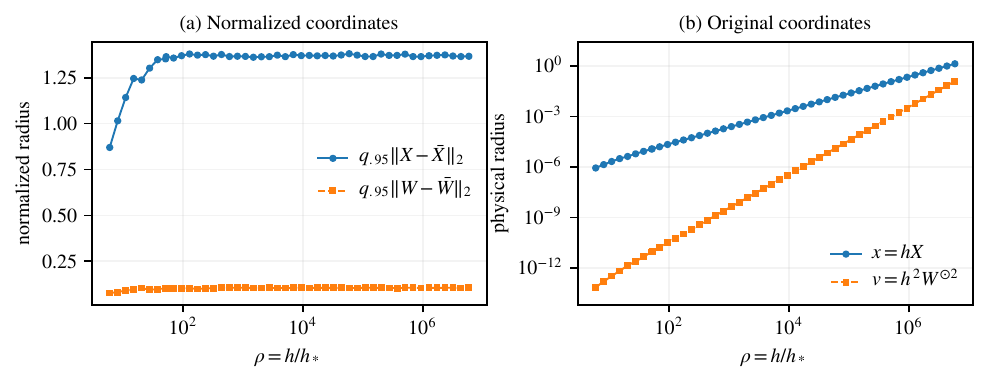}
    \caption{For the matrix and parameters in~\eqref{eq:case-matrix-rounded} and~\eqref{eq:canonical-parameters}, specified in Section~\ref{sec:numerical-parameters} below, large-$h$ orbit radii are approximately constant in normalized coordinates and scale as $h$ and $h^2$ in the original spatial and memory coordinates. For $Y\in\{X,W\}$, we define the centered radius $R_Y$ as the 95th percentile of the distances $\lVert Y_k-\bar Y\rVert_2$, where $\bar Y$ is the sample mean. We discard $100{,}000$ iterates and retain every fifth of the next $50{,}000$ states, giving $10{,}000$ observations. The right panel applies $x=hX$ and $v=h^2(W\odot W)$ before centering.}
    \label{fig:high-h-scaling}
\end{figure}

\section{Period-two orbits}\label{sec:period_two_orbits}

In the previous section, we show that every RMSProp orbit eventually enters a compact forward-invariant region. A natural next question is what kinds of dynamics occur there. We begin with period-two orbits for $h>h_*$.

We continue with the epsilon-outside map in normalized coordinates ($c=\epsilon/h$), given by~\eqref{eq:normalized-rmsprop-drift-lemma}. For a nontrivial period-two orbit $P_1=(X_1,W_1)\longleftrightarrow P_2=(X_2,W_2)$, we first show that the spatial coordinates are antipodal and the memory is the same at both points.

\begin{lemma}[Memory invariance and antipodal symmetry]\label{lem:period_two_invariance}
Every nontrivial period-two orbit satisfies
\begin{equation}\label{eq:period-two-symmetry}
W_1 = W_2 = |AX_1| = |AX_2| \quad\text{and}\quad X_2 = -X_1.
\end{equation}
\end{lemma}

\begin{proof}
Fix a coordinate and abbreviate $u=(W_1)_i$, $v=(W_2)_i$, $a=|(AX_1)_i|$, and $b=|(AX_2)_i|$. Eliminating the memory equations and equating the magnitudes of the opposing spatial displacements give
\[
(1+\beta)u^2=\beta a^2+b^2,\qquad
(1+\beta)v^2=a^2+\beta b^2,\qquad a(u+c)=b(v+c).
\]
If $a>b$, the last identity gives $au<bv$. Squaring and substituting the first two identities then gives $\beta a^4<\beta b^4$, a contradiction. Interchanging $a,b$ excludes $b>a$, so $a=b$ and $u=v=a$. This argument includes zero coordinates.

Thus $W_1=W_2=|AX_1|=|AX_2|$. Adding the spatial updates now yields $A(X_1+X_2)\oslash(W_1+c\mathbf1)=0$. Positivity of the denominator and invertibility of $A$ imply $X_2=-X_1$.
\end{proof}

\begin{theorem}[Reduced period-two equation]\label{thm:period_two_reduction}
Period-two cycles modulo $X\sim -X$ correspond bijectively to nonzero solutions of
\begin{equation}
2X = AX\oslash\bigl(|AX|+c\mathbf{1}\bigr) \quad\Longleftrightarrow\quad AX = 2cX\oslash(\mathbf{1}-2|X|), \qquad c=\epsilon/h. \label{eq:reduced_period_two}
\end{equation}
Every solution of the left equation satisfies $\|X\|_\infty<1/2$; the equivalence is understood on this open cube.
\end{theorem}

\begin{proof}
The preceding lemma gives $P_1=(X,|AX|)$ and $P_2=(-X,|AX|)$, so the spatial update is the left equation in~\eqref{eq:reduced_period_two}. Conversely, any nonzero solution makes the memory update fix $|AX|$ and the spatial update exchange the distinct points $P_1$ and $P_2$. The solutions $X$ and $-X$ describe the same geometric cycle.

Coordinatewise, $2|X_i|=|(AX)_i|/(|(AX)_i|+c)<1$, and $X_i$ and $(AX)_i$ have the same sign. Rearrangement gives $|(AX)_i|=2c|X_i|/(1-2|X_i|)$, hence the right equation. Conversely, within $(-1/2,1/2)^d$ that equation has positive denominators and reverses the same algebra, including coordinates with $X_i=0$.
\end{proof}

\subsection{Orbits near the stability threshold}

\begin{remark*}

To find period-two orbits, we use the functional
\begin{equation}
\Phi_c(X):=\frac12X^\top AX-c\sum_{i=1}^d p(X_i),\qquad p(t):=-\frac12\log(1-2|t|)-|t|,
\label{eq:period_two_existence_potential}
\end{equation}

We define it on $(-1/2,1/2)^d$. To find its critical points, we solve
\begin{equation}\label{eq:period-two-potential-gradient}
\nabla\Phi_c(X)=AX-2cX\oslash(\mathbf1-2|X|)=0.
\end{equation}

By Theorem~\ref{thm:period_two_reduction}, we obtain a period-two orbit from any nonzero critical point. To establish existence, it is therefore enough to find a nonzero maximizer of $\Phi_c$.
\end{remark*}

Since $\Phi_c$ tends to $-\infty$ at the boundary of its domain, it attains a maximum in the interior. For $h>h_*$, we find points near the origin where $\Phi_c$ is positive. We therefore obtain a positive maximum, attained away from the origin because $\Phi_c(0)=0$.

\begin{theorem}[Existence threshold]\label{thm:period_two_existence}
A nontrivial period-two orbit exists if and only if $h > h_* := 2\epsilon/L_A$ (equivalently, $0 < c < L_A/2$).
\end{theorem}

\begin{proof}
If $X\neq 0$ solves \eqref{eq:reduced_period_two}, then $X^\top AX = 2c\sum_{i} X_i^2/(1-2|X_i|) > 2c\|X\|_2^2$. Since $X^\top AX \le L_A\|X\|_2^2$, necessity ($2c < L_A$) follows immediately.

For sufficiency, use $\Phi_c$ from~\eqref{eq:period_two_existence_potential}. Since $p(t)\to\infty$ as $|t|\uparrow1/2$, while $X^\top AX$ is bounded on the cube, $\Phi_c\to-\infty$ at the boundary and attains a global maximum in the interior. The expansion $p(t)=t^2+\frac43|t|^3+O(t^4)$ determines its sign near the origin. For a unit eigenvector $e_{\max}$ associated with the largest eigenvalue $L_A$, so that $Ae_{\max}=L_Ae_{\max}$, expanding along $te_{\max}$ gives $\Phi_c(te_{\max}) = (\frac{L_A}{2}-c)t^2 - \frac{4c}{3}|t|^3\sum_i |e_{\max,i}|^3 + O(t^4)$. When $c < L_A/2$, $\Phi_c(te_{\max}) > 0 = \Phi_c(0)$ for small $t\neq 0$, so the maximum is positive and every global maximizer is nonzero. Such a maximizer is an interior critical point and therefore gives a period-two orbit by Theorem~\ref{thm:period_two_reduction}.
\end{proof}

\begin{remark*}

To find period-two orbits numerically, we maximize $\Phi_c$ in~\eqref{eq:period_two_existence_potential} and use~\eqref{eq:reduced_period_two} to recover the orbit, following the construction in Theorem~\ref{thm:period_two_existence}. We apply this method for $h>h_*$ close to the threshold in Section~\ref{sec:period_two_variational_experiment}.
\end{remark*}

\begin{theorem}[Collision with the origin]\label{thm:period_two_collision}
Every nonzero solution $X$ of \eqref{eq:reduced_period_two} satisfies $\|X\|_\infty \le \frac{d(L_A-2c)}{4c} = \frac{d}{2}(\frac{h L_A}{2\epsilon}-1)$. Thus, as $h \downarrow h_*$, every sequence of period-two orbits collapses to $(0,0)$.
\end{theorem}

\begin{proof}
Taking the inner product of \eqref{eq:reduced_period_two} with $X$ gives $4c\sum_i |X_i|^3/(1-2|X_i|) \le (L_A-2c)\|X\|_2^2$. For $r = \|X\|_\infty$, bounding the sum below by $r^3$ and $\|X\|_2^2 \le dr^2$ yields $4cr^3 \le d(L_A-2c)r^2$, proving the bound. As $c\uparrow L_A/2$, $X\to 0$ and $W=|AX|\to 0$.
\end{proof}

\begin{remark*}

For any sequence of period-two orbits with learning rates decreasing to $h_*$, we obtain convergence to the origin from Theorem~\ref{thm:period_two_collision}. Whether an individual branch can be continued all the way to $h_*$ is a separate question.
\end{remark*}

Under the additional generic assumptions that $L_A$ is simple and every coordinate of $e_{\max}$ is nonzero, we also show that the period-two branch emerging from the origin is locally exponentially asymptotically stable for $h>h_*$ sufficiently close to $h_*$. So far, we have established existence at each parameter without constructing a branch. We first construct a local branch approaching the origin, then use the Hessian of $\Phi_c$ to establish local uniqueness. Finally, we compute the Floquet multipliers to prove attraction.

\begin{lemma}[Construction of the period-two branch]\label{lem:near_threshold_period_two_branch}
Let $A=A^\top\succ0$, let $L_A=\lambda_{\max}(A)$ be simple, and let its unit eigenvector $e=e_{\max}$ satisfy $e_i\neq0$ for every $i$. Fix $\epsilon>0$ and $\beta\in(0,1)$, and set
\begin{equation*}
c=\epsilon/h,\qquad c_*=L_A/2,\qquad h_*=2\epsilon/L_A,\qquad
\delta=L_A-2c,\qquad \sigma_3=\sum_i|e_i|^3.
\end{equation*}
For $c<c_*$ sufficiently close to $c_*$, equivalently for $h>h_*$ sufficiently close to $h_*$, there is a smooth nontrivial period-two branch
\begin{equation}\label{eq:near_threshold_orbit}
P_1(c)=(X(c),W(c)),\qquad P_2(c)=(-X(c),W(c)),\qquad W(c)=|AX(c)|,
\end{equation}
with
\begin{equation*}
X(c)=\frac{\delta e}{2L_A\sigma_3}+O(\delta^2),\qquad
W(c)=\frac{\delta|e|}{2\sigma_3}+O(\delta^2).
\end{equation*}
Moreover, $W(c)>0$ and $\operatorname{sign}X(c)=\operatorname{sign}(AX(c))=\operatorname{sign}e$. Here $\epsilon$ remains fixed: $c\uparrow c_*$ corresponds to $h\downarrow h_*$ from above.
\end{lemma}

\begin{proof}
Use the parametrization
\begin{equation}\label{eq:near_threshold_parametrization}
X=t(e+z),\qquad z\in e^\perp,\qquad t>0.
\end{equation}
Substituting~\eqref{eq:near_threshold_parametrization} into~\eqref{eq:reduced_period_two} and dividing by $t$ gives
\begin{equation}\label{eq:near_threshold_residual}
\Psi(t,z,c):=A(e+z)-2c(e+z)\oslash(\mathbf1-2t|e+z|)=0.
\end{equation}
Every coordinate of $e$ is nonzero. Thus, in the neighborhood
\begin{equation*}
\|z\|_\infty<\tfrac12\min_i|e_i|,\qquad
\operatorname{sign}(e_i+z_i)=\operatorname{sign}e_i,\qquad |e+z|=\operatorname{diag}(\operatorname{sign}e)(e+z).
\end{equation*}
The absolute value is therefore a fixed linear map here, so it does not obstruct smoothness. For small $|t|$, the denominators remain positive; hence $\Psi$ extends smoothly through $(0,0,c_*)$.

To solve for $(z,c)$ as functions of $t$, the implicit-function theorem requires $D_{(z,c)}\Psi(0,0,c_*)$ to be invertible:
\begin{equation*}
\begin{gathered}
\Psi(0,0,c_*)=0, \quad
\mathcal L:=D_{(z,c)}\Psi(0,0,c_*):e^\perp\times\mathbb R\to\mathbb R^d.
\end{gathered}
\end{equation*}
Here the total derivative in $(z,c)$, with $t$ held fixed, is
\begin{equation}\label{eq:near_threshold_total_derivative}
D_{(z,c)}\Psi(t,z,c)(u,s)=D_z\Psi(t,z,c)u+\partial_c\Psi(t,z,c)s
=\left.\frac{d}{d\tau}\Psi(t,z+\tau u,c+\tau s)\right|_{\tau=0}.
\end{equation}
At $t=0$ the denominator in~\eqref{eq:near_threshold_residual} is $\mathbf1$, so~\eqref{eq:near_threshold_total_derivative} gives
\begin{equation}\label{eq:near_threshold_linearization}
\mathcal L(u,s)=\left.\frac{d}{d\tau}\bigl[A(e+\tau u)-2(c_*+\tau s)(e+\tau u)\bigr]\right|_{\tau=0}
=Au-2c_*u-2se=(A-L_AI)u-2se.
\end{equation}
Since $A$ is symmetric and $L_A$ is simple, $(A-L_AI)|_{e^\perp}$ is invertible. For any $y\in\mathbb R^d$, we have the unique orthogonal decomposition
\begin{equation}\label{eq:near_threshold_orthogonal_decomposition}
y=\underbrace{(e^\top y)e}_{\in\operatorname{span}\{e\}}
+\underbrace{(I-ee^\top)y}_{\in e^\perp}.
\end{equation}
Matching the components in~\eqref{eq:near_threshold_orthogonal_decomposition} with those in~\eqref{eq:near_threshold_linearization} gives
\begin{equation*}
\mathcal L(u,s)=y
\quad\Longleftrightarrow\quad
s=-\tfrac12e^\top y,\qquad
u=\bigl[(A-L_AI)|_{e^\perp}\bigr]^{-1}(I-ee^\top)y.
\end{equation*}
Thus the implicit-function theorem~\cite[Chapter~3]{krantzparks2013implicit} gives smooth $z(t),c(t)$ with $z(0)=0$, $c(0)=c_*$. Differentiating~\eqref{eq:near_threshold_residual} along $(t,z(t),c(t))$ at $t=0$ and then multiplying by $e^\top$ yields
\begin{equation}\label{eq:near_threshold_branch_derivative}
(A-L_AI)z'(0)-2c'(0)e-2L_Ae\odot|e|=0
\quad\Longrightarrow\quad
-2c'(0)-2L_A\sum_i|e_i|^3=0.
\end{equation}
By~\eqref{eq:near_threshold_branch_derivative}, $c'(0)=-L_A\sigma_3<0$, so $c(t)$ has a smooth local inverse. Since $c(0)=L_A/2$, Taylor expansion gives
\begin{equation}\label{eq:near_threshold_c_expansion}
c(t)=\frac{L_A}{2}-L_A\sigma_3t+O(t^2),\qquad
\delta(c(t))=L_A-2c(t)=2L_A\sigma_3t+O(t^2).
\end{equation}
Inverting~\eqref{eq:near_threshold_c_expansion} and substituting into~\eqref{eq:near_threshold_parametrization} and~\eqref{eq:near_threshold_orbit}, using $z(t)=O(t)$ and $Ae=L_Ae$, yields
\begin{align}
t&=\frac{\delta}{2L_A\sigma_3}+O(\delta^2),\notag
\\
X(c)&=t(e+z(t))=te+O(t^2)=\frac{\delta e}{2L_A\sigma_3}+O(\delta^2),\label{eq:near_threshold_X_expansion}\\
W(c)&=|AX(c)|=\left|\frac{\delta e}{2\sigma_3}+O(\delta^2)\right|=\frac{\delta|e|}{2\sigma_3}+O(\delta^2).\label{eq:near_threshold_W_expansion}
\end{align}
This constructs the branch for $c<c_*$, equivalently $h>h_*$ with $\epsilon$ fixed. Since every component of $e$ is nonzero, the leading terms give $W(c)>0$ and $\operatorname{sign}X(c)=\operatorname{sign}(AX(c))=\operatorname{sign}e$ for sufficiently small $\delta>0$. Equation~\eqref{eq:reduced_period_two} then shows that $G$ exchanges the distinct points $P_1(c)$ and $P_2(c)$.
\end{proof}

We next use the Hessian of $\Phi_c$ to prove local uniqueness of the period-two orbit at each fixed parameter, which is a necessary condition for hyperbolicity.

\begin{lemma}[Nondegeneracy of the period-two branch]\label{lem:near_threshold_period_two_nondegeneracy}
Under the assumptions and notation of Lemma~\ref{lem:near_threshold_period_two_branch}, the critical points $\pm X(c)$ are nondegenerate strict local maximizers of $\Phi_c$ from Equation~\eqref{eq:period_two_existence_potential} for sufficiently small $\delta>0$. Their Hessians are negative definite, and the eigenvalue approaching zero satisfies
\begin{equation*}
\lambda(\delta)=-\delta+O(\delta^2).
\end{equation*}
\end{lemma}

\begin{proof}
Recall that $\delta = L_A - 2c$. Now let $H_*=A-L_AI$ and $B=I-(2/\sigma_3)\operatorname{diag}(|e|)$. The expansion~\eqref{eq:near_threshold_X_expansion} from Lemma~\ref{lem:near_threshold_period_two_branch} gives
\begin{equation}\label{eq:near_threshold_hessian_denominator}
(1-2|X_i(c)|)^{-2}=1+4|X_i(c)|+O(\delta^2)
=1+\frac{2\delta|e_i|}{L_A\sigma_3}+O(\delta^2).
\end{equation}
Substituting~\eqref{eq:near_threshold_hessian_denominator} and $2c=L_A-\delta$ into the Hessian of $\Phi_c$ from~\eqref{eq:period_two_existence_potential} yields
\begin{equation}\label{eq:near_threshold_hessian}
\mathcal H(\delta):=D^2\Phi_c(X(c))
=A-2c\operatorname{diag}((1-2|X_i(c)|)^{-2})
=H_*+\delta B+O(\delta^2).
\end{equation}
The matrix $H_*$ has a simple zero eigenvalue with eigenvector $e$, and its other eigenvalues are negative. To determine how the zero eigenvalue changes, choose its smooth eigenpair:
\begin{equation}\label{eq:near_threshold_hessian_eigenpair}
\mathcal H(\delta)\xi(\delta)=\lambda(\delta)\xi(\delta),\qquad
\xi(0)=e,\quad\lambda(0)=0.
\end{equation}
Differentiating~\eqref{eq:near_threshold_hessian_eigenpair} with respect to $\delta$ and evaluating at zero gives the product-rule identity
\begin{equation*}
\mathcal H'(0)\xi(0)+\mathcal H(0)\xi'(0)
=\lambda'(0)\xi(0)+\lambda(0)\xi'(0).
\end{equation*}
Substitute $\mathcal H'(0)=B$, $\mathcal H(0)=H_*$ from~\eqref{eq:near_threshold_hessian}, and $\xi(0)=e$, $\lambda(0)=0$ from~\eqref{eq:near_threshold_hessian_eigenpair}:
\begin{equation}\label{eq:near_threshold_eigenpair_derivative}
Be+H_*\xi'(0)=\lambda'(0)e.
\end{equation}
Multiplying~\eqref{eq:near_threshold_eigenpair_derivative} by $e^\top$ eliminates $\xi'(0)$ because $e^\top H_*=0$. Since $e^\top e=1$ and $\sigma_3=\sum_i|e_i|^3$,
\begin{equation*}
\lambda'(0)=e^\top Be
=\sum_i e_i^2-\frac{2}{\sigma_3}\sum_i e_i^2|e_i|=1-2=-1.
\end{equation*}
Hence $\lambda(\delta)=-\delta+O(\delta^2)<0$; the other eigenvalues stay negative by continuity. Both critical points $\pm X(c)$ are therefore nondegenerate strict local maximizers, using the evenness of $\Phi_c$.

\end{proof}

\begin{theorem}[Stability of the period-two branch for \texorpdfstring{$h\gtrsim h_*$}{}]\label{thm:near_threshold_period_two_stability}
Under the assumptions of Lemma~\ref{lem:near_threshold_period_two_branch}, the period-two branch $\{P_1(c),P_2(c)\}$ is locally exponentially asymptotically stable for $h>h_*$ sufficiently close to $h_*$. The critical Floquet multiplier of $G^2$ satisfies
\begin{equation*}
\mu_{\mathrm{crit}}(c)=1-\frac{4(L_A-2c)}{L_A}+O((L_A-2c)^2),
\end{equation*}
and all other Floquet multipliers remain strictly inside the unit disk as $h\downarrow h_*$. The corresponding critical points $\pm X(c)$ are the nondegenerate strict local maximizers identified in Lemma~\ref{lem:near_threshold_period_two_nondegeneracy}.
\end{theorem}

\begin{proof}
Use the branch and notation from Lemma~\ref{lem:near_threshold_period_two_branch}, with $c$ fixed along each orbit. Define the linear symmetry
\begin{equation*}
T(X,W)=(-X,W),\qquad DT=T=\begin{bmatrix}-I_d&0\\0&I_d\end{bmatrix},\qquad T^2=I_{2d}.
\end{equation*}
The map satisfies $G\circ T=T\circ G$, and $P_2(c)=TP_1(c)=G(P_1(c))$. Differentiating the symmetry identity at $P_1(c)$ gives
\begin{equation*}
DG(P_2(c))\,T=T\,DG(P_1(c)),\qquad
DG(P_2(c))=T\,DG(P_1(c))\,T.
\end{equation*}
Thus, by the chain rule, the derivative of the second iterate is
\begin{equation*}
\begin{aligned}
D(G^2)(P_1(c))&=DG(P_2(c))\,DG(P_1(c))\\
&=T\,DG(P_1(c))\,T\,DG(P_1(c))=K(c)^2,
\end{aligned}
\qquad K(c):=T\,DG(P_1(c)).
\end{equation*}

The eigenvalues of $K(c)^2$ are the squares of the eigenvalues of $K(c)$. To see this, for any square matrix $K$, the factorization
\begin{equation*}
\det(K^2-\kappa^2I)=\det(K-\kappa I)\det(K+\kappa I)
\quad\Longrightarrow\quad
\operatorname{spec}(K^2)=\{\kappa^2:\kappa\in\operatorname{spec}(K)\}
\end{equation*}
holds over $\mathbb C$, since every complex number has a square root. Hence the Floquet multipliers are the squares of the eigenvalues of $K(c)$, including when $K(c)$ has nontrivial Jordan blocks.

We compute $DG(X,W)$ explicitly from~\eqref{eq:normalized-rmsprop-drift-lemma}. For a $G_2(X,W)$ with every component positive, set
\begin{equation*}
q=AX,\qquad D=\operatorname{diag}(G_2(X,W)),\qquad
\zeta=\operatorname{diag}(G_2(X,W)+c\mathbf1)=D+cI.
\end{equation*}
Differentiating the square-root memory update and then $G_1=X-q\oslash(G_2+c\mathbf1)$ gives
\begin{align}
D_XG_2&=(1-\beta)D^{-1}\operatorname{diag}(q)A,&
D_WG_2&=\beta D^{-1}\operatorname{diag}(W),\label{eq:near_threshold_general_memory_derivatives}\\
D_XG_1&=I-\zeta^{-1}A+\zeta^{-2}\operatorname{diag}(q)D_XG_2,&
D_WG_1&=\zeta^{-2}\operatorname{diag}(q)D_WG_2.\label{eq:near_threshold_general_spatial_derivatives}
\end{align}
Substituting~\eqref{eq:near_threshold_general_memory_derivatives} into~\eqref{eq:near_threshold_general_spatial_derivatives}, we obtain
\begin{equation}\label{eq:near_threshold_general_jacobian}
DG(X,W)=\begin{bmatrix}
I-\zeta^{-1}A+(1-\beta)\zeta^{-2}D^{-1}\operatorname{diag}(q^{\odot2})A
&\beta\zeta^{-2}D^{-1}\operatorname{diag}(q\odot W)\\
(1-\beta)D^{-1}\operatorname{diag}(q)A&\beta D^{-1}\operatorname{diag}(W)
\end{bmatrix}.
\end{equation}
At $P_1(c)$, put $X=X(c)$, $W=W(c)$, and $\zeta_c=\operatorname{diag}(W+c\mathbf1)$. The branch identities give
\begin{equation}\label{eq:near_threshold_jacobian_at_branch_identities}
G_2(P_1(c))=W=|AX|>0,\qquad
D^{-1}\operatorname{diag}(AX)=\operatorname{diag}(\operatorname{sign}e),\qquad
\zeta_c^{-1}AX=2X.
\end{equation}
Using~\eqref{eq:near_threshold_jacobian_at_branch_identities} in~\eqref{eq:near_threshold_general_jacobian} yields
\begin{equation}\label{eq:near_threshold_orbit_jacobian}
DG(P_1(c))=\begin{bmatrix}
I-\zeta_c^{-1}A+2(1-\beta)\zeta_c^{-1}\operatorname{diag}(X)\operatorname{diag}(\operatorname{sign}e)A&2\beta\zeta_c^{-1}\operatorname{diag}(X)\\
(1-\beta)\operatorname{diag}(\operatorname{sign}e)A&\beta I
\end{bmatrix}.
\end{equation}
Finally, multiplying~\eqref{eq:near_threshold_orbit_jacobian} by $T$ from the left changes the signs of the first block row:
\begin{equation}\label{eq:near_threshold_twisted_jacobian}
K(c)=\begin{bmatrix}
-I+\zeta_c^{-1}A-2(1-\beta)\zeta_c^{-1}\operatorname{diag}(X)\operatorname{diag}(\operatorname{sign}e)A&-2\beta\zeta_c^{-1}\operatorname{diag}(X)\\
(1-\beta)\operatorname{diag}(\operatorname{sign}e)A&\beta I
\end{bmatrix},
\end{equation}
Along the branch, $X(c),W(c)\to0$ and $\zeta_c\to c_*I$ as $c\uparrow c_*=L_A/2$. Consequently, $K(c)$ has the following one-sided limit:
\begin{equation}\label{eq:near_threshold_limiting_jacobian}
K_*:=\lim_{c\uparrow c_*}K(c)=\begin{bmatrix}-I+2A/L_A&0\\(1-\beta)\operatorname{diag}(\operatorname{sign}e)A&\beta I\end{bmatrix}.
\end{equation}
The eigenvalues of $K_*$ in~\eqref{eq:near_threshold_limiting_jacobian} are $\beta$ and $2\lambda_i(A)/L_A-1$. Only the simple eigenvalue $1$ lies on the unit circle; its right and left eigenvectors are
\begin{equation}\label{eq:near_threshold_critical_eigenvectors}
r=\begin{pmatrix}e\\L_A|e|\end{pmatrix},\qquad
\ell=\begin{pmatrix}e\\0\end{pmatrix},\qquad
K_*r=r,\quad\ell^\top K_*=\ell^\top,\quad\ell^\top r=1.
\end{equation}
The reduced equation~\eqref{eq:reduced_period_two} and the branch expansions~\eqref{eq:near_threshold_X_expansion}--\eqref{eq:near_threshold_W_expansion} give
\begin{equation}\label{eq:near_threshold_multiplier_identities}
W_i+c=\frac{c}{1-2|X_i|},\qquad
\sum_i e_i^2|X_i|=\frac{\delta}{2L_A}+O(\delta^2),\qquad
\frac{L_A}{c}=2+\frac{2\delta}{L_A}+O(\delta^2).
\end{equation}
Substituting~\eqref{eq:near_threshold_multiplier_identities} into the first block row of~\eqref{eq:near_threshold_twisted_jacobian}, with the eigenvectors~\eqref{eq:near_threshold_critical_eigenvectors}, yields
\begin{equation}\label{eq:near_threshold_multiplier_projection}
\ell^\top(K(c)-K_*)r
=\frac{L_A}{c}\sum_i e_i^2(1-2|X_i|)^2-2
=\left(2+\frac{2\delta}{L_A}\right)\left(1-\frac{2\delta}{L_A}\right)-2+O(\delta^2)
=-\frac{2\delta}{L_A}+O(\delta^2).
\end{equation}
Since $K(c)-K_*=O(\delta)$, simple-eigenvalue perturbation and~\eqref{eq:near_threshold_multiplier_projection} give
\begin{equation*}
\kappa_{\mathrm{crit}}=1+\ell^\top(K(c)-K_*)r+O(\delta^2)
=1-\frac{2\delta}{L_A}+O(\delta^2),\qquad
\mu_{\mathrm{crit}}=\kappa_{\mathrm{crit}}^2=1-\frac{4\delta}{L_A}+O(\delta^2)\in(0,1).
\end{equation*}
All other Floquet multipliers remain inside the unit disk by continuity, and symmetry gives the same spectrum at $P_2$. By~\eqref{eq:near_threshold_W_expansion}, $W(c)>0$, so $G$ is smooth near both points. The discrete-map stability theorem~\cite[Chapter~4]{elaydi2005difference} therefore proves local exponential asymptotic stability.
\end{proof}

\subsection{Period-two orbits for \texorpdfstring{$h\gg h_*$}{h much greater than h*}}

For large learning rates, we look for period-two orbits whose normalized spatial coordinates are close to the vertices of the cube $[-1/2,1/2]^d$. We call these vertices \emph{corners} and write them as $s/2$, with $s\in\{-1,1\}^d$. Below, we give a condition for finding such an orbit near a given corner. And in fact, for every symmetric positive definite $A$, at least one pair of opposite corners satisfies this condition, so we obtain a period-two orbit for all sufficiently large $h$. The condition need not hold at every corner, and other period-two orbits may also exist.

\begin{definition}[Admissible corner]\label{def:admissible_corner}
A sign vector $s\in\{-1,1\}^d$ is \emph{weakly admissible} for $A$ if $s_i(As)_i\geq0$ for all $i=1,\ldots,d$, and is \emph{admissible} (or \emph{strictly admissible}) for $A$ if $s_i(As)_i>0$ for all $i=1,\ldots,d$ (that is, $As$ shares the componentwise sign pattern of $s$).
\end{definition}

\begin{theorem}[Corner asymptotics]\label{thm:corner_sign_consistency}
If a solution branch satisfies $X(c) \to s/2$ as $c\downarrow 0$, then $s_i(As)_i\geq0$ for all $i$, so $s$ is weakly admissible. Conversely, every strictly admissible sign vector $s$ generates a unique smooth solution branch near $s/2$ with asymptotic expansion
\begin{equation}
X_s(c) = \frac{s}{2} - c\,(s\oslash|As|) + O(c^2) \quad\Longleftrightarrow\quad \widetilde{X}_s(h) = \frac{s}{2} - \frac{\epsilon}{h}\,(s\oslash|As|) + O(h^{-2}) \quad(h\to\infty).\label{eq:corner_branch_expansion}
\end{equation}
\end{theorem}

\begin{proof}
If $X(c)\to s/2$, then $\operatorname{sign}(X(c)) = \operatorname{sign}(AX(c)) = s$ for small $c>0$. Taking $c\downarrow 0$ yields $s_i(As)_i \ge 0$.

Conversely, when $s$ is admissible, $\operatorname{sign}(AX) = s$ holds in a full neighborhood of $s/2$. On this neighborhood, let $\mathcal E(X,c)$ denote the left-hand side minus the right-hand side of the first equation in~\eqref{eq:reduced_period_two}. This residual is smooth with $\mathcal E(s/2,0)=0$ and Jacobian $D_X\mathcal E(s/2,0)=2I_d$. By the Implicit Function Theorem, there exists a unique local curve $X_s(c)$ satisfying $\mathcal E(X_s(c),c)=0$. Differentiating implicitly at $c=0$ gives $X_s'(0) = -s\oslash|As|$, establishing \eqref{eq:corner_branch_expansion}.
\end{proof}

\begin{theorem}[Existence of admissible corners]\label{thm:spd_admissible_corner}
Every SPD matrix $A \succ 0$ possesses at least one admissible corner. Consequently, a corner period-two orbit exists for all sufficiently large $h$.
\end{theorem}

\begin{proof}
Let $s^* \in \operatorname{arg\,max}_{s\in\{-1,1\}^d} s^\top As$. Flipping the $i$th coordinate to form $\widehat{s} = s^* - 2s_i^*e_i$, maximality requires $0 \le (s^*)^\top As^* - \widehat{s}^\top A\widehat{s} = 4(s_i^*(As^*)_i - A_{ii})$. Since $A\succ 0$, the diagonal entries satisfy $A_{ii} > 0$, hence $s_i^*(As^*)_i \ge A_{ii} > 0$ for every coordinate $i$.
\end{proof}

Under the condition, which we call \textit{diagonal dominance}, where $A_{ii}>\sum_{j\ne i}|A_{ij}|$, we have:
\[
s_i(As)_i\ge A_{ii}-\sum_{j\ne i}|A_{ij}|>0
\]
for every sign vector $s$. Theorem~\ref{thm:corner_sign_consistency} then gives a local branch near each of the $2^d$ corners. Since opposite corners describe the same period-two orbit, these branches give $2^{d-1}$ distinct period-two orbits for sufficiently small $c>0$.

\begin{remark}
We have established corner branches for sufficiently large $h$. To relate them to the branches near $h_*$, we must study their continuation between the two regimes. This is especially important because the results above do not determine whether a corner branch reaches $h_*$ or whether a branch emerging from the origin approaches a corner as $h\to\infty$. We investigate these questions numerically in Section~\ref{sec:period_two_corner_experiment}.
\end{remark}

\section{Numerical conventions and methods}
\label{sec:numerical_methods}

To study the dynamics in the trapping region, we obtain numerical parameterizations of invariant circles and tori using the parameterization method (see Section~\ref{rem:parm}), then follow their branches as the learning rate varies. We study their stability through normal Lyapunov exponents and compare the results with the behavior of nearby trajectories.

\subsection{Parameters, map convention, and orientation}
\label{sec:numerical-parameters}

For the three-dimensional case study, we use the normalized epsilon-outside map $G$ with the parameters
\begin{equation}\label{eq:canonical-parameters}
d=3,\qquad \beta=0.98,\qquad \epsilon=10^{-7}.
\end{equation}

We write $Z=(X,W)\in\mathbb R^6$. We use the second iterate $P=G^2$ to study each component of the paired circles separately. Throughout the remainder of this paper, we specify whether iteration counts and rotation numbers refer to $G$ or $P$.

Unless stated otherwise, the trajectory examples and parameter surveys use the normalized initial state
\begin{equation}\label{eq:canonical-initial-state}
X_0=(0.3,0.4,-0.3),\qquad W_0=(0.2,0.2,0.2).
\end{equation}
For autonomous Adam (Section~\ref{sec:Adam}), we initialize the momentum coordinates at zero. We state separately when a trajectory is continued from the final state at the preceding parameter value.

Because RMSProp rescales gradients coordinatewise, we specify both the eigenvalues and the orientation of the Hessian. We construct a positive-definite matrix from a diagonal matrix $\Lambda=\operatorname{diag}(\lambda_1,\lambda_2,\lambda_3)$ with positive entries by taking
\begingroup
\begin{equation}\label{eq:hessian-rotation-construction}
A(\vartheta)=Q(\vartheta)\Lambda Q(\vartheta)^\top,
\qquad Q(\vartheta)=R_z(\vartheta_z)R_y(\vartheta_y)R_x(\vartheta_x).
\end{equation}
\endgroup
Here $Q$ is a rotation matrix using the Euler angle conventions from~\cite[Chapter~2]{murray1994robotic}:
\begingroup
\small
\begin{equation}\label{eq:coordinate-rotations}
R_x(\alpha)=\begin{pmatrix}1&0&0\\0&\cos\alpha&-\sin\alpha\\0&\sin\alpha&\cos\alpha\end{pmatrix},\quad
R_y(\alpha)=\begin{pmatrix}\cos\alpha&0&\sin\alpha\\0&1&0\\-\sin\alpha&0&\cos\alpha\end{pmatrix},\quad
R_z(\alpha)=\begin{pmatrix}\cos\alpha&-\sin\alpha&0\\\sin\alpha&\cos\alpha&0\\0&0&1\end{pmatrix}.
\end{equation}
\endgroup

For the main case study, we use a fixed approximation to the matrix in~\eqref{eq:hessian-rotation-construction}. We take $(\vartheta_x,\vartheta_y,\vartheta_z)=(60^\circ,-60^\circ,-143.5^\circ)$ and $\Lambda=\operatorname{diag}(0.5,0.8,1.18)$. The entries of $A$ are given in Section~\ref{sec:case-matrix} and rounded to four decimal places below.
\begingroup
\begin{equation}\label{eq:case-matrix-rounded}
A\approx\begin{pmatrix}
0.7621 & -0.0964 & 0.0886\\
-0.0964 & 1.1191 & 0.1679\\
0.0886 & 0.1679 & 0.5988
\end{pmatrix}.
\end{equation}
\endgroup

This gives $h_*\approx1.6949\times10^{-7}$ by~\eqref{eq:h_star}. Finally, in the orientation survey of Section~\ref{sec:angle-survey}, we vary the angles in~\eqref{eq:hessian-rotation-construction} while keeping the eigenvalues fixed.

\subsection{Computation and continuation of invariant circles and tori}
\label{sec:parametrization_methods_stability}

We numerically compute Fourier series expansions for attracting invariant circles 
and invariant tori using the parameterization method.

In the numerical examples below, we repeatedly find invariant sets in disjoint pairs (chaotic attractors, circles, periodic points), with iterates alternating between the two, as illustrated in Figures~\ref{fig:primary-crossing-composite} and~\ref{fig:four-panel-circles}. We call a pair of disjoint circles exchanged by $G$ a \textit{period-two invariant circle} and parameterize its two components together. Here period two refers to the exchange of components, independently of the motion within each component.

The two components $K_1,K_2:\mathbb T\to\mathbb R^6$, with $\mathbb T=\mathbb R/(2\pi\mathbb Z)$, satisfy
\begin{equation}
G(K_1(\theta))=K_2(\theta+\omega),\qquad G(K_2(\theta))=K_1(\theta+\omega).
\label{eq:invariant-circle-pair}
\end{equation}
Thus each component is invariant under $G^2$, with
\begin{equation}\label{eq:circle-return-rotation}
G^2(K_j(\theta))=K_j(\theta+2\omega),\qquad \rho_P=\omega/\pi\pmod1.
\end{equation}
Here $\omega$ is the angular increment in~\eqref{eq:invariant-circle-pair}. We define the rotation number $\rho_P$ of $P$ as the number of turns per iteration of $P=G^2$, modulo one. In the numerical tables we choose its signed representative in $(-1/2,1/2]$. Thus a negative value records rotation in the opposite direction for the chosen orientation of the parameterization.

We solve the above invariance equations, together with phase conditions, for the Fourier coefficients and rotation number $\rho_P$. Appendix~\ref{app:fourier_newton} gives a more detailed mathematical formulation for the parameterization of invariant circles and two-tori.

Since the parameterization method uses an iterative scheme to solve a functional equation 
describing the invariant circle, an essential step is obtaining a good enough initial 
approximation of the rotation number and Fourier coefficients.  For this we exploit
the Birkhoff averaging method of Das, Saiki, Sander, and Yorke~\cite{MR3526529,das2017quantitative,MR3755876} (see also~\cite{MR4104977}). In this sense, our work extends the approach of Blessing and 
Mireles James~\cite{blessing2023weighted} to dissipative maps in higher dimensions.
We then improve these approximations with Newton's method and follow the branches by fixed-parameter or pseudo-arclength continuation~\cite{keller1977numerical}.

Another advantage of solving a functional equation is that, by 
treating the learning rate as a bifurcation parameter, 
we follow attracting invariant circles through the loss of stability
via numerical continuation.
Section~\ref{sec:continuation-checks} gives the continuation procedure and accuracy checks.

\subsection{Lyapunov exponents and orbit classification}
\label{sec:trajectory-diagnostics}

We estimate the leading finite-time Lyapunov exponent (FTLE) $\widehat\lambda_1$ using the method of Benettin et al.~\cite{benettin1980part2}, with an unweighted average over $N$ consecutive iterations after discarding an initial transient. In this full-state estimate, we include perturbations in all spatial and memory variables, as well as the momentum variables for Adam. For invariant circles and tori, we compute normal Lyapunov exponents $\widehat\lambda_{\perp,i}$, which exclude tangent directions~\cite{jorba2001normal}. Appendix~\ref{app:fourier_newton} gives their construction and the weighted averages that we use for the circle and torus stability comparisons in Sections~\ref{sec:coexisting-circles} and~\ref{sec:c0-secondary-response}.

We report full-state Lyapunov exponent estimates per iterate of $G$ for RMSProp and $G_{\beta_1}$ for autonomous Adam, and normal Lyapunov exponents per iteration of $P=G^2$. Dividing the latter by two expresses them per iterate of $G$. Table~\ref{tab:individual-lyapunov-values} in Appendix~\ref{app:additional-numerical-methods} gives the full-state estimates and the numbers of discarded and measured iterations for the individual trajectory examples.

\subsubsection{Recurrence tests}

For $N$ measured states $Z_0,\ldots,Z_{N-1}$ after the discarded transient, we define the \textit{normalized recurrence error} using the final $m$ states:
\begin{equation}
 e_p=\frac{\operatorname{median}_{N-m+p\le k<N}\|Z_k-Z_{k-p}\|_2}
 {\max\{1,\max_{N-m\le k<N}\|Z_k\|_2\}},\qquad p=1,2.
 \label{eq:recurrence-error}
\end{equation}
For example, small $e_2$ together with appreciable $e_1$ suggests period-two motion. We use $m=128$ in the two-parameter RMSProp and Adam surveys, and $m=1{,}024$ in the Hessian-orientation survey.

\subsubsection{Criteria used in the parameter surveys}
\label{sec:survey-classification}

We group invariant circles, which are one-dimensional tori, and higher-dimensional tori in a single torus category for the parameter surveys. We first screen for fixed-point and period-two behavior using recurrence, then classify the remaining trajectories using the largest finite-time Lyapunov exponent. For the parameter surveys of Sections~\ref{sec:angle-survey}, \ref{sec:beta-survey}, and~\ref{sec:Adam}, we use the following four classification criteria, which do not cover every finite-window outcome:
\begingroup
\begin{equation}
\begin{array}{ll}
\text{Fixed point:}&e_1\le10^{-6},\\
\text{Period two:}&e_1>10^{-6},\quad e_2\le10^{-6},\\
\text{Torus:}&|\widehat\lambda_1|\le2\times10^{-4},\quad e_1,e_2>10^{-6},\\
\text{Chaotic:}&\widehat\lambda_1>2\times10^{-4},\quad e_1,e_2>10^{-6}.
\end{array}
\label{eq:survey-classification}
\end{equation}
\endgroup

For each survey, we discard 50,000 iterates of $G$, or $G_{\beta_1}$ for Adam, and measure Lyapunov exponents over the next 50,000. The parameter grids and the values of $m$ are given in Section~\ref{sec:survey-measurements}.

\begin{remark}
We use fairly coarse classification criteria for our parameter surveys. The recurrence checks do not systematically identify longer-period orbits, and the Lyapunov exponent cutoff limits our ability to distinguish quasiperiodic motion from weakly chaotic trajectories. We flag samples in the torus class with estimated Lyapunov exponents above $10^{-5}$ per iterate, together with trajectories that fit none of the four criteria, for longer computations and manual inspection. The surveys describe finite-window outcomes from one initial state at each grid point, so they do not determine the asymptotic dynamics or detect every coexisting attractor.
\end{remark}

\subsection{Computation of period-two orbits}\label{sec:finding_period_two_orbits}

The analytical results suggest two ways to find period-two orbits numerically. Near $h_*$, we use the variational formulation in Theorem~\ref{thm:period_two_existence} to search for them by optimization. For $h\gg h_*$, we use the corner expansions from Theorem~\ref{thm:corner_sign_consistency} to obtain initial approximations, then continue the branches toward smaller learning rates. For the corner continuation below, we use the parameters in Section~\ref{sec:numerical-parameters}.

\subsubsection{Finding period-two orbits near the threshold}\label{sec:period_two_variational_experiment}

For $h>h_*$ close to the threshold, we locate nonzero critical points of $\Phi_c$ from~\eqref{eq:period_two_existence_potential} on $(-1/2,1/2)^d$. We choose small random initial vectors near the eigenspace of $A$ for $\lambda_{\max}(A)$, as suggested by Theorem~\ref{thm:near_threshold_period_two_stability}. Starting from these vectors, we minimize $-\Phi_c$ using a box-constrained optimization algorithm such as L-BFGS-B~\cite{byrd1995limited,zhu1997algorithm}, with bounds strictly inside the open hypercube.

We refine the resulting critical points with Newton's method applied to~\eqref{eq:period-two-potential-gradient}, then recover the orbit points from~\eqref{eq:period-two-symmetry}. After identifying $X$ and $-X$ as the same period-two orbit, we check both map equations.

\subsubsection{Continuation from admissible corners}\label{sec:period_two_corner_experiment}

The corners are $s/2\in\frac12\{-1,1\}^d$, where $s\in\{-1,1\}^d$. We identify $s$ and $-s$, since they describe the same period-two orbit. In three dimensions, we therefore consider four representative corners, $s_1=(+++)$, $s_2=(++-)$, $s_3=(+-+)$, and $s_4=(+--)$, all of which are admissible for the matrix $A$ specified in Section~\ref{sec:numerical-parameters}. For each representative corner, we initialize at large $h$ using the leading terms of~\eqref{eq:corner_branch_expansion} and track branches toward smaller $h$ using pseudo-arclength continuation~\cite{keller1977numerical}. Along each path, we compute the Floquet multipliers to determine orbital stability.

We plot the continued branches against the normalized learning rate
\begin{equation}\label{eq:normalized-learning-rate}
\rho=h/h_*.
\end{equation}
For the matrix $A$ specified in Section~\ref{sec:numerical-parameters}, branch $s_4$ connects directly to the origin, whereas $s_1, s_2, s_3$ encounter fold bifurcations at $\rho \approx 2.18839, 13.35115, 4.41042$.

\begin{figure}[htbp]
\centering
\includegraphics[width=0.95\textwidth]{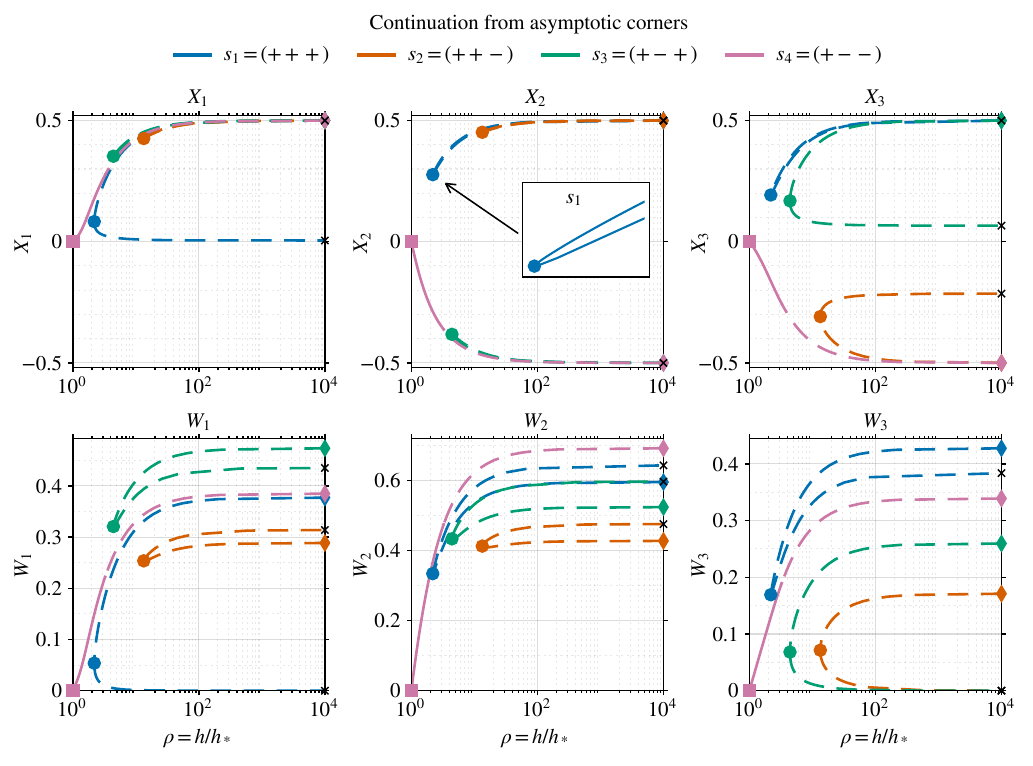}
\caption{We begin near the corners $s/2\in\frac12\{-1,1\}^d$ at large $h$ and continue the period-two branches toward smaller learning rates, using the matrix~\eqref{eq:case-matrix-rounded} and parameters~\eqref{eq:canonical-parameters}. The upper and lower rows show $X$ and $W$ versus $\rho=h/h_*$. Solid segments mark attracting parts of the branches, and dashed segments mark parts that are unstable under $G^2$, with at least one Floquet multiplier outside the unit circle.  Markers indicate fold bifurcations ($\bullet$), asymptotic corners ($\blacklozenge$), collision with the origin ($\blacksquare$), and last computed points ($\times$). The inset in the $X_2$ panel enlarges the sharp turn of the blue branch near $\rho=2.19$. The red and green branches have analogous sharp turns in the $X_1$ and $X_2$ projections.}
\label{fig:period_two_case_A_corner_paths}
\end{figure}

\section{Numerical bifurcation analysis}\label{sec:bifurcation-analysis}
\label{sec:dynamics_trapping_region}

We now return to the question raised in the introduction: what kinds of dynamics occur in the trapping region? We study the normalized RMSProp map with the matrix~\eqref{eq:case-matrix-rounded} and parameters~\eqref{eq:canonical-parameters}, using $P=G^2$ for each component of a period-two circle. Figure~\ref{fig:rmsprop-master-bifurcation-diagram} gives an overview as we increase the learning rate. We then examine three features in more detail: coexistence of invariant circles, a secondary Neimark--Sacker bifurcation, and a circle-doubling sequence. We also find intervals of learning rates for which the computed trajectories have positive finite-time Lyapunov exponents. Appendix~\ref{app:asym-numerical-inventory} collects additional sampled transitions and period-two circles. Figures~\ref{fig:same-plot-circles} and~\ref{fig:six-coexisting-sets} show further projections of the coexisting sets, and Appendix~\ref{app:orbit-gallery} gives additional trajectories.

\begin{figure}[!htbp]
\centering
\includegraphics[width=\linewidth]{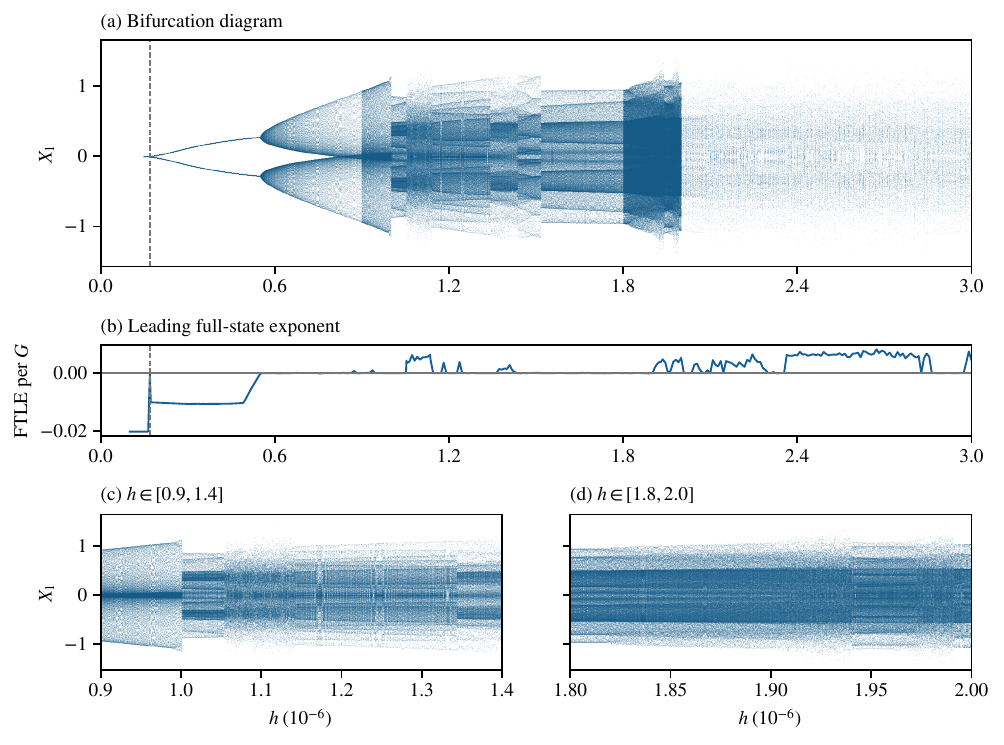}
\caption{Bifurcation diagram for the matrix~\eqref{eq:case-matrix-rounded} and parameters~\eqref{eq:canonical-parameters}.  For the upper panel, we begin with $X_0=W_0=10^{-8}\mathbf1$ and sample 9,000 increasing learning rates. We carry the final state forward with a small spatial perturbation before each run, keeping the initial memory coordinates at least $10^{-8}$. At each learning rate, we discard 50,000 iterates of $G$ and plot 150 values of $X_1$. The middle panel gives the leading full-state Lyapunov exponent per iterate of $G$ from independently initialized trajectories with 50,000 discarded and 20,000 measured iterates. These trajectories start from~\eqref{eq:canonical-initial-state}, so they may approach different attractors from those in the upper panel when several coexist. The bottom panels enlarge $h\in[0.9,1.4]\times10^{-6}$ and $h\in[1.8,2.0]\times10^{-6}$. The dashed vertical line marks $h_*$.}
\label{fig:rmsprop-master-bifurcation-diagram}
\end{figure}

\subsection{The primary invariant circle}

As we increase the learning rate, we find a conjugate pair of Floquet multipliers of the period-two orbit crossing the unit circle near $h_c=5.5222246\times10^{-7}$, suggesting a Neimark--Sacker bifurcation. The modulus of the multiplier pair rises from $0.9991767$ at $h=5.4946135\times10^{-7}$ to $1.0016684$ at $h=5.5792602\times10^{-7}$. Above this crossing, we compute the numerically attracting period-two circle $C_{\mathrm{NS}}$, shown in Figure~\ref{fig:primary-crossing-composite}.

\begin{figure}[!htbp]
\centering
\includegraphics[width=\linewidth]{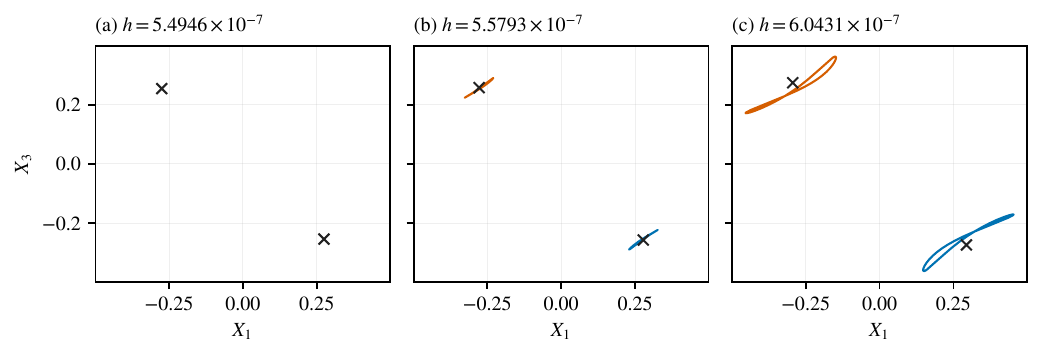}
\caption{The appearance of the primary invariant circle for the matrix~\eqref{eq:case-matrix-rounded} and parameters~\eqref{eq:canonical-parameters} in the $X_1$--$X_3$ projection. We use $h=5.4946$, $5.5793$, and $6.0431$ times $10^{-7}$ in the three panels. Crosses mark the two points of the period-two orbit while blue and orange curves show the alternating components of $C_{\mathrm{NS}}$.}
\label{fig:primary-crossing-composite}
\end{figure}

We call a circle component $q$-fold when its displayed projection has $q$ nearby strands. As $h$ increases, $C_{\mathrm{NS}}$ loses normal stability. We first observe a sign change in its leading normal Lyapunov exponent near $h=0.99245872\times10^{-6}$, together with a 2-fold circle candidate. By $h=1.0514\times10^{-6}$, $C_{\mathrm{NS}}$ has one numerically unstable normal direction.

\FloatBarrier
\subsection{Multistability}
\label{sec:coexisting-circles}

At $h=1.0514\times10^{-6}$, we find three numerically attracting period-two circles $C_0,C_1,C_2$, together with the saddle period-two circle $C_{\mathrm{NS}}$. Table~\ref{tab:asym-common-circles} gives their rotation numbers and leading normal Lyapunov exponents. For each of the three attracting circles, we find trajectories from nearby initial conditions that remain close for five million iterations of $P$, providing numerical evidence of multistability at this parameter value.

\begin{table}[htbp]
\centering\small
\caption{Four period-two circles for $A$ at $h=1.0514\times10^{-6}$. Rotation numbers and normal Lyapunov exponents are per iteration of $P=G^2$. We count signs of the weighted estimates using a threshold of $10^{-7}$ per iteration of $P$, with measurement checks in Section~\ref{sec:normal-lyapunov}.}
\label{tab:asym-common-circles}
\begin{tabular}{@{}lcc@{}}
\toprule
Circle & $\rho_P$ & $\widehat\lambda_{\perp,1}^{(P)}$\\
\midrule
$C_{\mathrm{NS}}$ & $0.03369894$ & $+6.086\times10^{-2}$\\
$C_0$ & $-0.01226718$ & $-1.224\times10^{-6}$\\
$C_1$ & $0.00244958$ & $-2.698\times10^{-5}$\\
$C_2$ & $0.00275643$ & $-4.071\times10^{-4}$\\
\bottomrule
\end{tabular}
\end{table}

\begin{figure}[!htbp]
\centering
\includegraphics[width=0.95\linewidth]{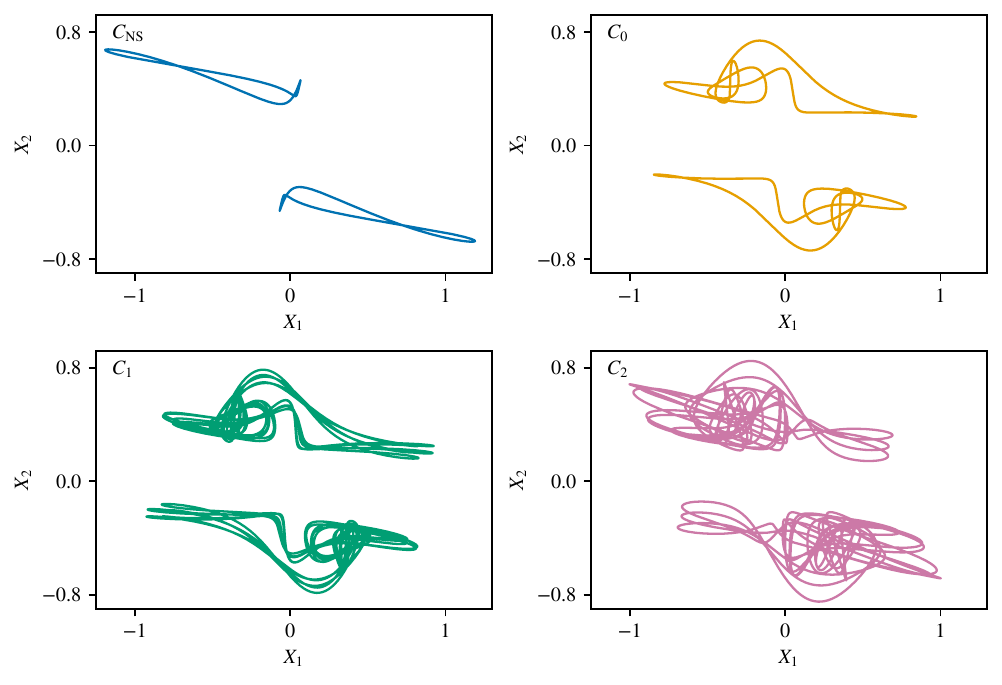}
\caption{Four distinct period-two circles at $h=1.0514\times10^{-6}$ for the matrix~\eqref{eq:case-matrix-rounded} and parameters~\eqref{eq:canonical-parameters}. Each panel shows both circle components in the $X_1$--$X_2$ projection, with iterates alternating between them under $G$. The circles $C_0,C_1,C_2$ attract nearby trajectories whereas $C_{\mathrm{NS}}$ has one unstable normal direction.}
\label{fig:four-panel-circles}
\end{figure}

\subsection{A secondary Neimark--Sacker bifurcation}
\label{sec:c0-secondary-response}

Visual evidence suggests a secondary Neimark--Sacker bifurcation of $C_0$ near $h=1.05147235\times10^{-6}$. The mean of the leading pair of normal Lyapunov exponents changes sign between $h=1.05147230\times10^{-6}$ and $h=1.05147240\times10^{-6}$. Below this interval, where $C_0$ is numerically attracting, we compute a small, weakly unstable saddle-torus candidate. This ordering suggests a subcritical bifurcation. At $h=1.0514\times10^{-6}$, we compute an attracting two-torus and a saddle two-torus that coexist with $C_0$. Figures~\ref{fig:c0-candidate-torus} and~\ref{fig:saddle-torus-c0} compare their geometry and sections with those of $C_0$.

\begin{figure}[!htbp]
\centering
\includegraphics[width=0.89\linewidth]{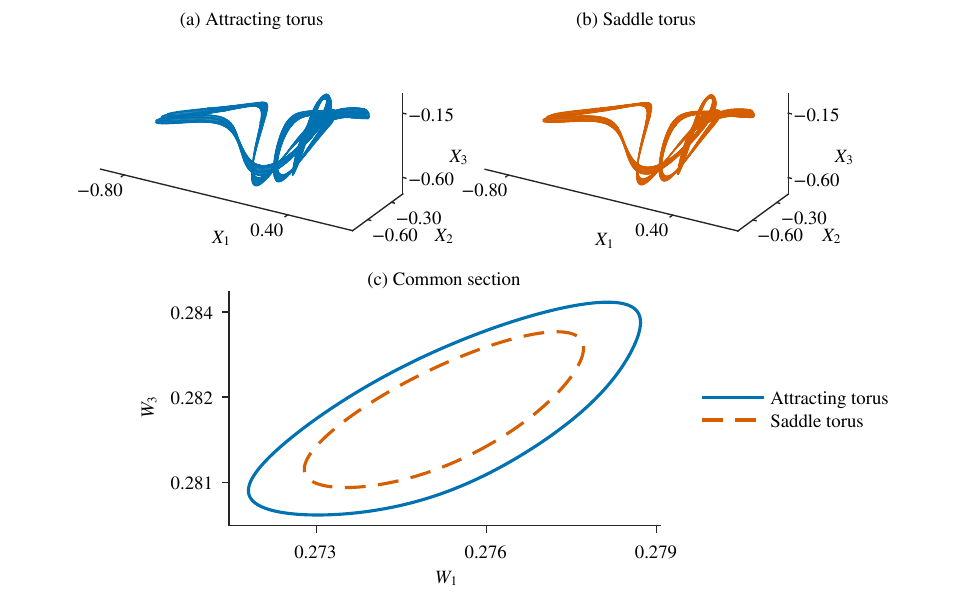}
\caption{Attracting and saddle two-tori for the matrix~\eqref{eq:case-matrix-rounded} and parameters~\eqref{eq:canonical-parameters} at $h=1.0514\times10^{-6}$. Panels (a,b) show one $P$-invariant component of each torus in $X_1$--$X_2$--$X_3$, and applying $G$ gives the other component. Panel (c) shows sections at $X_1=0.18905212$. We keep points with $X_2>-0.5$ at which $X_1$ increases as the first torus angle $\theta_1$ increases (Appendix~\ref{app:fourier_newton}).}
\label{fig:c0-candidate-torus}
\end{figure}

\begin{figure}[!htbp]
\centering
\includegraphics[width=0.89\linewidth]{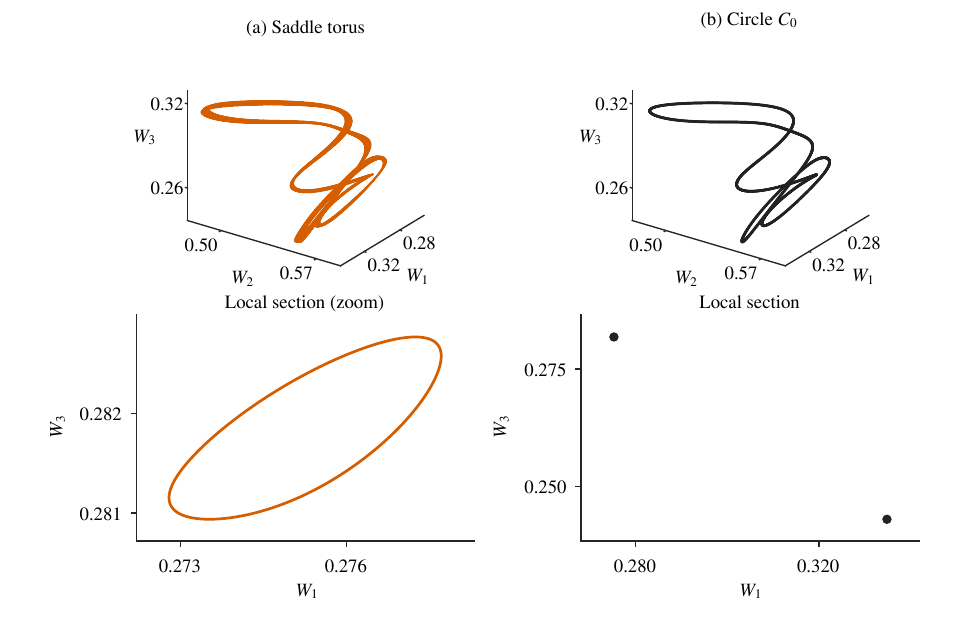}
\caption{The saddle torus and $C_0$ for the matrix~\eqref{eq:case-matrix-rounded} and parameters~\eqref{eq:canonical-parameters} at $h=1.0514\times10^{-6}$. Each lower panel shows the section $X_1=0.18905212$ in $W_1$--$W_3$. The torus section is a closed curve, whereas the circle gives two section points. The lower panels use different scales to compare the two geometries.}
\label{fig:saddle-torus-c0}
\end{figure}

\subsection{Chaotic dynamics}
\label{sec:chaotic_dynamics}

At $h=1.06\times10^{-6}$, Figure~\ref{fig:chaotic-attractor-1.06} shows the trajectory from the initial state~\eqref{eq:canonical-initial-state} after discarding 250,000 iterations of $P$. Its geometry and positive leading full-state Lyapunov exponent over the measured iterations suggest chaotic motion. We leave the bifurcations leading to this chaotic regime for future work.

\begin{figure}[!htbp]
\centering
\includegraphics[width=0.95\linewidth]{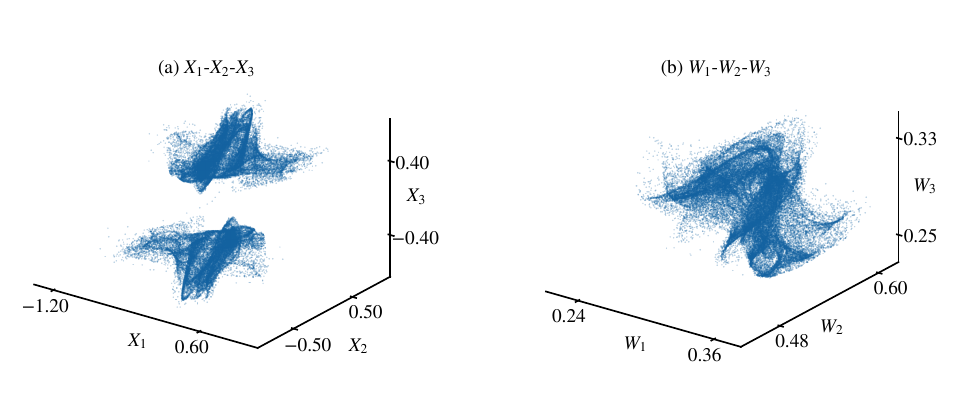}
\caption{Finite-window chaotic dynamics for the matrix~\eqref{eq:case-matrix-rounded} and parameters~\eqref{eq:canonical-parameters} at $h=1.06\times10^{-6}$. The panels show $X_1$--$X_2$--$X_3$ and $W_1$--$W_2$--$W_3$ projections of iterates 250,000--300,000 of $P$ from the initial state~\eqref{eq:canonical-initial-state}.}
\label{fig:chaotic-attractor-1.06}
\end{figure}

\subsection{A circle-doubling sequence}
\label{sec:phase_locking_torus_doubling}

Near $h=1.95\times10^{-6}$, we find a circle-doubling sequence, illustrated in Figure~\ref{fig:second-locking-cascade}. Here we again call a circle component $q$-fold when its displayed projection has $q$ nearby strands. In the plotted iterates of $P$ on one component, the number of strands doubles from three to six, then to twelve and twenty-four, as we increase $h$. At slightly larger learning rates, we also observe chaotic motion with positive finite-time Lyapunov exponents, as shown in Figure~\ref{fig:rmsprop-master-bifurcation-diagram}. The observed sequence suggests a torus-doubling route to chaos, although we have not identified the individual bifurcations.

\begin{figure}[!htbp]
    \centering
    \includegraphics[width=0.95\linewidth]{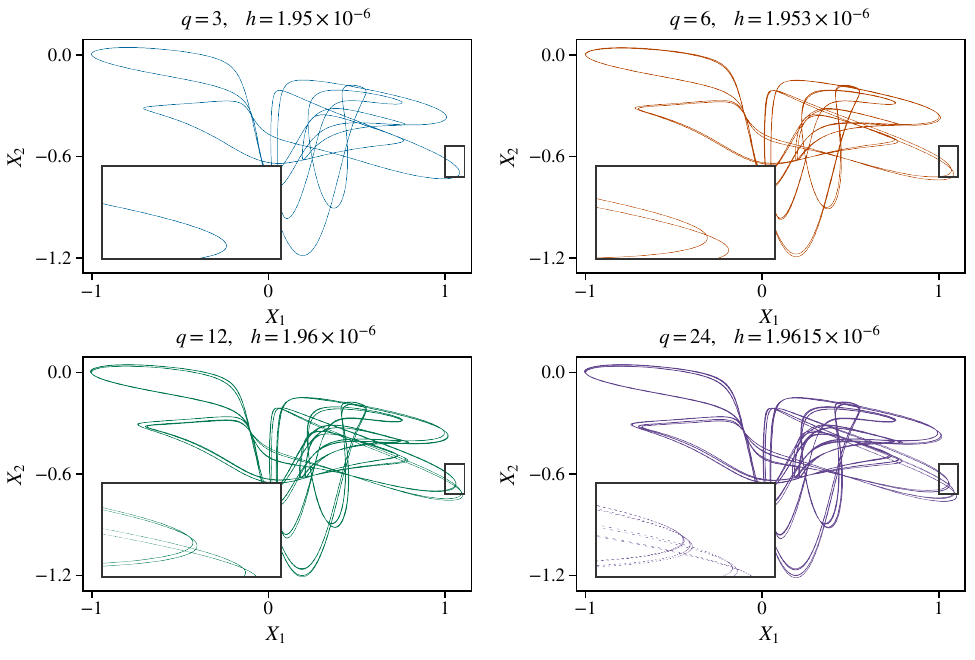}
    \caption{Circle doubling for the matrix~\eqref{eq:case-matrix-rounded} and parameters~\eqref{eq:canonical-parameters}, shown in the $X_1$--$X_2$ projection of iterates of $P=G^2$. In reading order, $(q,h)=(3,1.950\times10^{-6})$, $(6,1.953\times10^{-6})$, $(12,1.960\times10^{-6})$, and $(24,1.9615\times10^{-6})$. For each trajectory, we start from~\eqref{eq:canonical-initial-state}, discard 500,000 iterations of $P$, and plot the next 262,144. Insets show the splitting into additional strands.}
    \label{fig:second-locking-cascade}
\end{figure}
\FloatBarrier

\section{Survey of parameters}\label{sec:survey}

We now vary the Hessian orientation, memory decay, and dimension to examine their effects on the dynamics. We then test a simple learning-rate reduction schedule.

\subsection{Dependence on the Hessian orientation}\label{sec:angle-survey}

We use a $6^\circ$ ZYX Euler-angle grid of $111{,}600$ triples with fixed spectrum $(0.5,0.8,1.18)$, $\beta=0.98$, $\epsilon=10^{-7}$, and $h=3\times10^{-6}$, following the protocol in Section~\ref{sec:survey-classification}. Section~\ref{sec:survey-measurements} gives the angular ranges. The Euler-angle representation of $SO(3)$ is redundant, and a uniform grid in these coordinates does not sample rotations uniformly~\cite{lavalle2006planning}.

Using the finite-window criteria in~\eqref{eq:survey-classification}, we classify $103{,}248$ trajectories as period two, $6{,}043$ as torus motion, and $2{,}309$ as chaotic. The non-periodic cases in Figure~\ref{fig:euler-sweep-classification} form narrow bands. Thus, even with the eigenvalues fixed, changing the Hessian orientation changes the observed motion. 

\begin{figure}[htbp]
    \centering
    \includegraphics[width=0.98\linewidth]{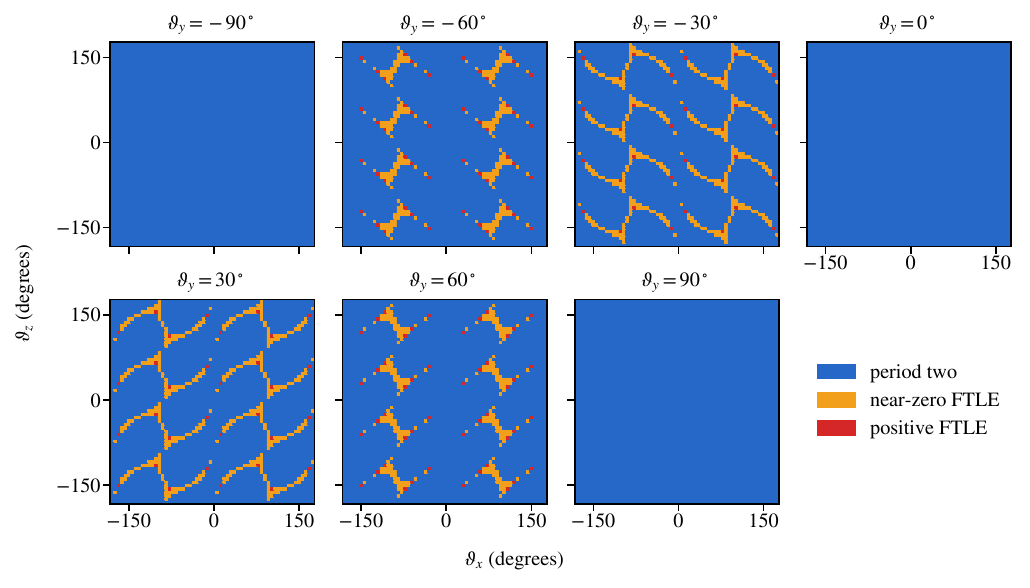}
    \caption{We find different types of motion as we change the Hessian orientation while keeping its spectrum fixed. We use a uniform ZYX Euler-angle grid of $111{,}600$ triples at $h=3\times10^{-6}$, with $\beta$ and $\epsilon$ given in~\eqref{eq:canonical-parameters}. The panels show slices at $\vartheta_y\in\{-90^\circ, -60^\circ, -30^\circ, 0^\circ, 30^\circ, 60^\circ, 90^\circ \}$. We use the measurement protocol and classification criteria in Section~\ref{sec:survey-classification} for each trajectory. Blue, orange, and red denote the finite-window period-two, torus, and chaotic classifications, respectively.}
    \label{fig:euler-sweep-classification}
\end{figure}

\subsection{Dependence on memory decay \texorpdfstring{$\beta$}{beta}}\label{sec:beta-survey}

We vary $h$ and $\beta$ with the fixed matrix $A$, using the $196\times161$ grid in Figure~\ref{fig:rmsprop-h-beta-diagram} and the protocol of Section~\ref{sec:survey-classification}. Using these criteria, we classify 806 trajectories as approaching the origin, 28,089 as period two, 1,833 as torus motion, and 825 as chaotic. We examine the remaining three samples further because they fit none of the criteria. In longer computations, we find periodic motion with periods 108, 214, and 1,814 iterates of $G$, and we refine these orbits by multiple shooting, as described in Section~\ref{sec:higher-period-motion}. The survey shows a period-two region and bands of invariant circles, tori, and chaotic dynamics for $\beta$ close to one.

\begin{figure}[htbp]
\centering
\includegraphics[width=\linewidth]{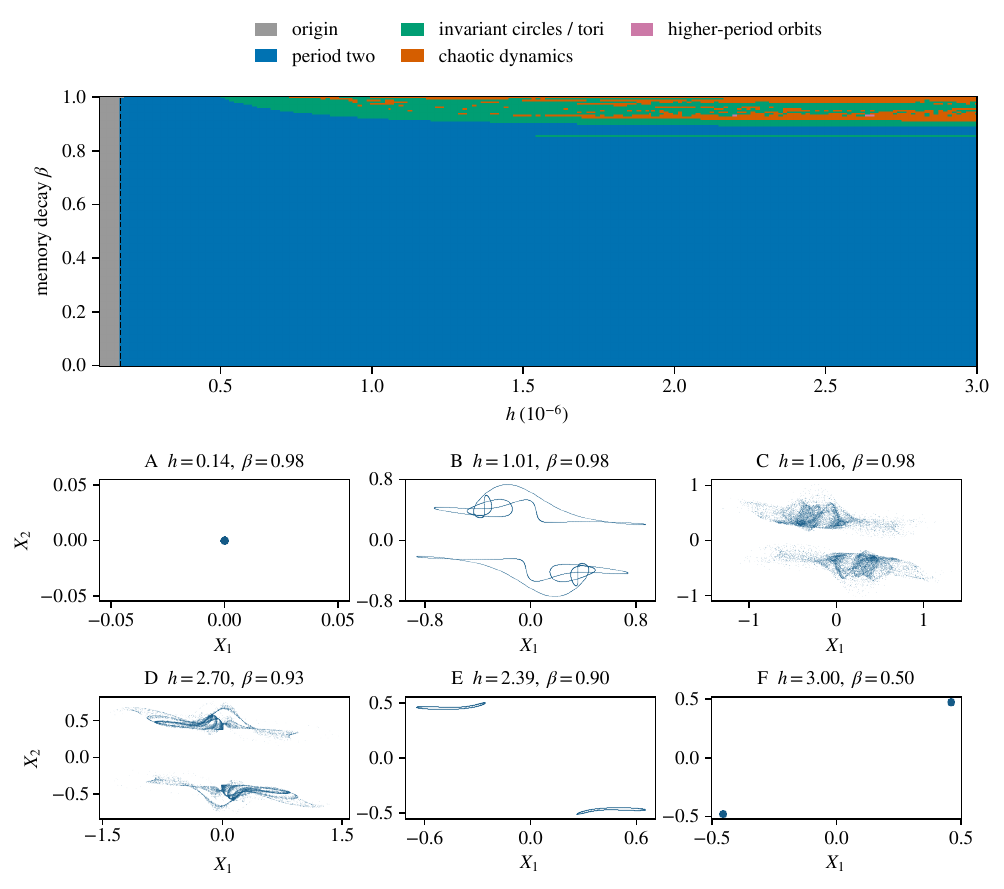}
\caption{The $(h,\beta)$ parameter diagram for epsilon-outside RMSProp with $A$ from~\eqref{eq:case-matrix-rounded} and $\epsilon$ from~\eqref{eq:canonical-parameters}. The grid is specified in Section~\ref{sec:survey-measurements}. The dashed line is $h=h_*$.  Panels A--F show 20,000 states after 50,000 discarded iterates of $G$ at selected parameters. Magenta marks the three samples manually classified as higher-period motion.}
\label{fig:rmsprop-h-beta-diagram}
\end{figure}

\subsection{Dimension}
\label{sec:higher-dimensional}

To examine how the dynamics change with dimension, we consider examples in dimensions one, two, four, and five.

\subsubsection{Examples in dimensions one and two}
\label{sec:low-dimensional-bifurcation-clouds}

In the one-dimensional computations, we find convergence to the origin below $h_*$ and to a stable period-two orbit above $h_*$.

For $d=2$, we use $ A_2=Q(0.84975)\operatorname{diag}(0.5,0.8)Q(0.84975)^\top \text{ and } \beta=0.5, $ that is,
\[
A_2\approx
\begin{bmatrix}
0.669252296871&-0.148759366311\\
-0.148759366311&0.630747703129
\end{bmatrix},\qquad Q(\theta)=
\begin{bmatrix}
\cos\theta&-\sin\theta\\
\sin\theta&\cos\theta
\end{bmatrix}.
\]
Here $Q(\theta)$ is the planar rotation matrix, with $\theta$ measured in radians.

Equations~\eqref{eq:h_star} and~\eqref{eq:normalized-learning-rate}, with $\lambda_{\max}=0.8$ and $\epsilon=10^{-7}$, give
\[
h_*=2.5\times10^{-7}.
\]
The sweep contains $2{,}259$ uniformly spaced values of $\rho\in[0.2,30]$, equivalently $h\in[0.05,7.5]\times10^{-6}$, with $\Delta\rho\approx0.01320$.  At each parameter value, we start a new trajectory in the normalized coordinates, with
\[
X_0(h)=\frac{\epsilon}{h}A_2^{-1}b_2,\qquad
W_0(h)=\frac{\epsilon}{h}a_2,
\]
where $b_2=(0.3,0.4)^\top$, and $a_2=(0.2,0.2)^\top$.

In the two-dimensional example, we find higher-period cycles and coexistence of different types of motion. At $\rho=15$, Figure~\ref{fig:low-dimensional-bifurcation-clouds} shows a stable period-two orbit together with a trajectory whose geometry and positive finite-time Lyapunov exponent suggest chaotic motion.

We continue the period-two branch emerging from the origin over $\rho\in[11.2,20]$ and find that it remains numerically attracting. At selected parameter values, a search over 96 initial conditions finds chaotic motion coexisting with this period-two orbit.

\begin{figure}[htbp]
\centering
\includegraphics[width=\linewidth]{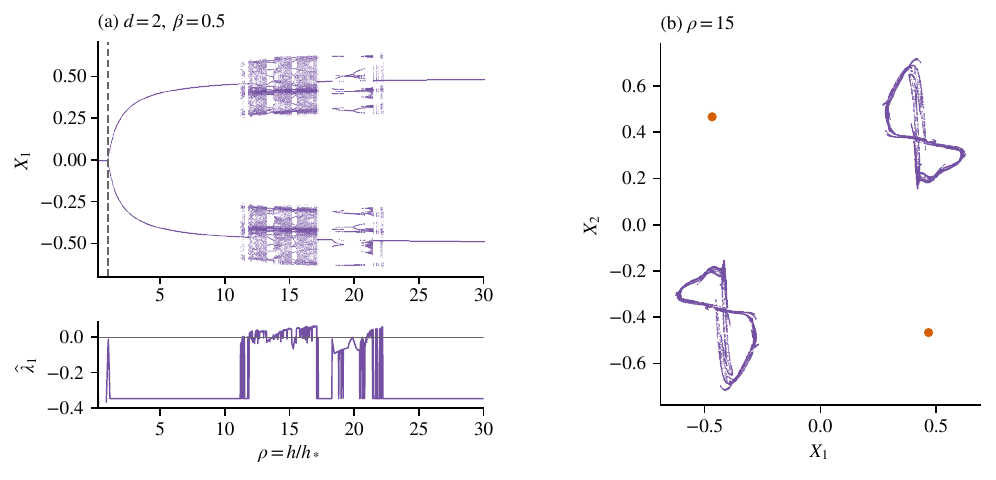}
\caption{Epsilon-outside RMSProp in dimension two with $A_2$ as above, $\beta=0.5$, and $\epsilon=10^{-7}$. (a) Bifurcation diagram from separately initialized trajectories at 2,259 values of $\rho\in[0.2,30]$. We plot the 160 post-transient $X_1$ values at each parameter value, and the strip gives the leading FTLE per iterate of $G$. (b) Coexisting chaotic attractor (purple) and stable period-two orbit (orange) at $\rho=15$, or $h=3.75\times10^{-6}$. The purple trajectory starts from $X_0=A_2^{-1}(0.3,0.4)^\top/37.5$, $W_0=(0.2,0.2)^\top/37.5$; the orange trajectory starts from $X_0=(0.4,-0.4)^\top$, $W_0=(0.3,0.3)^\top$. Both show 20,000 consecutive states after 500,000 discarded iterates of $G$; the orange iterates alternate between the two points $P_1$ and $P_2$ of the period-two orbit.}
\label{fig:low-dimensional-bifurcation-clouds}
\end{figure}

\subsubsection{Coupled extensions in dimensions four and five}\label{sec:coupled-extensions}
\label{sec:coupled-higher-dimensional}

We test whether the three-dimensional behavior persists when we add one or two coordinates and couple them to the original variables through
\[ A_d(\eta)=\begin{bmatrix}A&\eta B_d\\ \eta B_d^\top&D_d\end{bmatrix},\qquad \eta\geq0. \]
At $\eta=0$, the block-diagonal matrix $A\oplus D_d$ recovers the three-dimensional map on the invariant coordinate subspace $X_4=\cdots=X_d=W_4=\cdots=W_d=0$.
For $d=4$, we take
\[ D_4=[0.15],\qquad B_4=\frac{1}{\sqrt{0.40^2+0.70^2+0.55^2}}\begin{bmatrix}0.40\\-0.70\\0.55\end{bmatrix}. \]
For $d=5$, we take $D_5=\operatorname{diag}(0.15,0.12)$ and normalize the two columns separately:
\[ B_5=\begin{bmatrix}0.50/\nu_1&-0.20/\nu_2\\-0.35/\nu_1&0.75/\nu_2\\0.60/\nu_1&0.45/\nu_2\end{bmatrix},\qquad \nu_1=\sqrt{0.50^2+0.35^2+0.60^2},\quad \nu_2=\sqrt{0.20^2+0.75^2+0.45^2}. \]

We use $\eta=0.10$ in both dimensions, for which $A_d(0.10)$ is positive definite.

To obtain the bifurcation diagrams in Figure~\ref{fig:higher-dimensional-bifurcations}, we initialize a trajectory at each learning rate using the first three spatial and memory coordinates from~\eqref{eq:canonical-initial-state} and perturbations in the additional coordinates:
\[ X_\perp^{(0)}=10^{-5}\quad(d=4),\qquad X_\perp^{(0)}=10^{-5}(1,-1/2)\quad(d=5),\qquad W_\perp^{(0)}=10^{-5}\mathbf1. \]

With this coupling, the trajectories at $h=1.06\times10^{-6}$ have positive leading finite-time Lyapunov exponents, consistent with chaotic motion.

\begin{figure}[htbp]
    \centering
    \includegraphics[width=\linewidth]{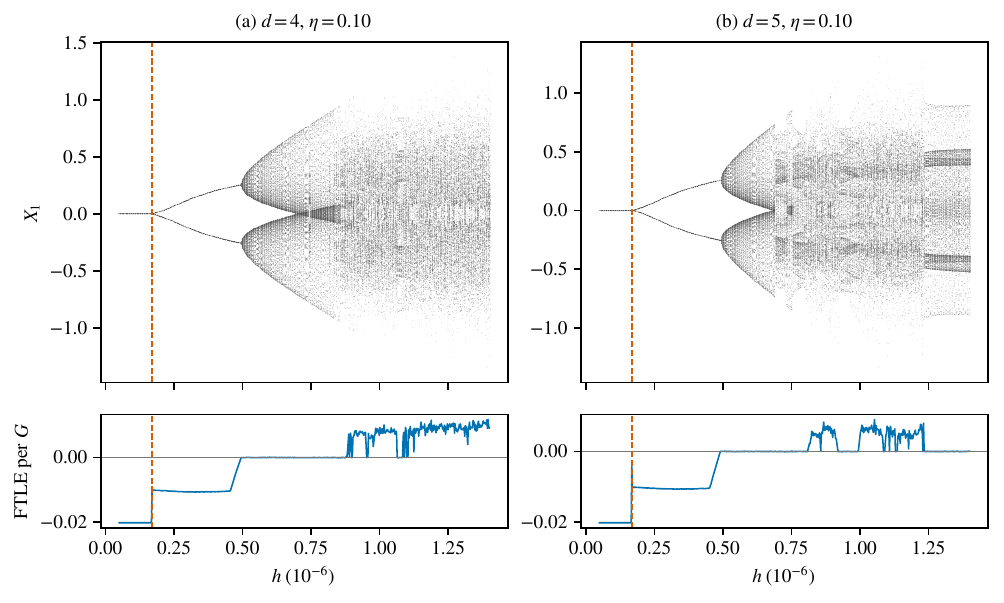}
    \caption{Bifurcation diagrams for $d=4$ and $d=5$ extensions. Both panels use parameters $\eta=0.10$, $\beta=0.98$, and $\epsilon=10^{-7}$. Dashed lines indicate the stability thresholds $2\epsilon/\lambda_{\max}(A_d)$. The lower strips provide the leading FTLE per iterate of $G$ computed over 20,000 consecutive iterations after discarding an initial transient.}
    \label{fig:higher-dimensional-bifurcations}
\end{figure}

\subsection{Adaptive learning rates}\label{sec:adaptive-learning-rate}

In practice, we want to reduce the loss while using as few iterates as possible. For fixed $h\ge h_*$, we prove an eventual spatial bound proportional to $h$ in Theorem~\ref{thm:rmsprop-trapping-bound}. The observed post-transient radii in Figure~\ref{fig:high-h-scaling} agree with this scaling in the original coordinates.

We therefore try beginning with a relatively large learning rate and decreasing it when the average loss stops decreasing. Figure~\ref{fig:learning-rate-loss} illustrates the initial descent and subsequent oscillating loss at fixed $h=0.1$.

\begin{figure}[htbp]
    \centering
    \includegraphics[width=0.95\linewidth]{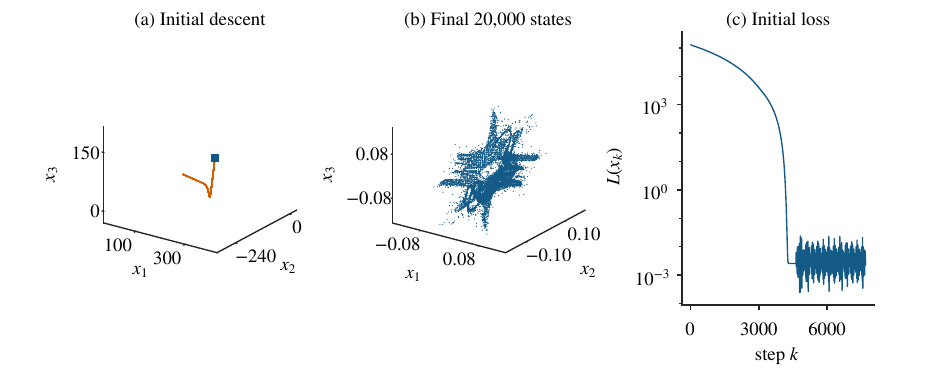}
    \caption{With the learning rate fixed at $h=0.1$, we find that the loss initially decreases and then fluctuates around a positive level. We use epsilon-outside RMSProp with the matrix~\eqref{eq:case-matrix-rounded} and parameters~\eqref{eq:canonical-parameters}, starting from $x_0=(400,-300,200)$ and $v_0=0$. The three panels show the first 6,000 states in the $x$ spatial coordinates, the final $20{,}000$ states, and the loss through $k=7700$, respectively.}
    \label{fig:learning-rate-loss}
\end{figure}

To decide when to reduce $h$, we compare the average of the latest 300 losses with the average of the preceding 300 losses:
\begin{equation*}
\overline L_k=\frac{1}{300}\sum_{j=0}^{299}L(x_{k-j}),
\qquad
|\overline L_k-\overline L_{k-300}|\le0.10\max\{\overline L_k,\overline L_{k-300}\}.
\end{equation*}
After each update, we shift both averaging windows of 300 iterates forward by one iterate. When the inequality holds for 300 consecutive updates, we divide $h$ by ten. After each reduction, we retain $(x_k,v_k)$, reset the count of consecutive updates satisfying the inequality, and collect 600 new losses before checking again. We start at $h=0.1$ and stop reducing at $h=10^{-5}$.

For the parameters and initial state of Figure~\ref{fig:learning-rate-loss}, the reductions occur at $k=5474,6764,8048,9326$, after which we run another 3,000 iterations at $h=10^{-5}$. In Figure~\ref{fig:adaptive-learning-rate-schedule}, the average loss decreases after each reduction, although the loss continues to oscillate.

\begin{figure}[htbp]
\centering
\includegraphics[width=\linewidth]{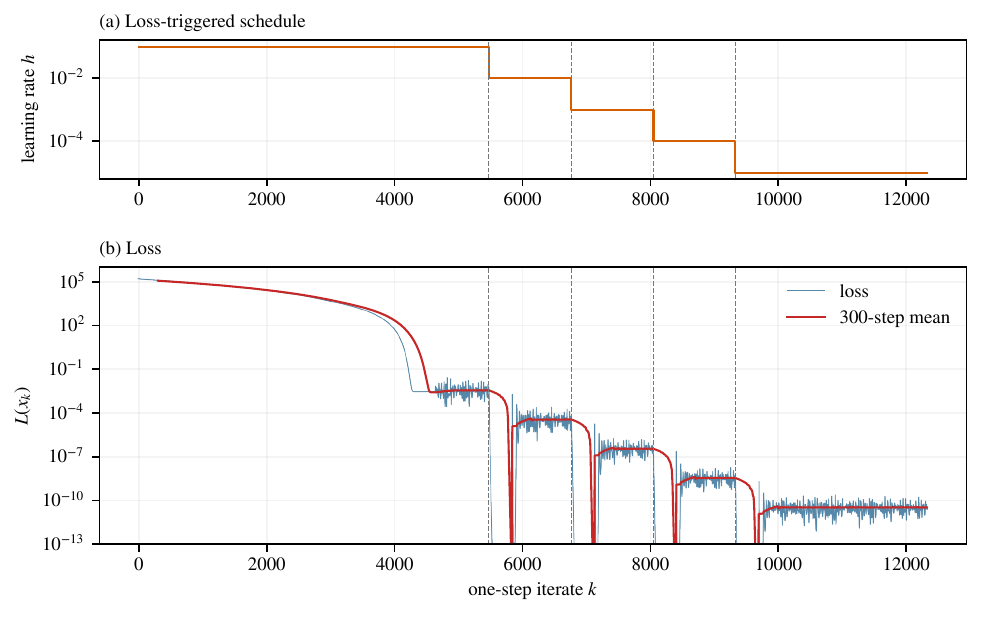}
\caption{Adaptive learning-rate reduction. (a) Learning rate versus iterate, where dashed vertical lines mark reductions in learning rate $h$. (b) Loss at every iterate (blue) and the average of the latest 300 loss values (red), plotted at every iterate $k\ge299$. Here the matrix and parameters are given in~\eqref{eq:case-matrix-rounded} and~\eqref{eq:canonical-parameters}. }
\label{fig:adaptive-learning-rate-schedule}
\end{figure}

\section{Autonomous Adam}\label{sec:Adam}

Adam was introduced by Kingma and Ba~\cite{kingma2015adam}. We write their step size as $h$, initialize $m_0=v_0=0$, and let $x_0$ be the initial spatial vector. At iteration $k\geq1$, the original (non-autonomous) Adam algorithm is:
\begin{equation*}\begin{aligned} g_k&=\nabla f_k(x_{k-1}), & m_k&=\beta_1m_{k-1}+(1-\beta_1)g_k, & v_k&=\beta_2v_{k-1}+(1-\beta_2)g_k^{\odot2},\\ \widehat m_k&=\frac{m_k}{1-\beta_1^k}, & \widehat v_k&=\frac{v_k}{1-\beta_2^k}, & x_k&=x_{k-1}-h\,\widehat m_k\oslash\left(\sqrt[\odot]{\widehat v_k}+\epsilon\mathbf{1}\right). \end{aligned}\end{equation*}
The denominator convention is epsilon outside, as in Kingma--Ba Algorithm~1 and the cited PyTorch and Keras implementations \cite{kingma2015adam,pytorchAdamAPI2026,kerasAdamSource2026}. Here $f_k$ denotes the stochastic objective sampled at iteration $k$. In the deterministic quadratic setting considered here, we set $f_k=L$ for every $k$, with $L$ given by~\eqref{eq:quadratic}, so that $g_k=Ax_{k-1}$. The ``bias-correction'' factors $(1-\beta_1^k)^{-1}$ and $(1-\beta_2^k)^{-1}$ make the finite-time iteration dependent on $k$ and hence non-autonomous.

For $0\leq\beta_1,\beta_2<1$, the bias-correction factors approach one as the iteration number increases. We therefore study the autonomous map obtained by replacing them with their limiting values. In unscaled variables, this map is given below, with $F_V$ from~\eqref{eq:memory} evaluated at $\beta=\beta_2$.
\begin{equation}
m_{k+1}=\beta_1m_k+(1-\beta_1)Ax_k,\quad
v_{k+1}=F_V(x_k,v_k),\quad
x_{k+1}=x_k-h\,m_{k+1}\oslash\left(\sqrt[\odot]{v_{k+1}}+\epsilon\mathbf{1}\right).
\label{eq:adam-autonomous-unscaled}
\end{equation}

Using $X,W$ from~\eqref{eq:rmsprop-normalization} and scaling the momentum as $M=m/h$, we obtain the map $G_{\beta_1}$ below. Its memory component $G_2$ is given by~\eqref{eq:normalized-rmsprop-drift-lemma} with $\beta=\beta_2$.
\begin{equation}\label{eq:adam-autonomous-scaled}
\begin{aligned}
M_{k+1}&=\beta_1M_k+(1-\beta_1)AX_k,\qquad W_{k+1}=G_2(X_k,W_k),\\
X_{k+1}&=X_k-M_{k+1}\oslash\left(W_{k+1}+c\mathbf{1}\right).
\end{aligned}
\end{equation}

\subsection{Continuation of invariant circles into Adam}
\label{sec:adam-continuation}

At $\beta_1=0$, eliminating $M$ from~\eqref{eq:adam-autonomous-scaled} gives~\eqref{eq:normalized-rmsprop-drift-lemma} with $\beta=\beta_2$. To initialize continuation in $\beta_1$, we add momentum coordinates to an RMSProp period-two circle $K_j=(X_j,W_j)$.

The momentum relation in~\eqref{eq:adam-autonomous-scaled} at $\beta_1=0$, shifted back one iterate, gives the gradient at the preceding state. By~\eqref{eq:invariant-circle-pair}, the predecessor of $K_j(\theta)$ is $K_{3-j}(\theta-\omega)$.
This gives the momentum coordinates for the two circle components,
\[
\widetilde K_1(\theta)=\bigl(X_1(\theta),AX_2(\theta-\omega),W_1(\theta)\bigr),\qquad
\widetilde K_2(\theta)=\bigl(X_2(\theta),AX_1(\theta-\omega),W_2(\theta)\bigr).
\]
We write $G_0$ for $G_{\beta_1}$ at $\beta_1=0$. With this notation, one iteration of Adam gives
\[
G_0\bigl(\widetilde K_j(\theta)\bigr)
=\bigl(X_{3-j}(\theta+\omega),AX_j(\theta),W_{3-j}(\theta+\omega)\bigr)=\widetilde K_{3-j}(\theta+\omega).
\]

Using these parameterizations as initial approximations at $h=7.80889203\times10^{-7}$, we continue the RMSProp period-two circle to autonomous Adam at $\beta_1=0.01,0.02$. Figure~\ref{fig:adam-small-circle} compares these continuations with the RMSProp case $\beta_1=0$.

\begingroup
\captionsetup{skip=5pt}
\setlength{\textfloatsep}{6pt plus 2pt minus 2pt}
\setlength{\floatsep}{8pt plus 2pt minus 2pt}
\begin{figure}[htbp]
\centering
\includegraphics[width=\linewidth]{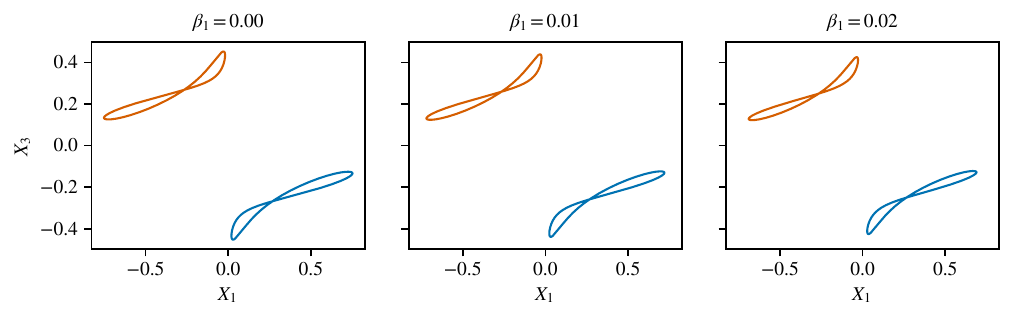}
\caption{An RMSProp period-two circle at $h=7.80889203\times10^{-7}$ and its autonomous-Adam continuations at $\beta_1=0.01,0.02$, with the matrix~\eqref{eq:case-matrix-rounded} and parameters~\eqref{eq:canonical-parameters}, using $\beta_2=\beta$.}
\label{fig:adam-small-circle}
\end{figure}

We then follow the period-two circle as $\beta_1$ increases. Figure~\ref{fig:adam-continued-circles-visible} shows it shrinking toward the period-two orbit near $\beta_1\approx0.21550$. Along the period-two branch, we compute a conjugate pair of Floquet multipliers crossing the unit circle near the same value, suggesting a Neimark--Sacker bifurcation.

\begin{figure}[!htbp]
    \centering
    \includegraphics[width=0.95\linewidth]{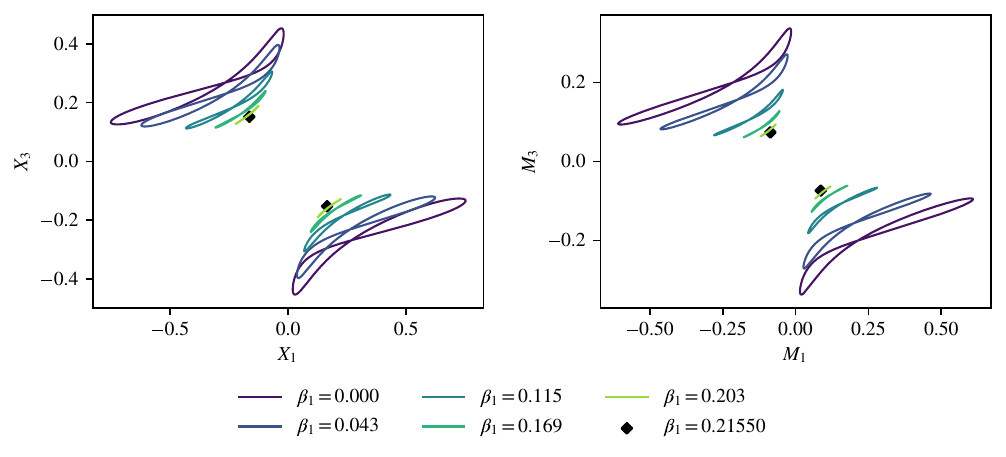}
    \caption{Autonomous-Adam circle continuation at $h=7.8089\times10^{-7}$. The left and right panels show the scaled $X_1$--$X_3$ and $M_1$--$M_3$ projections. Colors encode $\beta_1$, and black diamonds mark points on the last computed period-two circle near $\beta_1=0.21550$.}
    \label{fig:adam-continued-circles-visible}
\end{figure}

At the larger learning rate $h=1.92\times10^{-6}$, we independently initialize trajectories at $\beta_1=0,0.01,0.02$. Their leading finite-time Lyapunov exponents are positive, and Figure~\ref{fig:adam-small-chaos} shows their geometry.

\begin{figure}[htbp]
\centering
\includegraphics[width=\linewidth]{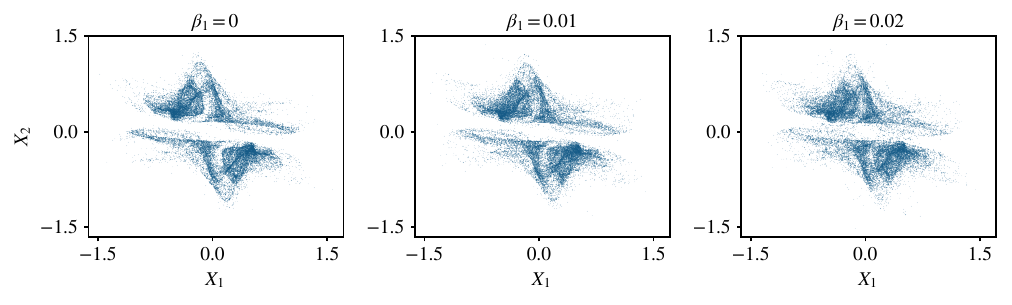}
\caption{Autonomous-Adam trajectories at $h=1.92\times10^{-6}$ and $\beta_1=0,0.01,0.02$, using the parameters of Figure~\ref{fig:adam-small-circle}. Trajectory settings are given in Appendix~\ref{app:adam-gallery}.}
\label{fig:adam-small-chaos}
\end{figure}

We vary $h$ and $\beta_1$ on a $195\times181$ grid, using the protocol of Section~\ref{sec:survey-classification} and the parameter ranges in Section~\ref{sec:survey-measurements}.

We classify $6{,}868$ samples as approaching the origin, $2{,}030$ as period two, $7{,}768$ as invariant circles or tori, and $18{,}629$ as chaotic (Figure~\ref{fig:adam-h-beta1-diagram}). These are finite-window classifications, subject to the limitations in Section~\ref{sec:survey-classification}. The dashed curve gives the stability boundary of the origin for epsilon-outside Adam without bias correction~\cite[Corollary III.5]{bock2021local},
\[
h=\frac{2\epsilon(1+\beta_1)}{(1-\beta_1)\lambda_{\max}(A)}.
\]
Appendix~\ref{app:adam-gallery} provides additional trajectories, circle projections, and the trajectory settings for panels E and F of Figure~\ref{fig:adam-h-beta1-diagram}.

\begin{figure}[!htbp]
\centering
\includegraphics[width=0.90\linewidth]{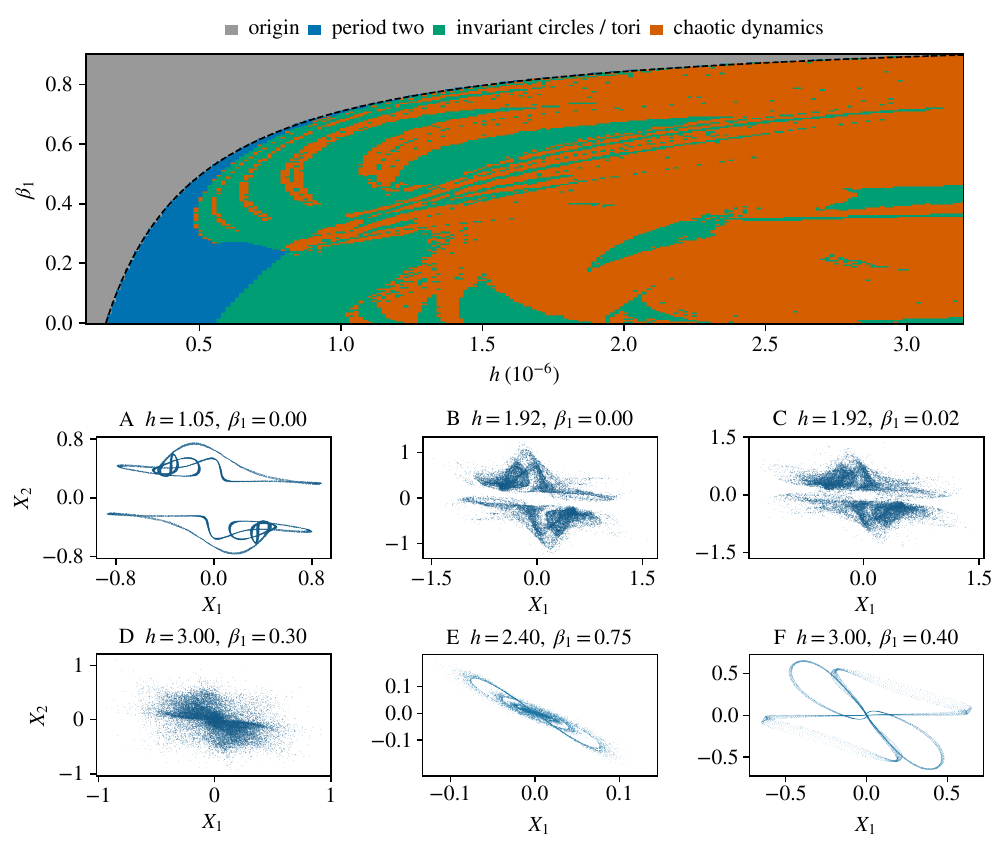}
    \caption{Autonomous Adam in the $(h,\beta_1)$ plane, with the matrix and other parameters of Figure~\ref{fig:adam-small-circle}. The dashed curve is the stability threshold of the origin. Panels A--F show $X_1$--$X_2$ projections at $(h/10^{-6},\beta_1)=(1.05,0),(1.92,0),(1.92,0.02),(3.00,0.30),(2.40,0.75),(3.00,0.40)$.}
    \label{fig:adam-h-beta1-diagram}
\end{figure}
\endgroup

\FloatBarrier

\section{Conclusion}

We begin with the question of what happens to RMSProp when the learning rate is too large for convergence to the minimizer. For positive-definite quadratic losses, we show that the origin is Lyapunov unstable whenever $h>h_*=2\epsilon/\lambda_{\max}(A)$. But the iterates do not simply escape to infinity. Regardless of the initial condition, they eventually enter a bounded region whose size we estimate explicitly. The bounds are proportional to $h$ in the spatial variables and to $h^2$ in the memory variables. Thus, even when RMSProp need not minimize the loss, we obtain explicit bounds on how close its iterates eventually remain to the minimizer.

The next question is what actually happens inside this region. We establish that period-two orbits exist precisely above the critical learning rate. Under additional spectral assumptions, we obtain an attracting branch emerging from the origin near $h_*$. At the other extreme, we find period-two orbits near admissible corners for sufficiently large $h$. These give two ways to begin looking for periodic motion, one near the stability threshold and the other at large learning rates. We use numerical continuation to explore how the resulting branches relate to one another.

Our numerical experiments show that period-two motion is only a small part of the picture. As we increase the learning rate, we find invariant circles, invariant tori, coexistence of attracting sets, and circle-doubling sequences. We also find trajectories with positive finite-time Lyapunov exponents, consistent with chaotic motion. Even for a quadratic loss with a unique minimizer, there is therefore quite a bit of dynamics to understand. Our parameter surveys show that the observed behavior depends on Hessian orientation, memory decay, and dimension. We also continue selected invariant circles into autonomous Adam, beginning from the RMSProp case $\beta_1=0$.

A more systematic study of the parameter space, with broader hyperparameter searches in arbitrary dimensions, could examine how often multistability and chaotic motion occur beyond the selected examples. The work of Mischaikow and collaborators on DSGRN~\cite{cummins2016combinatorial} suggests one possible starting point, since their framework describes recurrent sets and partitions parameter space according to qualitative dynamics.

Another direction is to study the numerically observed invariant circles and tori, their bifurcations, and the chaotic regimes using rigorous methods. Computer-assisted methods such as those developed by Tucker~\cite{tucker2002rigorous} and Wilczak~\cite{wilczak2010uniformly} provide examples of rigorous proofs of chaotic attractors and suggest approaches to selected examples in RMSProp.

Finally, the results raise questions about the optimization problem that motivates this study. Our learning-rate reduction example shows how decreasing $h$ moves a trajectory with persistent fluctuations toward the minimizer. Further work could examine how to use the bounds and the observed dynamics to choose when to reduce the learning rate. Other questions concern extending the instability and trapping arguments beyond quadratic losses and understanding transitions between neighborhoods of different minima.

\section*{Acknowledgments}

We used OpenAI Codex and Anthropic Claude to assist with mathematical derivations, developing and reviewing numerical code, preparing figures, and editing the manuscript. We also used Google Gemini to review manuscript wording and formatting. The authors assume responsibility for all content.

\appendix
\section{Fourier parameterization and Newton's method}
\label{app:fourier_newton}

For a period-two circle, we solve the coupled equations~\eqref{eq:invariant-circle-pair} for $G$ and study the dynamics and normal stability of each component using $P=G^2$. We represent each parameterization on the equally spaced angular grid $\theta_\ell=2\pi\ell/M$, $\ell=0,\ldots,M-1$, using its discrete Fourier modes. With either an even or an odd number of nodes, this gives
\begin{equation*}
K_j^{(M)}(\theta)=\operatorname{Re}\sum_{k\in I_M}\widehat K_{j,k}e^{\mathrm i k\theta},
\qquad I_M=\{-\lfloor M/2\rfloor,\ldots,\lceil M/2\rceil-1\}.
\end{equation*}
The coefficients are the discrete Fourier transform divided by $M$. Additionally, we impose a phase condition to remove the freedom to shift the angular variable:
\begin{equation*}
\Phi(K_1)=\langle K_1(0)-v_{\mathrm{ref}},\eta_{\mathrm{ref}}\rangle=0.
\end{equation*}
Here $v_{\mathrm{ref}}$ and $\eta_{\mathrm{ref}}$ are a chosen reference point and reference tangent vector, respectively, and the brackets denote the Euclidean inner product. This condition fixes the angular origin provided $\langle K_1'(0),\eta_{\mathrm{ref}}\rangle\ne0$ at the solution.

We initialize Newton's method from a trajectory or a previously computed parameterization. When starting from a trajectory, we estimate the rotation number and Fourier coefficients using weighted Birkhoff averages, following the references in Section~\ref{sec:parametrization_methods_stability}.

With these initial approximations, we truncate the Fourier series and apply Newton's method to~\eqref{eq:invariant-circle-pair} and the phase condition, solving for $\omega$ and the real values of $K_1$ and $K_2$ at the $M$ grid nodes. At each iteration, we use the Fourier representation to evaluate angular shifts and derivatives.

For a two-torus, we similarly use a two-angle Fourier parameterization $K:\mathbb T^2\to\mathbb R^6$ and solve
\begin{equation}
G^2(K(\theta_1,\theta_2))-K(\theta_1+\Omega_1,\theta_2+\Omega_2)=0,
\qquad (\theta_1,\theta_2)\in\mathbb T^2,
\label{eq:two-torus-invariance}
\end{equation}
with unknown rotation numbers $\rho_{P,j}=\Omega_j/(2\pi)\pmod1$, $j=1,2$, measured in turns per iteration of $P=G^2$. Here $\Omega_j$ is the angular advance per iteration of $P$, whereas $\omega$ for a circle is the advance per iterate of $G$. To fix the origins of both torus angles, we impose the two phase conditions
\begin{equation}
\Phi_j(K)=\frac{1}{(2\pi)^2}\int_{\mathbb T^2}
\left\langle K(\theta)-K_{\mathrm{ref}}(\theta),
\partial_{\theta_j}K_{\mathrm{ref}}(\theta)\right\rangle\,d\theta=0,
\qquad j=1,2,
\label{eq:two-torus-phase-conditions}
\end{equation}
where $K_{\mathrm{ref}}$ is a reference parameterization, $\theta=(\theta_1,\theta_2)$, and $\partial_{\theta_j}K_{\mathrm{ref}}$ is its derivative with respect to the $j$th angle. These phase conditions remove the two angular shifts when their derivatives with respect to those shifts form an invertible $2\times2$ matrix. We then apply Newton's method to~\eqref{eq:two-torus-invariance} and~\eqref{eq:two-torus-phase-conditions}, solving for the real values of $K$ at the grid nodes and the angular increments $\Omega_1,\Omega_2$. Once $K$ is computed, applying $G$ to $K(\mathbb T^2)$ gives the second component.

\subsection{Normal Lyapunov exponents}
\label{sec:normal-lyapunov-method}
We measure growth perpendicular to a circle or torus by removing the tangent directions at each iteration of $P=G^2$~\cite{jorba2001normal}. Let $K$ parameterize one component of a period-two circle or a torus pair. An orbit on this component satisfies
\[
z_j=K(\theta^{(j)}),\qquad z_{j+1}=P(z_j).
\]
Here $\theta^{(0)}$ is the initial angle for a circle or the pair of initial angles for a two-torus. The angular state $\theta^{(j)}$ advances by $2\omega$ for a circle and by $(\Omega_1,\Omega_2)$ for a two-torus at each iteration, with all angles taken modulo $2\pi$. The tangent vectors at $z_j$ are $K'(\theta^{(j)})$ for a circle and $\partial_{\theta_1}K(\theta^{(j)}),\partial_{\theta_2}K(\theta^{(j)})$ for a two-torus. We let the columns of $Q_j$ form an orthonormal basis for the space perpendicular to these vectors, so in $\mathbb R^6$ the matrix $Q_j$ has five columns for a circle and four for a two-torus. Applying $DP(z_j)$ and projecting onto the normal space at $z_{j+1}$ gives
\begin{equation}\label{eq:normal-quotient-cocycle}
B_j=Q_{j+1}^{\top}DP(z_j)Q_j.
\end{equation}
\begin{definition}[Normal Lyapunov exponents]\label{def:normal-lyapunov-exponents}
When the limits exist, we define the \textit{normal Lyapunov exponents} along the orbit by
\[
\lambda_{\perp,i}(z_0)=\lim_{N\to\infty}\frac{1}{N}
\log\sigma_i\!\left(B_{N-1}\cdots B_0\right).
\]
Here $\sigma_i$ denotes the $i$th largest singular value, and the normal Lyapunov exponents measure growth per iteration of $P$.
\end{definition}

We evaluate $K$ at successive angles to keep the calculation on the computed circle or torus, including when it is unstable. Following the orthonormalization method of Benettin et al.~\cite{benettin1980part2}, we use the QR iteration
\[
B_jU_j=U_{j+1}R_j,\qquad U_j^{\top}U_j=I,\qquad (R_j)_{ii}>0,
\]
where $R_j$ is upper triangular.

After discarding initial iterations to reduce dependence on the starting perturbation directions, we estimate the normal Lyapunov exponents over the next $N$ consecutive iterations of $P$:
\[
\begin{aligned}
\widehat\lambda_{\perp,i}^{\mathrm{unweighted}}
 &=\frac{1}{N}\sum_{j=0}^{N-1}\log(R_j)_{ii},\\
\widehat\lambda_{\perp,i}^{\mathrm{weighted}}
 &=\sum_{j=0}^{N-1}w_j\log(R_j)_{ii}.
\end{aligned}
\]
For the weighted estimates, we use the smooth weighting function of Das et al.~\cite{das2017quantitative}, evaluated at the midpoints
\[
s_j=\frac{j+1/2}{N},\qquad
w_j=\frac{\exp[-1/(s_j(1-s_j))]}
{\displaystyle\sum_{k=0}^{N-1}\exp[-1/(s_k(1-s_k))]},
\qquad j=0,\ldots,N-1.
\]
Section~\ref{sec:normal-lyapunov} gives the computational settings and compares estimates obtained with different iteration counts and initial points on the circles.

\section{Proof of the general trapping theorem}
\label{app:exterior-energy-trapping-proof}

\begingroup
\theoremstyle{plain}
\newtheorem*{trappingrestatement}{Theorem~\ref{thm:exterior-energy-trapping}}
\begin{trappingrestatement}[Trapping by exterior energy decrease; restated]
Let $F:\mathbb R^d\times\mathbb R_{\ge0}^d\to\mathbb R^d\times\mathbb R_{\ge0}^d$ and $V:\mathbb R^d\times\mathbb R_{\ge0}^d\to\mathbb R$ be continuous. Suppose $V$ is radially unbounded, namely
\begin{equation*}
V(z)\longrightarrow+\infty\quad\text{as }\|z\|\longrightarrow\infty,
\qquad z\in\mathbb R^d\times\mathbb R_{\ge0}^d.
\tag{\ref{eq:trapping-radial-unboundedness}}
\end{equation*}
Let $R\subset\mathbb R^d\times\mathbb R_{\ge0}^d$ be bounded, and assume
\begin{equation*}
V(F(z))<V(z)\qquad\text{for every }z\in(\mathbb R^d\times\mathbb R_{\ge0}^d)\setminus R.
\tag{\ref{eq:trapping-exterior-decrease}}
\end{equation*}
Then there exists $C_0>0$ such that, for every $C\ge C_0$, the sublevel set $\Omega_C=\{z\in\mathbb R^d\times\mathbb R_{\ge0}^d:V(z)\le C\}$ has the following properties:
\begin{enumerate}
\renewcommand{\labelenumi}{(\roman{enumi})}
\setlength{\itemsep}{2pt}
\item \textbf{Compactness.} $\Omega_C$ is compact.
\item \textbf{Containment of the region.} $R\subset\operatorname{int}_{\mathbb R^d\times\mathbb R_{\ge0}^d}(\Omega_C)$.
\item \textbf{Strict trapping.} $F(\Omega_C)\subset\operatorname{int}_{\mathbb R^d\times\mathbb R_{\ge0}^d}(\Omega_C)$. In particular, an orbit cannot leave $\Omega_C$ after entering.
\item \textbf{Finite-time entry.} For every $z_0\in\mathbb R^d\times\mathbb R_{\ge0}^d$, there exists an integer $N=N(z_0,C)\ge0$ such that $F^n(z_0)\in\Omega_C$ for all $n\ge N$.
\end{enumerate}
The interiors in parts (ii) and (iii) are relative to $\mathbb R^d\times\mathbb R_{\ge0}^d$, so neighborhoods are taken within this state space.
\end{trappingrestatement}
\endgroup

\begin{proof}
We refer to $V$ as the energy and write $\overline R$ for the closure of $R$. Since $\mathbb R^d\times\mathbb R_{\ge0}^d$ is closed and $R$ is bounded, $\overline R$ is a compact subset of $\mathbb R^d\times\mathbb R_{\ge0}^d$. The composition $V\circ F$ is continuous because both $F$ and $V$ are continuous. If $R$ is nonempty, the extreme value theorem therefore gives finite maxima
\[
M_1=\max_{z\in\overline R}V(z),\qquad
M_2=\max_{z\in\overline R}V(F(z)).
\]
Choose any $C_0>\max\{0,M_1,M_2\}$. Then $C_0>0$ and, for every $z\in\overline R$,
\[
V(z)\le M_1<C_0,\qquad V(F(z))\le M_2<C_0.
\]
If $R$ is empty, choose any $C_0>0$. Fix $C\ge C_0$.

\emph{(i) Compactness.} By continuity of $V$, the sublevel set $\Omega_C=V^{-1}((-\infty,C])$ is closed relative to $\mathbb R^d\times\mathbb R_{\ge0}^d$. Since this state space is itself closed in $\mathbb R^{2d}$, $\Omega_C$ is closed in $\mathbb R^{2d}$. Radial unboundedness gives a radius $M>0$ such that $V(z)>C$ whenever $z\in\mathbb R^d\times\mathbb R_{\ge0}^d$ and $\|z\|>M$. Thus every $z\in\Omega_C$ satisfies $\|z\|\le M$, so $\Omega_C$ is bounded. It is therefore compact by the Heine--Borel theorem.

\emph{(ii) Containment of the bounded region.} The choice of $C_0$ gives $V(z)<C_0\le C$ for every $z\in R$, which proves $R\subset\{V<C\}$. Separately, continuity of $V$ makes $\{z\in\mathbb R^d\times\mathbb R_{\ge0}^d:V(z)<C\}$ open relative to $\mathbb R^d\times\mathbb R_{\ge0}^d$. This relatively open set is contained in $\Omega_C$, so each of its points is an interior point of $\Omega_C$ relative to the state space. Consequently,
\[
R\subset\{V<C\}\subset\operatorname{int}_{\mathbb R^d\times\mathbb R_{\ge0}^d}(\Omega_C).
\]
\emph{(iii) Strict trapping.} To prove that $F$ maps $\Omega_C$ into its relative interior, fix $z\in\Omega_C$ and distinguish two cases.

\emph{Case 1: $z\in R$.} The choice of $C_0$ gives
\[
V(F(z))<C_0\le C.
\]
No comparison between $V(F(z))$ and $V(z)$ is required here: the energy may increase inside $R$, but its value after one iterate remains below $C$.

\emph{Case 2: $z\in\Omega_C\setminus R$.} Since $z$ lies outside $R$, its energy decreases. Since $z\in\Omega_C$, we also have $V(z)\le C$. Thus
\[
V(F(z))<V(z)\le C.
\]
In both cases, $V(F(z))<C$. Therefore $F(\Omega_C)\subset\{V<C\}\subset\operatorname{int}_{\mathbb R^d\times\mathbb R_{\ge0}^d}(\Omega_C)$.

\emph{(iv) Finite-time entry.} Suppose an orbit $z_n=F^n(z_0)$ never enters $\Omega_C$. Its energy decreases, so it remains in the compact set $K=\{z\in \mathbb R^d\times\mathbb R_{\ge0}^d:C\le V(z)\le V(z_0)\}$. This set is disjoint from $R$. The continuous function $V-V\circ F$ is strictly positive on $K$ and therefore has a positive minimum $\delta$. Hence $V(z_n)\le V(z_0)-n\delta$. For sufficiently large $n$, this gives $V(z_n)<C$, contrary to our assumption. Every orbit therefore enters $\Omega_C$, and strict trapping ensures that it remains there.
\end{proof}

\begingroup
\section{Additional numerical methods}\label{app:additional-numerical-methods}
\subsection{Matrix used in the three-dimensional examples}
\label{sec:case-matrix}
For the fixed three-dimensional examples, we use the matrix below, with $\epsilon=10^{-7}$ and memory decay $\beta=0.98$ for RMSProp or $\beta_2=0.98$ for autonomous Adam.
\[
A=\frac{1}{144115188075855872}
\begin{pmatrix}
109836008153513888 & -13893213032827795 & 12762003043063982\\ -13893213032827795 & 161280689414189920 & 24193156135327376\\ 12762003043063982 & 24193156135327376 & 86288968860418720
\end{pmatrix}.
\]

\subsection{Continuation and numerical checks}
\label{sec:continuation-checks}
\paragraph{Following a branch.}
We vary $h$ for RMSProp or $\beta_1$ for Adam and apply Newton's method from the preceding solution or a linear extrapolation. Near a turning point, we align angular origins, scale the unknowns, and use pseudo-arclength continuation. To assess possible fold bifurcations, we reduce the step size, increase the Fourier resolution, and compare aligned solutions at the same parameter. A reversal in the direction of the continuation parameter, together with distinct solutions at a common parameter and agreement under these checks, provides numerical evidence of a fold bifurcation.

\paragraph{Circle checks.}
We evaluate the invariance equations~\eqref{eq:invariant-circle-pair} on angular grids with more points than the grid used by Newton's method, and offset the evaluation points to check accuracy between the original nodes. We define the \textit{invariance error} as the $\ell^\infty$ norm of this residual. For a circle computed with $M$ nodes, we use at least $\max(2048,4M)$ evaluation points. On each $L$-point evaluation grid, we sample at $\theta_\ell=2\pi(\ell+s)/L$, $\ell=0,\ldots,L-1$, for $s=\sqrt{2}/10$, $0.371$, and $(\sqrt{5}-1)/2$. Each value of $s$ moves all evaluation points by the fraction $s$ of the grid spacing $2\pi/L$. We then repeat the check on at least $\max(4096,4M)$ points.

For the primary period-two circle in Figure~\ref{fig:primary-crossing-composite}, the invariance error is below $5\times10^{-14}$. At $h=1.0514\times10^{-6}$, we use a Newton tolerance of $2\times10^{-11}$ and require an invariance error below $10^{-10}$.

As a check that we have used enough Fourier modes for each circle component, we require the sum of squared coefficient magnitudes for $|k|>\tfrac34\lfloor M/2\rfloor$ to be less than $10^{-8}$ of the sum over all nonzero frequencies. This tests whether the coefficients are small at frequencies close to the largest frequency, $\lfloor M/2\rfloor$. We define the tangent norm as the Euclidean length $\|K_j'(\theta)\|_2$ of the derivative of the circle parameterization with respect to $\theta$, in the full state space. For a regular parameterization, this derivative is nonzero, so we require its norm to exceed $10^{-8}$ at the sampled points. We also check that the memory coordinates remain positive.

\paragraph{Torus checks.}
For a two-torus, we evaluate the invariance equations~\eqref{eq:two-torus-invariance} at angles between the computational nodes. We also check the magnitudes of the Fourier coefficients near the largest frequencies in each angle, the linear independence of the two tangent vectors, and the positivity of the memory coordinates.

For the two-tori at $h=1.0514\times10^{-6}$, we use a Newton tolerance of $10^{-12}$ and obtain errors below $10^{-14}$ on the computational grids. For a torus computed on an $M_1\times M_2$ grid, we evaluate on both a $(2M_1+1)\times(2M_2+1)$ grid and a $3M_1\times3M_2$ grid. On each evaluation grid, we add the offsets $(0.371,0.237)$ and $(0.1414,0.618)$ to $(\theta_1,\theta_2)$, with angles measured in radians. Table~\ref{tab:fourier-settings} gives the computational grids and invariance errors on these denser, offset grids.

\begin{table}[H]
\centering\small
\caption{Fourier grids and errors measured on denser, offset grids at $h=1.0514\times10^{-6}$ for the circle and torus comparisons. Grid sizes count nodes, rather than Fourier mode indices. Circle invariance errors refer to~\eqref{eq:invariant-circle-pair}, and torus invariance errors to~\eqref{eq:two-torus-invariance}.}
\label{tab:fourier-settings}
\begin{tabular}{@{}lrr@{}}
\toprule
Object & Computational grid & Maximum error\\
\midrule
$C_{\mathrm{NS}}$ & $384$ & $1.5\times10^{-14}$\\
$C_0$ & $384$ & $7.3\times10^{-15}$\\
$C_1$ & $2048$ & $1.7\times10^{-14}$\\
$C_2$ & $2048$ & $2.4\times10^{-14}$\\
Attracting two-torus & $321\times81$ & $4.6\times10^{-13}$\\
Saddle two-torus & $257\times65$ & $2.5\times10^{-12}$\\
\bottomrule
\end{tabular}
\end{table}

\subsection{Distance to an invariant circle}
\label{sec:circle-distance}
We measure Euclidean distance in the full normalized phase space. For a state $Z=(X,W)$ and a period-two circle $C=K_1(\mathbb T)\cup K_2(\mathbb T)$, we use
\begin{equation}\label{eq:supp-circle-distance}
d(Z,C)=\min_{\substack{j\in\{1,2\}\\\theta\in[0,2\pi)}}\|Z-K_j(\theta)\|_2.
\end{equation}
We first search a dense angular grid on both circle components, then refine several candidate angles using the Fourier parameterization and retain the smallest distance. To check sensitivity to the angular grid, we repeat the calculation on a second grid.
\endgroup

\begingroup
\subsection{Survey measurements and longer checks}
\label{sec:survey-measurements}
We use the classification rule in~\eqref{eq:survey-classification} for all three surveys. Table~\ref{tab:survey-settings} gives the grids, discarded and measured lengths, recurrence window $m$, and map convention. Initial states follow the convention in Section~\ref{sec:numerical-parameters}. We compute the leading full-state Lyapunov exponent as described in Section~\ref{sec:trajectory-diagnostics}, reporting it per iterate of the map listed in the final column of Table~\ref{tab:survey-settings}.

\begin{table}[H]
\centering\footnotesize
\setlength{\tabcolsep}{4pt}
\caption{Survey grids and measurement settings. Discarded and measured lengths count iterates of the map in the last column. Recurrence uses the final $m$ measured states.}
\label{tab:survey-settings}
\begin{tabular}{@{}>{\raggedright\arraybackslash}p{0.12\linewidth}>{\raggedright\arraybackslash}p{0.39\linewidth}rrrl@{}}
\toprule
Survey & Grid & Discarded & Measured & $m$ & Map\\
\midrule
RMSProp & 193 equally spaced $h\in[0.10,3.00]\times10^{-6}$, plus $h=1.00,1.06,1.92\times10^{-6}$; 161 values of $\beta\in[0,0.999]$. & $50{,}000$ & $50{,}000$ & $128$ & $G$\\[5pt]
Adam & 193 equally spaced $h\in[0.10,3.20]\times10^{-6}$, plus $h=1.92,3.00\times10^{-6}$; 181 values of $\beta_1\in[0,0.90]$. & $50{,}000$ & $50{,}000$ & $128$ & $G_{\beta_1}$\\[5pt]
Hessian orientation & $111{,}600$ ZYX Euler-angle triples in $6^\circ$ increments. First and third angles range from $-180^\circ$ to $174^\circ$, and the second from $-90^\circ$ to $90^\circ$. & $50{,}000$ & $50{,}000$ & $1{,}024$ & $G$\\
\bottomrule
\end{tabular}
\end{table}
The first two surveys use the fixed matrix~\eqref{eq:case-matrix-rounded} and $\epsilon=10^{-7}$, with $\beta_2=0.98$ for Adam. The orientation survey varies the rotations about the $z$, $y$, and $x$ axes with eigenvalues $(0.5,0.8,1.18)$, $\beta=0.98$, $\epsilon=10^{-7}$, and $h=3\times10^{-6}$.

We selected the tolerances by comparing the classifications produced by candidate rules with those obtained using more post-transient iterations. For samples outside the four categories, we compute longer trajectories and search for higher-period orbits by comparing states separated by more than two iterates.

\paragraph{Bifurcation diagram.}
For the upper panel of Figure~\ref{fig:rmsprop-master-bifurcation-diagram}, before each run we perturb every spatial coordinate by an amount in $[-5\times10^{-9},5\times10^{-9}]$ and replace each initial memory coordinate $W_i$ by $\max\{W_i,10^{-8}\}$. For the middle panel, we initialize each trajectory independently from~\eqref{eq:canonical-initial-state}.

\subsection{Individual-trajectory Lyapunov exponent measurements}
\label{sec:individual-lyapunov}
Table~\ref{tab:individual-lyapunov-values} gives the leading full-state estimates per iterate of $G$ for RMSProp and per iterate of $G_{\beta_1}$ for Adam. Discarded and measured lengths use the map indicated in each row.

\begin{table}[H]
\centering\small
\caption{Leading full-state Lyapunov exponent estimates for the individual examples.}
\label{tab:individual-lyapunov-values}
\begin{tabular}{@{}p{0.40\linewidth}rrr@{}}
\toprule
Example & Discarded & Measured & Estimate\\
\midrule
RMSProp, $d=3$ (Figure~\ref{fig:chaotic-attractor-1.06}) & $250{,}000P$ & $50{,}000P$ & $4.39\times10^{-3}$\\
RMSProp, $d=2$ (purple, Figure~\ref{fig:low-dimensional-bifurcation-clouds}) & $500{,}000G$ & $20{,}000G$ & $4.824\times10^{-2}$\\
Coupled RMSProp, $d=4$ & $55{,}000G$ & $20{,}000G$ & $6.74\times10^{-3}$\\
Coupled RMSProp, $d=5$ & $55{,}000G$ & $20{,}000G$ & $7.26\times10^{-3}$\\
Adam, $\beta_1=0$ & $50{,}000G_{\beta_1}$ & $50{,}000G_{\beta_1}$ & $2.52\times10^{-3}$\\
Adam, $\beta_1=0.01$ & $50{,}000G_{\beta_1}$ & $50{,}000G_{\beta_1}$ & $6.06\times10^{-3}$\\
Adam, $\beta_1=0.02$ & $50{,}000G_{\beta_1}$ & $50{,}000G_{\beta_1}$ & $5.83\times10^{-3}$\\
\bottomrule
\end{tabular}
\end{table}
The three-dimensional and coupled RMSProp examples use $h=1.06\times10^{-6}$, with $\eta=0.10$ in dimensions four and five. Each coupled estimate comes from one trajectory initialized as in Section~\ref{sec:coupled-extensions}, with all entries of its initial perturbation vector equal to $1/\sqrt{2d}$. The two-dimensional example uses $\rho=15$, $\beta=0.5$, and $\epsilon=10^{-7}$. The Adam examples use $h=1.92\times10^{-6}$.

\subsection{Normal Lyapunov exponent measurements}
\label{sec:normal-lyapunov}
At $h=1.0514\times10^{-6}$, the leading normal Lyapunov exponent estimates for $C_0$ and the two-tori are close to zero. We therefore compare weighted and unweighted estimates after different numbers of measured iterations, using two initial points on each circle and one on each torus. Table~\ref{tab:normal-lyapunov-settings} gives the settings for the method in Appendix~\ref{app:fourier_newton}.

\begin{table}[H]
\centering\small
\caption{Settings for the normal Lyapunov-exponent comparisons at $h=1.0514\times10^{-6}$. The Discarded and Measured columns give numbers of iterations of $P$.}
\label{tab:normal-lyapunov-settings}
\begin{tabular}{@{}lrrr@{}}
\toprule
Object & Interpolation grid & Discarded & Measured\\
\midrule
Circle & $16{,}384$ & $8{,}192$ & $65{,}536$ and $131{,}072$\\
Two-torus & $2048\times256$ & $8{,}192$ & $131{,}072$ and $262{,}144$\\
\bottomrule
\end{tabular}
\end{table}
{\small\noindent We start the circle calculations at $K(0)$ and $K(0.713)$, with angles in radians, and each torus calculation at $K(0,0)$. We evaluate the parameterizations and tangent vectors by periodic cubic interpolation on the grids in Table~\ref{tab:normal-lyapunov-settings}, checking selected torus values against direct Fourier evaluation.\par}

We count positive and negative weighted normal Lyapunov exponent estimates using a threshold of $10^{-7}$ per iteration of $P$. We estimate where $C_0$ loses normal stability by tracking the mean of its two leading normal Lyapunov exponents.

When we increase the number of measured iterations or change the initial point on a circle, we find that the weighted estimates retain their signs, while some unweighted estimates change sign (Table~\ref{tab:normal-lyapunov-checks}). With the larger number of measured iterations, the attracting and saddle tori have leading weighted normal Lyapunov exponents $-2.289\times10^{-6}$ and $+1.178\times10^{-6}$ per iteration of $P$, respectively.

\begin{table}[H]
\centering\small
\caption{Ranges of normal Lyapunov exponent estimates at $h=1.0514\times10^{-6}$, in units of $10^{-6}$ per iteration of $P$. The $C_0$ row includes its two leading normal Lyapunov exponents at two initial angles and both iteration counts. Each torus row includes its leading normal Lyapunov exponent at one pair of initial angles and both iteration counts.}
\label{tab:normal-lyapunov-checks}
\begin{tabular}{@{}lrr@{}}
\toprule
Object & Weighted average & Unweighted average\\
\midrule
$C_0$ & $[-1.227,-1.222]$ & $[-18.762,32.334]$\\
Attracting two-torus & $[-2.289,-2.285]$ & $[-4.712,0.186]$\\
Saddle two-torus & $[1.1779,1.1781]$ & $[0.9516,4.2077]$\\
\bottomrule
\end{tabular}
\end{table}

\begin{samepage}
\subsection{Higher-period motion}
\label{sec:higher-period-motion}
We use longer recurrence checks to find periodic motion missed by the fixed-point and period-two tests. For a proposed period $p$, we refine all consecutive states $z_0,\ldots,z_{p-1}$ together by solving the multiple-shooting equations
\begin{equation}\label{eq:multiple-shooting-residual}
\begin{aligned}
\mathcal R_p(z_0,\ldots,z_{p-1})
 &=\bigl(G(z_j)-z_{j+1}\bigr)_{j=0}^{p-1},\qquad z_p=z_0,\\
\mathcal R_p&=0.
\end{aligned}
\end{equation}

For Adam, we replace $G$ with $G_{\beta_1}$. We measure the residual in the $\ell^\infty$ norm,
\[
\|\mathcal R_p\|_{\ell^\infty}
=\max_{0\le j<p}\max_i\left|\bigl(G(z_j)-z_{j+1}\bigr)_i\right|.
\]
We also check proper divisors of $p$ to distinguish the reported period from a shorter one. Table~\ref{tab:survey-periodic-shooting} lists the six refined orbits, of which the first three are the magenta samples in Figure~\ref{fig:rmsprop-h-beta-diagram}. We find the other three through longer recurrence searches on samples flagged by the survey criteria.

\begin{table}[H]
\centering
\caption{Numerically refined higher-period orbits. Periods count iterates of $G$ for RMSProp and iterates of $G_{\beta_1}$ for Adam. The error is $\|\mathcal R_p\|_{\ell^\infty}$ for the residual in~\eqref{eq:multiple-shooting-residual}.}
\label{tab:survey-periodic-shooting}
\begin{tabular}{@{}lrrrr@{}}
\toprule
Map & $h/10^{-6}$ & $\beta$ or $\beta_1$ & Period & Error\\
\midrule
RMSProp & 2.652604167 & 0.93031875 & 108 & $6.66\times10^{-16}$\\
RMSProp & 2.637500000 & 0.93031875 & 214 & $6.66\times10^{-16}$\\
RMSProp & 2.199479167 & 0.93031875 & 1,814 & $8.88\times10^{-16}$\\
RMSProp & 1.920000000 & 0.93656250 & 4,706 & $7.77\times10^{-16}$\\
Adam & 1.343229167 & 0.43500000 & 6,455 & $7.77\times10^{-16}$\\
Adam & 2.150520833 & 0.65500000 & 7,077 & $3.89\times10^{-16}$\\
\bottomrule
\end{tabular}
\end{table}
\end{samepage}

\FloatBarrier
\endgroup

\section{Additional bifurcation observations}
\label{app:asym-numerical-inventory}

Table~\ref{tab:asym-numerical-inventory} records the sequences and invariant objects underlying the examples in Section~\ref{sec:bifurcation-analysis}. Here $A$ is given by~\textup{(\ref{eq:case-matrix-rounded})}, and, unless stated otherwise, we report Lyapunov exponents per iteration of $P=G^2$ and count iterations of $P$.  For the $C_{\mathrm{NS}}$ doubling candidates, we estimate the transition locations from sign changes of the leading normal Lyapunov exponent.

\begingroup\footnotesize
\setlength{\LTpre}{6pt}\setlength{\LTpost}{6pt}
\begin{longtable}{@{}>{\raggedright\arraybackslash}p{0.24\linewidth}>{\raggedright\arraybackslash}p{0.21\linewidth}>{\raggedright\arraybackslash}p{0.48\linewidth}@{}}
\caption{Observed invariant objects and transition diagnostics for $A$. Parameter values are in units of $10^{-6}$. An interval lists the lower and upper sampled learning rates between which we observe the reported change.}
\label{tab:asym-numerical-inventory}\\
\toprule
Object or sequence & $10^6h$ & Observation \\
\midrule
\endfirsthead
\multicolumn{3}{@{}l}{\tablename~\thetable\ continued}\\
\toprule
Object or sequence & $10^6h$ & Observation \\
\midrule
\endhead
\bottomrule
\endfoot
Primary period-two orbit and $C_{\mathrm{NS}}$ & $0.55222246$ & A conjugate pair of Floquet multipliers crosses the unit circle, suggesting a Neimark--Sacker bifurcation. We compute $C_{\mathrm{NS}}$ above the crossing.\\
First $C_{\mathrm{NS}}$ doubling candidate & near $0.99245872$ & A 2-fold circle is computed at $0.9950$.\\
Second doubling candidate & $[0.997750,0.997752]$ & A 4-fold circle approaches the 2-fold branch.\\
Third doubling candidate & $[1.0013690,1.0013695]$ & 8-fold circles approach the 4-fold branch.\\
Fourth doubling candidate & $[1.00175,1.00180]$ & A 16-fold circle is computed at $1.001845$.\\
Fifth doubling candidate & $[1.001839,1.0018425]$ & 32-fold circles are computed at $1.001845$ and $1.0018475$.\\
Sixth doubling candidate & $[1.001845,1.0018475]$ & We observe a sign change in the leading normal Lyapunov exponent and numerical evidence of a non-orientable leading normal line bundle.\\
Coexisting period-two circles & $1.0514$ & $C_0,C_1,C_2$ are attracting. $C_{\mathrm{NS}}$ has one unstable normal direction.\\
$C_0$ and two-tori & $1.0514$; $[1.05147230,1.05147240]$ & Attracting and saddle tori coexist with $C_0$. The mean of the two leading normal Lyapunov exponents of $C_0$ changes sign across the listed interval.\\
Chaotic trajectories & $1.0536,1.0539$ & Positive Lyapunov exponent estimates persist after $2\times10^8$ additional iterations of $P$.\\
Additional circle candidates & $1.055$ & Eight computed period-two circles, including pairs on both sides of numerical fold bifurcations and a 2-fold circle candidate in the $C_2$ family.\\
Chaotic trajectories & $1.06$ & A leading Lyapunov exponent estimate of $4.39\times10^{-3}$ per iterate of $G$ for the initial state in \eqref{eq:canonical-initial-state}.\\
Later circle-doubling sequence & $1.950,1.953,1.960,1.96150$ & Trajectories with $q=3,6,12,24$ strands. See Figure~\ref{fig:second-locking-cascade}.\\
\end{longtable}
\endgroup

\subsection{Additional doubling observations in the \texorpdfstring{$C_{\mathrm{NS}}$}{CNS} family}

In Figure~\ref{fig:cns-additional-doubling}, we compare four computed circles in the $C_{\mathrm{NS}}$ family at the same learning rate, showing the same enlarged region for every circle. Each displayed $q$-fold component is one closed curve.

We compute distinct 16-fold and 32-fold circles at $h=1.001845\times10^{-6}$. Continuing the 32-fold circle to $h=1.0018475\times10^{-6}$ changes the sign of its estimated leading normal Lyapunov exponent, as shown in Table~\ref{tab:cns-perturbation-check}. At the latter learning rate, we find numerical evidence of a non-orientable leading normal line bundle. A unit vector chosen continuously along this bundle returns with the opposite sign after one circuit of each component. Together, these observations suggest another circle doubling.

We also check whether nearby trajectories approach the 32-fold circle on either side of the sign change. At each learning rate, we follow 16 states displaced by $10^{-7}$ from the circle, perpendicular to its tangent, for 16,384 iterations of $P$. We calculate their final distances using~\eqref{eq:supp-circle-distance} and the angular refinement in Section~\ref{sec:circle-distance}. The distances agree closely on the two evaluation grids. We compare the final distances at the two learning rates in Table~\ref{tab:cns-perturbation-check}.

\begin{table}[H]
\centering\small
\caption{Normal Lyapunov exponent estimates and final distances for the 32-fold circle. Estimates are per iteration of $P$. Distances refer to all 16 perturbed states after 16,384 iterations of $P$.}
\label{tab:cns-perturbation-check}
\begin{tabular}{@{}rrl@{}}
\toprule
$h/10^{-6}$ & Leading estimate & Final distances (all 16)\\
\midrule
$1.001845$ & $-6.95\times10^{-4}$ & At most $5.45\times10^{-9}$\\
$1.0018475$ & $+4.98\times10^{-4}$ & About $3\times10^{-6}$ or more\\
\bottomrule
\end{tabular}
\end{table}

\begin{figure}[!htbp]
\centering
\includegraphics[width=0.95\linewidth]{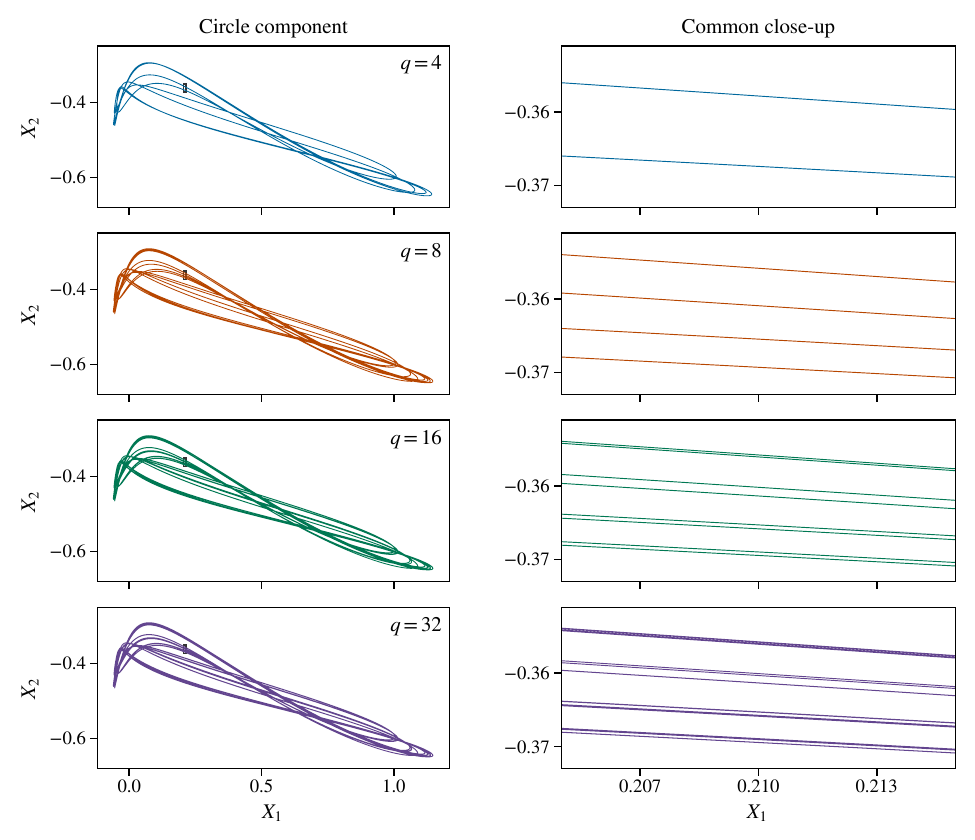}
\caption{Additional doubling observations in the $C_{\mathrm{NS}}$ family at $h=1.001845\times10^{-6}$. The left column shows Fourier parameterizations of the $q=4,8,16,32$ circle components in the $X_1$--$X_2$ projection. The right column enlarges the same boxed region with identical axis limits in every row.}
\label{fig:cns-additional-doubling}
\end{figure}

\captionsetup{skip=5pt}
\makeatletter\setlength{\@fptop}{0pt}\setlength{\@fpsep}{12pt}\setlength{\@fpbot}{0pt plus 1fil}\makeatother
\FloatBarrier
\begin{samepage}
\section{Additional numerical figures}\label{app:additional-figures}
Iterates alternate between the two components of each circle or torus pair, and $P=G^2$ maps each component to itself.
\end{samepage}
\begin{figure}[!htbp]
    \centering
    \includegraphics[width=0.95\linewidth]{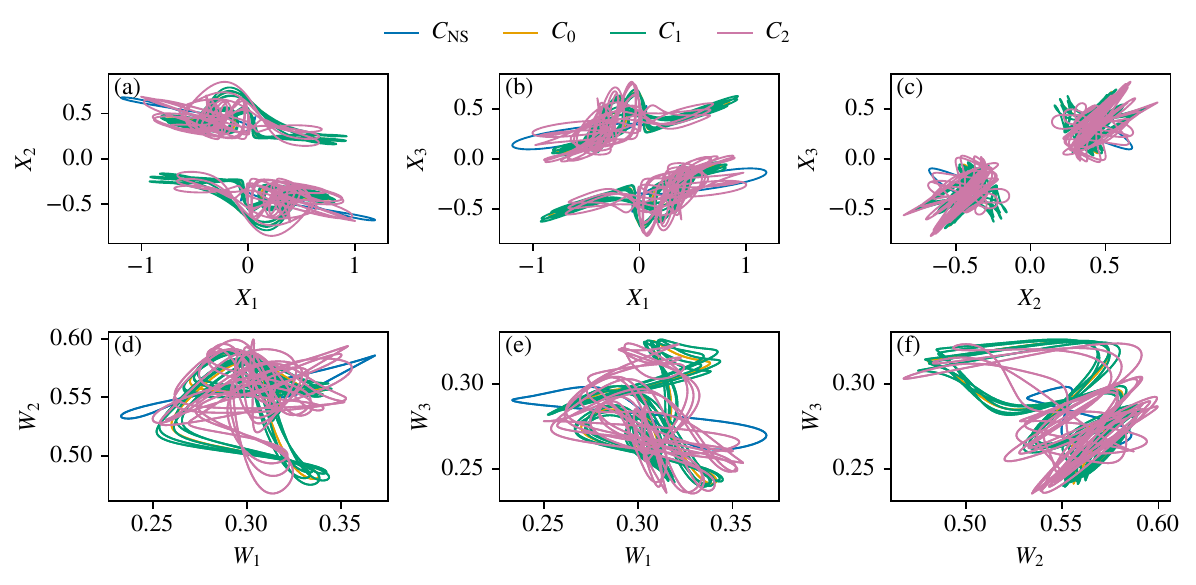}
    \caption{Four distinct period-two circles at $h=1.0514\times10^{-6}$. The upper row shows the $X_1$--$X_2$, $X_1$--$X_3$, and $X_2$--$X_3$ projections; the lower row shows $W_1$--$W_2$, $W_1$--$W_3$, and $W_2$--$W_3$. We show both components of the attracting circles $C_0,C_1,C_2$ and the saddle circle $C_{\mathrm{NS}}$.}
    \label{fig:same-plot-circles}
\end{figure}
\begin{figure}[!htbp]
    \centering
    \includegraphics[width=0.95\linewidth]{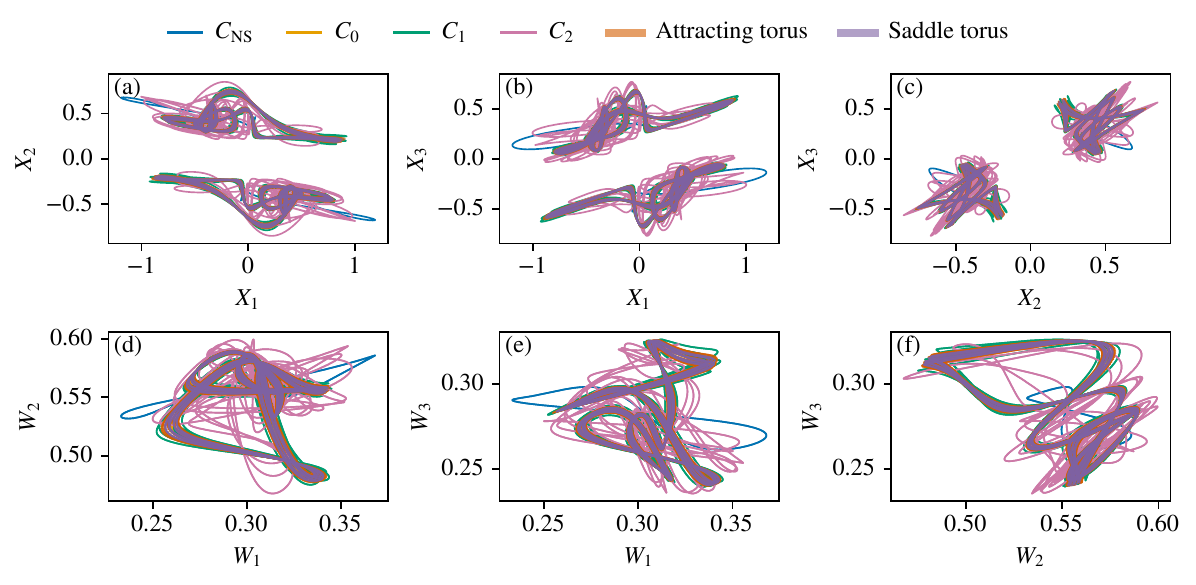}
    \caption{Four period-two circles and two invariant-torus pairs coexist at $h=1.0514\times10^{-6}$. We use the same six projections, axis bounds, and circle colors as in Figure~\ref{fig:same-plot-circles}, and add shaded points for the attracting and saddle tori. We show both components of every circle and torus pair. The tori are difficult to distinguish at this scale, so Figure~\ref{fig:c0-candidate-torus} shows their cross-sections.}
    \label{fig:six-coexisting-sets}
\end{figure}

\begin{figure}[!htbp]
    \centering
\includegraphics[width=0.95\linewidth]{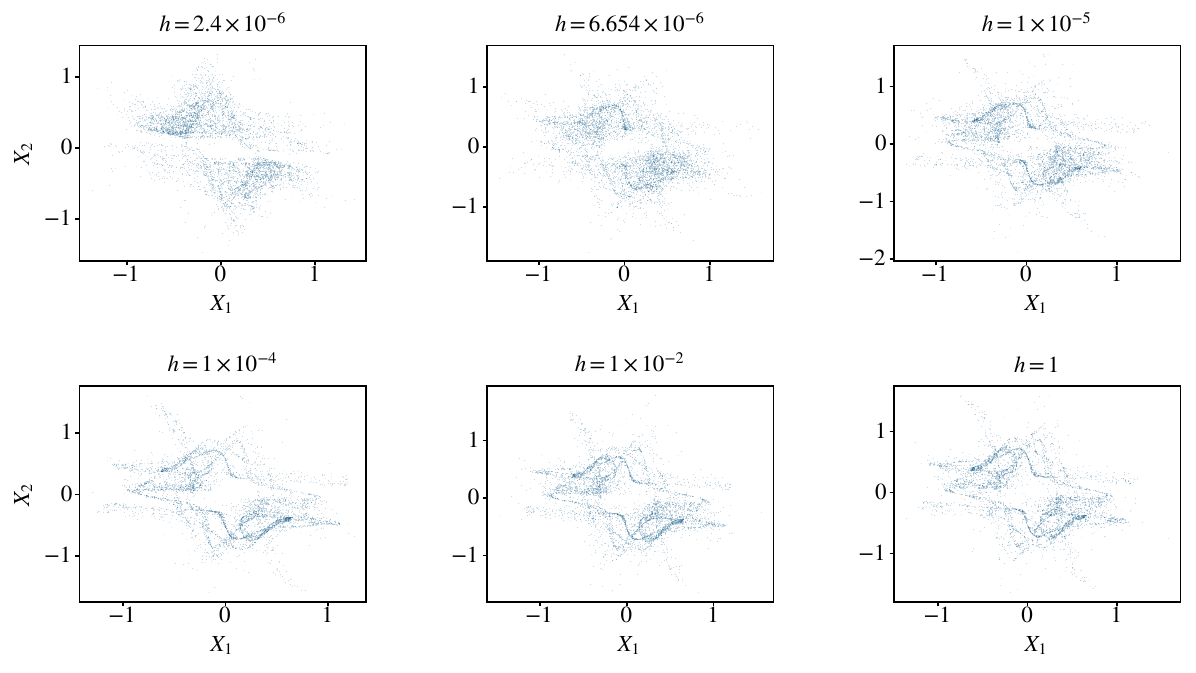}
    \caption{At large $h$, the trajectories shown occupy approximately the same spatial scale in normalized coordinates. Panels show $X_1$--$X_2$ projections at $h\in\{2.4\times10^{-6},6.654\times10^{-6},10^{-5},10^{-4},10^{-2},1\}$. For each panel, we discard $100{,}000$ iterates of $G$ and compute the next $50{,}000$ states. We plot every fourth state from the final $20{,}000$.}
    \label{fig:high-h-clouds}
\end{figure}
\clearpage
\subsection{Epsilon-inside comparison}
For the epsilon-inside map $G_I$, we use $h_{*,I}=2\sqrt{\epsilon}/\lambda_{\max}(A)$ and $\rho_I=h/h_{*,I}$. As we increase the learning rate for the same matrix $A$, we find invariant circles near $\rho_I=2$ and chaotic dynamics near $\rho_I=4$.
\begin{figure}[!htbp]
    \centering
    \includegraphics[width=\linewidth]{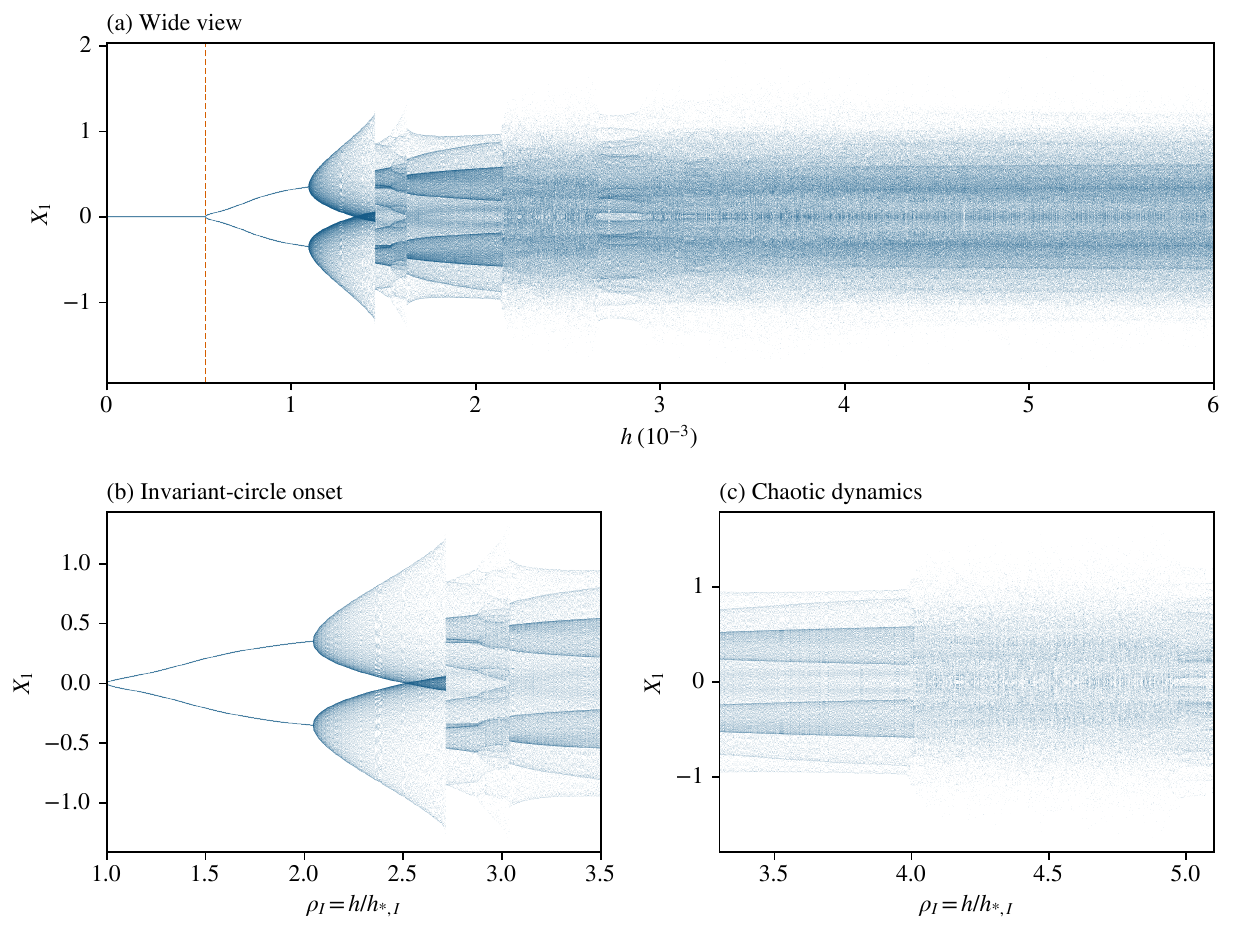}
    \caption{Bifurcation diagram for the normalized epsilon-inside map $G_I$ at the parameters in~\eqref{eq:canonical-parameters}. Panel~(a) uses a linear axis for $h$ beginning at zero and marks the instability threshold $h_{*,I}=2\sqrt{\epsilon}/\lambda_{\max}(A)$. Panels~(b) and (c) enlarge the period-two-to-invariant-circle transition and the appearance of chaotic motion using $\rho_I=h/h_{*,I}$. At each of the 9,999 parameter values, we continue from the preceding final state with a deterministic perturbation of amplitude $10^{-8}$, discard $50{,}000$ iterates of $G_I$, and plot the next 150 values of $X_1$.}
    \label{fig:epsilon-inside-bifurcation-atlas}
\end{figure}
\clearpage
\section{Additional RMSProp trajectories}\label{app:orbit-gallery}
In the six examples below, we plot bounded trajectories with positive leading finite-time Lyapunov exponents after discarding the initial transient.
\begin{figure}[!htbp]
\centering
\includegraphics[width=\linewidth]{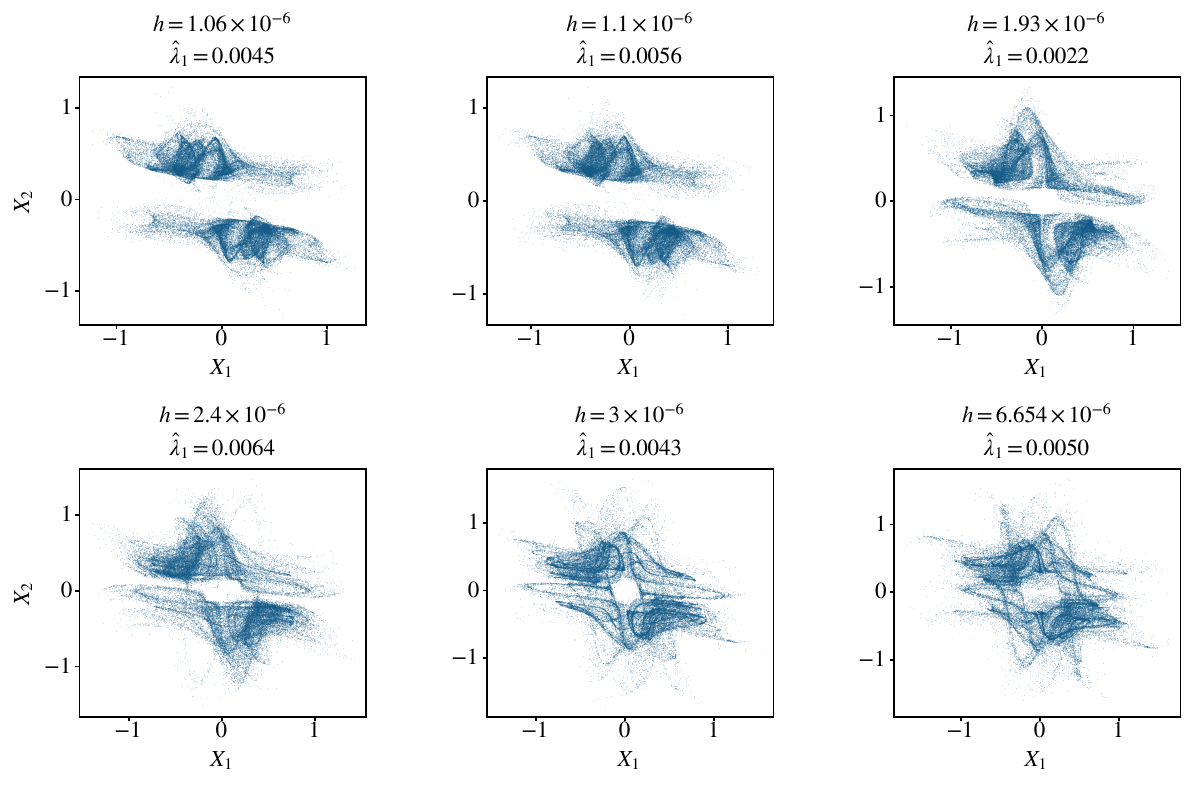}
\caption{RMSProp chaotic-dynamics gallery in the $X_1$--$X_2$ plane for $A$, $\beta=0.98$ and $\epsilon=10^{-7}$. For each panel, we start at $X=(0.3,0.4,-0.3)$, $W=(0.2,0.2,0.2)$, discard 750,000 iterates of $G$, and plot the next 50,000 iterates of $G$. The leading finite-time Lyapunov exponent (FTLE) above each panel is per iterate of $G$.}
\label{fig:strange-gallery}
\end{figure}
\clearpage
\section{Additional autonomous-Adam plots}\label{app:adam-gallery}
\captionsetup{justification=raggedright,singlelinecheck=false}
\subsection{Additional circle continuations}
\label{sec:additional-continuation-results}
We continue autonomous-Adam period-two circles in $\beta_1$ at $h=6.30261051\times10^{-7}$ and $7.80889203\times10^{-7}$. Table~\ref{tab:adam-last-circles} lists the last pair computed on each branch. For one component, we define the \textit{spatial amplitude} as the largest coordinate range along the circle, $\max_i(\max_\theta X_i(\theta)-\min_\theta X_i(\theta))$. We measure the invariance error for $G_{\beta_1}^2$ on finer, shifted angular grids.

\begin{table}[H]
\centering\small
\caption{Last computed period-two circles on two autonomous-Adam continuation branches. The invariance errors are measured for the second iterate $G_{\beta_1}^2$.}
\label{tab:adam-last-circles}
\begin{tabular}{@{}rrrr@{}}
\toprule
$h/10^{-7}$ & Last $\beta_1$ & Spatial amplitude & Invariance error\\
\midrule
6.302611 & 0.08591319 & $1.56\times10^{-6}$ & $4.52\times10^{-12}$\\ 7.808892 & 0.21550406 & $1.93\times10^{-6}$ & $9.44\times10^{-15}$\\
\bottomrule
\end{tabular}
\end{table}
At the smaller learning rate, the circles approach the period-two orbit near $\beta_1\approx0.085913$. In separate period-two calculations, we find a conjugate pair of Floquet multipliers crossing the unit circle near $\beta_1=0.08591318628$ at $h=6.30261051\times10^{-7}$ and near $\beta_1=0.21550405864$ at $h=7.80889203\times10^{-7}$. Figure~\ref{fig:adam-extra-continuation} gives further spatial and momentum projections.

\begin{figure}[!htbp]
\centering
\includegraphics[width=0.95\linewidth]{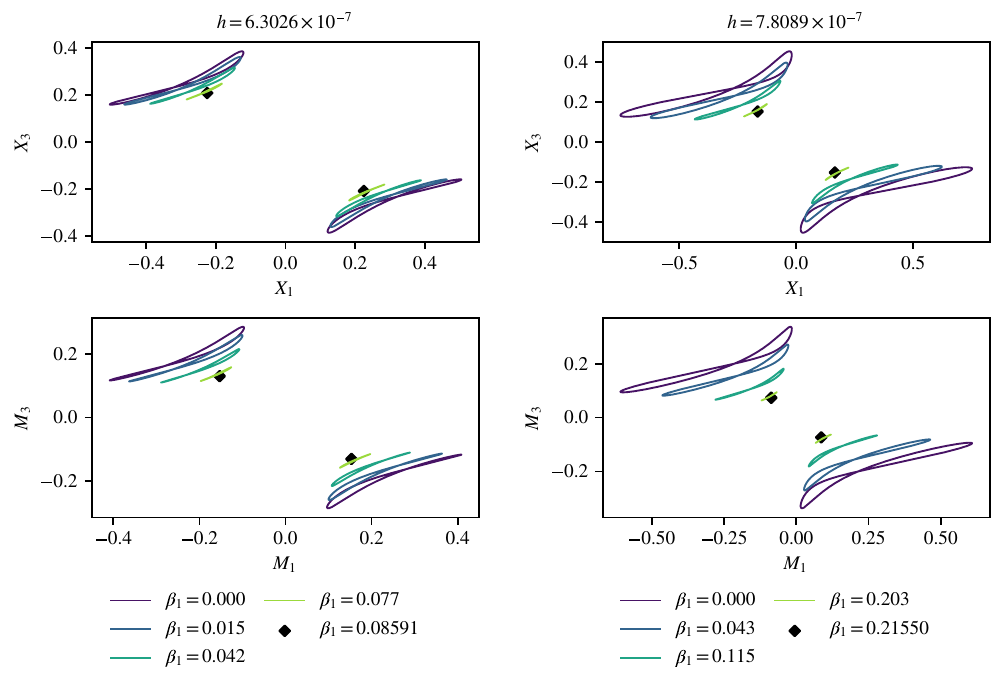}
\caption{Additional $X$ and $M$ projections of two autonomous-Adam circle continuations. The learning rates are $h=6.3026\times10^{-7},7.8089\times10^{-7}$, with $\beta_2=0.98$ and $\epsilon=10^{-7}$. Color denotes $\beta_1$, and black diamonds mark points on the last period-two circle obtained on each branch, near $\beta_1=0.085913$ and $0.21550$, respectively.}
\label{fig:adam-extra-continuation}
\end{figure}
\FloatBarrier

\clearpage
\begin{samepage}
\subsection{Trajectory examples}
\label{sec:adam-trajectory-examples}
The fixed-learning-rate trajectory figures below, Figure~\ref{fig:adam-small-chaos}, and panels E and F of Figure~\ref{fig:adam-h-beta1-diagram} use $A$, $\beta_2=0.98$, $\epsilon=10^{-7}$, and the scaled state $(X,M,W)$. Initial states follow Section~\ref{sec:numerical-parameters}. Table~\ref{tab:adam-trajectory-settings} gives the trajectory lengths.

\begin{table}[H]
\centering\small
\caption{Trajectory lengths for the fixed-learning-rate Adam trajectory figures. All lengths count iterates of $G_{\beta_1}$.}
\label{tab:adam-trajectory-settings}
\begin{tabular}{@{}lrr@{}}
\toprule
Learning rate & Discarded & Computed after discard\\
\midrule
$h=1.05\times10^{-6}$ & $100{,}000$ & $50{,}000$\\
Other learning rates & $100{,}000$ & $240{,}000$\\
\bottomrule
\end{tabular}
\end{table}
\end{samepage}

For both trajectory lengths in Table~\ref{tab:adam-trajectory-settings}, we plot two consecutive states out of every four from the final 50,000 states, retaining both even and odd iterates. The finite-time Lyapunov exponent estimates use different numbers of discarded and measured iterations, given in Table~\ref{tab:individual-lyapunov-values} for the individual examples and in the caption of Figure~\ref{fig:supp-adam_fixed_h1p05_beta1_ftle} for the $\beta_1$ sweep.
\begin{figure}[!htbp]
    \centering
    \includegraphics[width=0.95\linewidth]{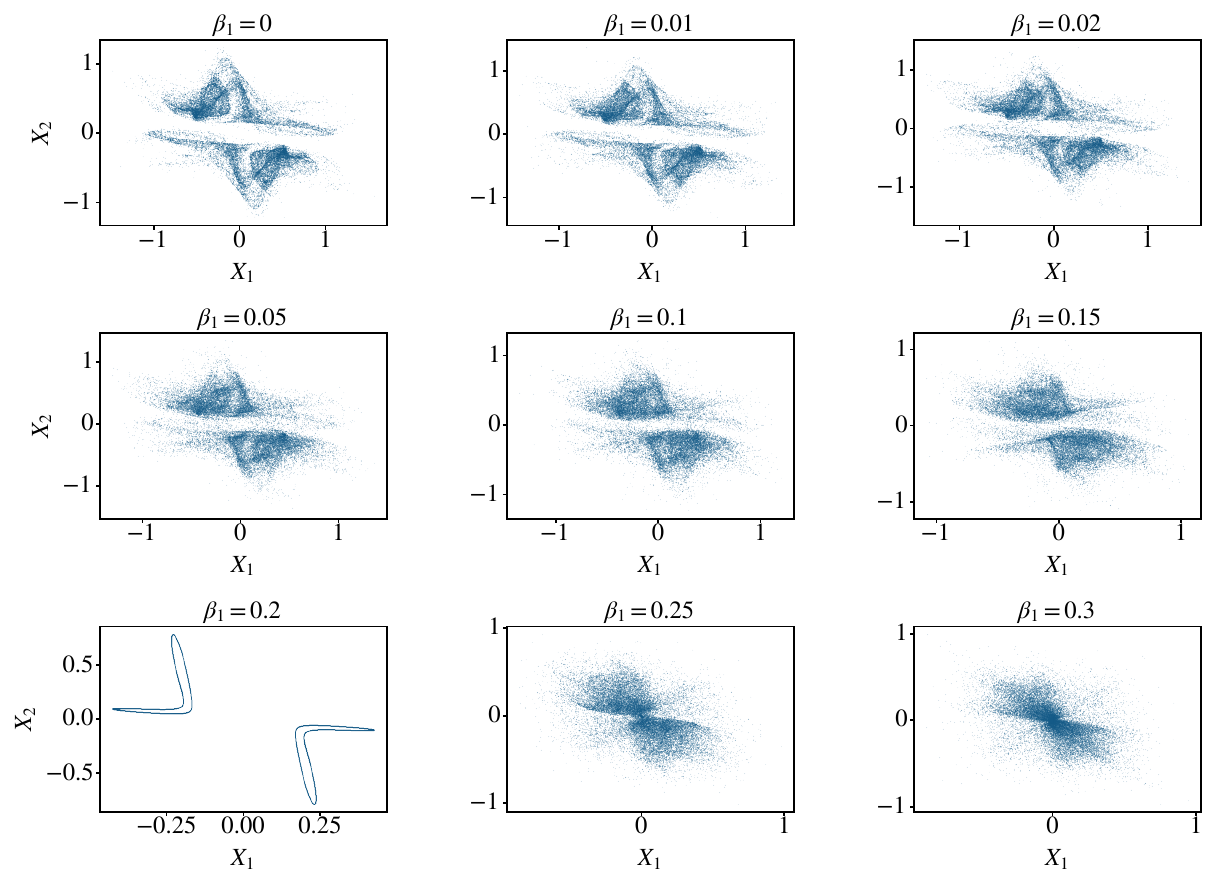}
    \caption{Spatial projections of autonomous-Adam trajectories at $h=1.92\times10^{-6}$ and $\beta_1\in\{0,0.01,0.02,0.05,0.10,0.15,0.20,0.25,0.30\}$.}
    \label{fig:adam_fixed_h1p92_beta1_node_gallery}
\end{figure}
\begin{figure}[!htbp]
    \centering
    \includegraphics[width=0.90\linewidth]{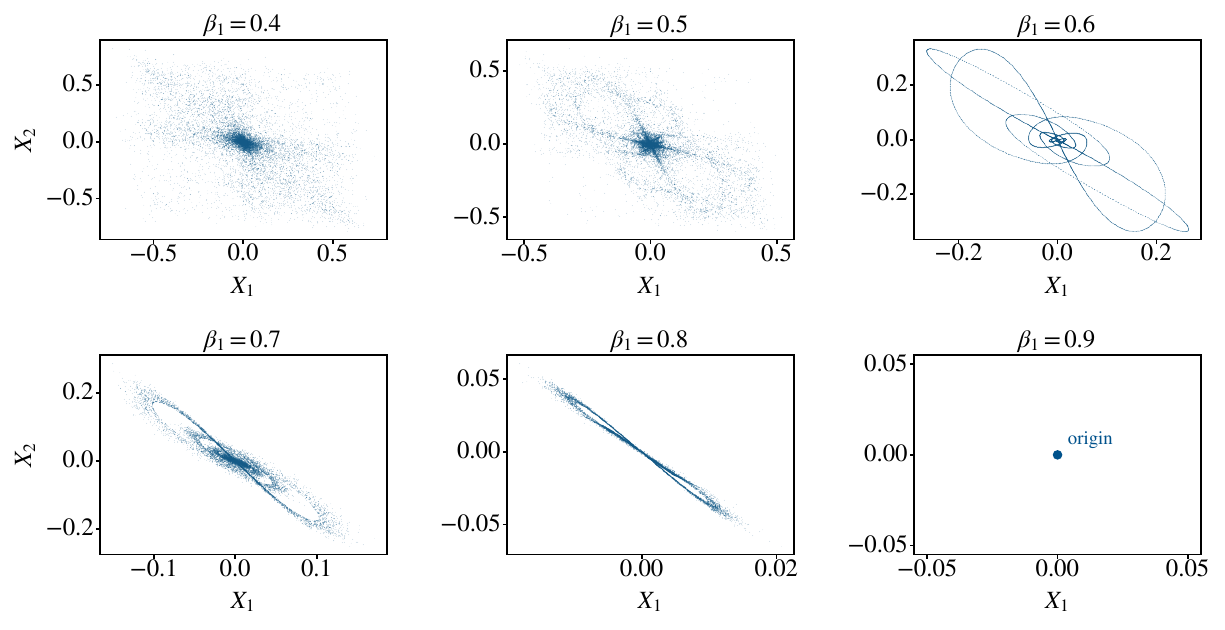}
    \caption{Spatial projections of autonomous-Adam trajectories at $h=1.92\times10^{-6}$ and $\beta_1\in\{0.4,0.5,0.6,0.7,0.8,0.9\}$. We observe chaotic motion at $\beta_1=0.8$ and convergence to the origin at $\beta_1=0.9$. The origin marker is enlarged for visibility.}
    \label{fig:adam_fixed_h1p92_beta1_high_gallery}
    \par\medskip
\includegraphics[width=0.90\linewidth]{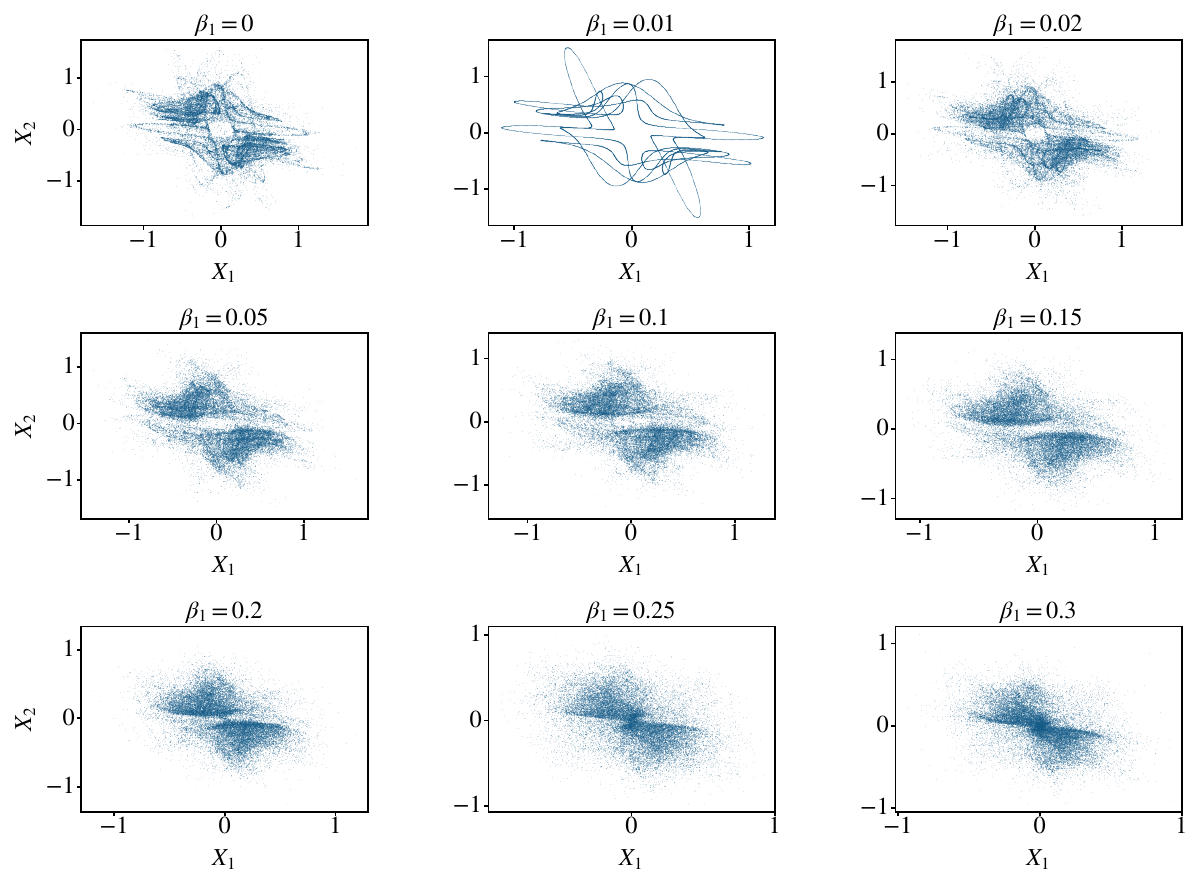}
    \caption{Spatial projections of autonomous-Adam trajectories at $h=3\times10^{-6}$ and $\beta_1\in\{0,0.01,0.02,0.05,0.10,0.15,0.20,0.25,0.30\}$.}
    \label{fig:adam_fixed_h3p00_beta1_node_gallery}
\end{figure}
\begin{figure}[!htbp]
    \centering
    \includegraphics[width=0.95\linewidth]{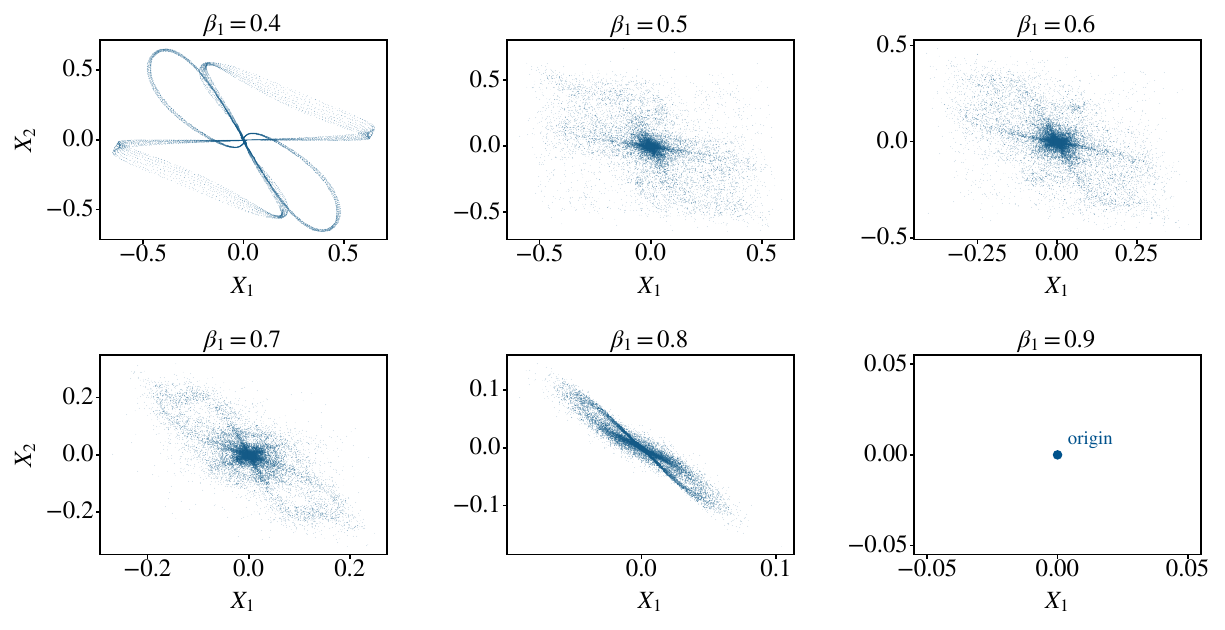}
    \caption{Spatial projections of autonomous-Adam trajectories at $h=3\times10^{-6}$ and $\beta_1\in\{0.4,0.5,0.6,0.7,0.8,0.9\}$. We observe convergence to the origin at $\beta_1=0.9$. The origin marker is enlarged for visibility.}
    \label{fig:adam_fixed_h3p00_beta1_high_gallery}
\par\medskip
\includegraphics[width=\linewidth]{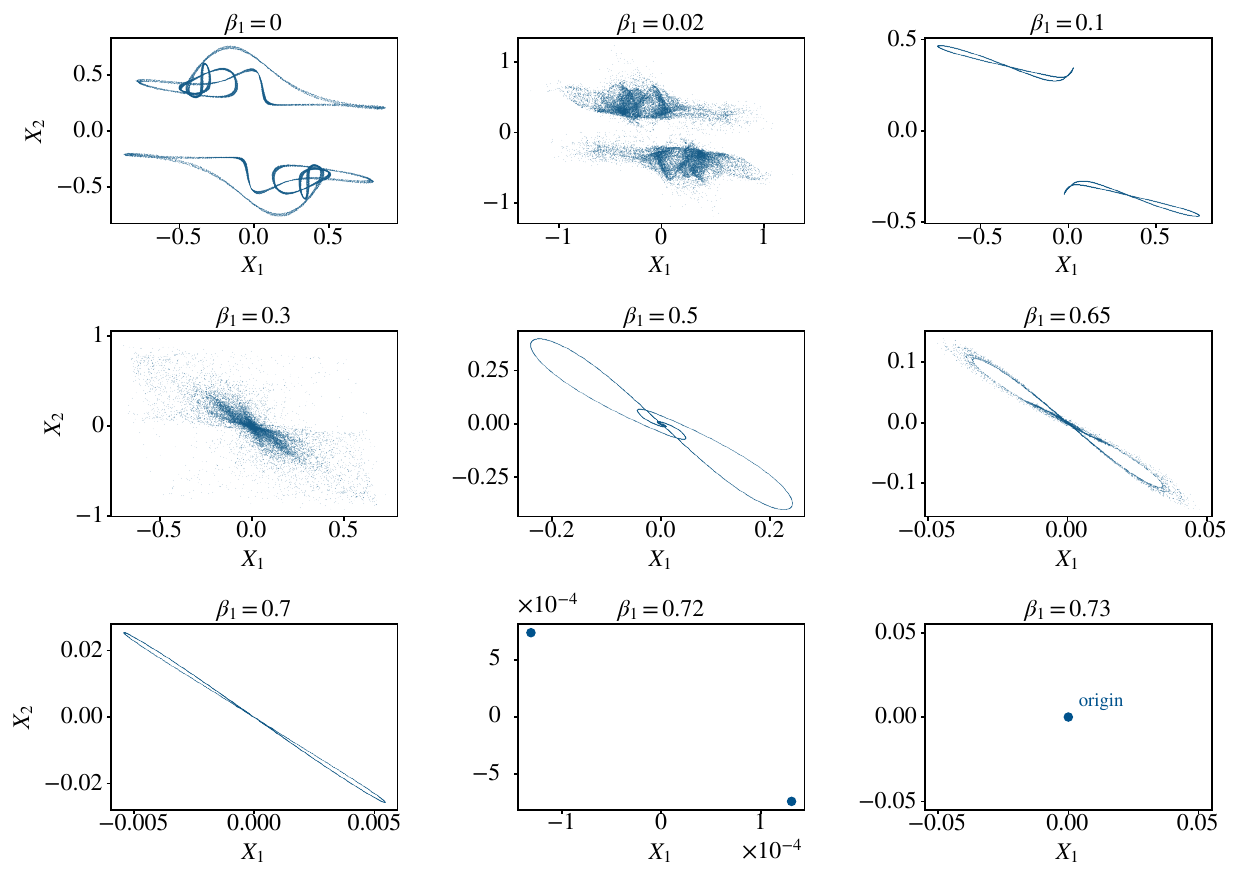}
\caption{Spatial projections of autonomous-Adam trajectories at $h=1.05\times10^{-6}$ and $\beta_1\in\{0,0.02,0.10,0.30,0.50,0.65,0.70,0.72,0.73\}$. At $\beta_1=0.73$, the trajectory converges to the origin. Markers for the period-two orbit at $\beta_1=0.72$ and the origin at $\beta_1=0.73$ are enlarged for visibility.}
\label{fig:supp-adam_fixed_h1p05_beta1_attractors}
\end{figure}
\begin{figure}[!htbp]
\centering
\includegraphics[width=\linewidth]{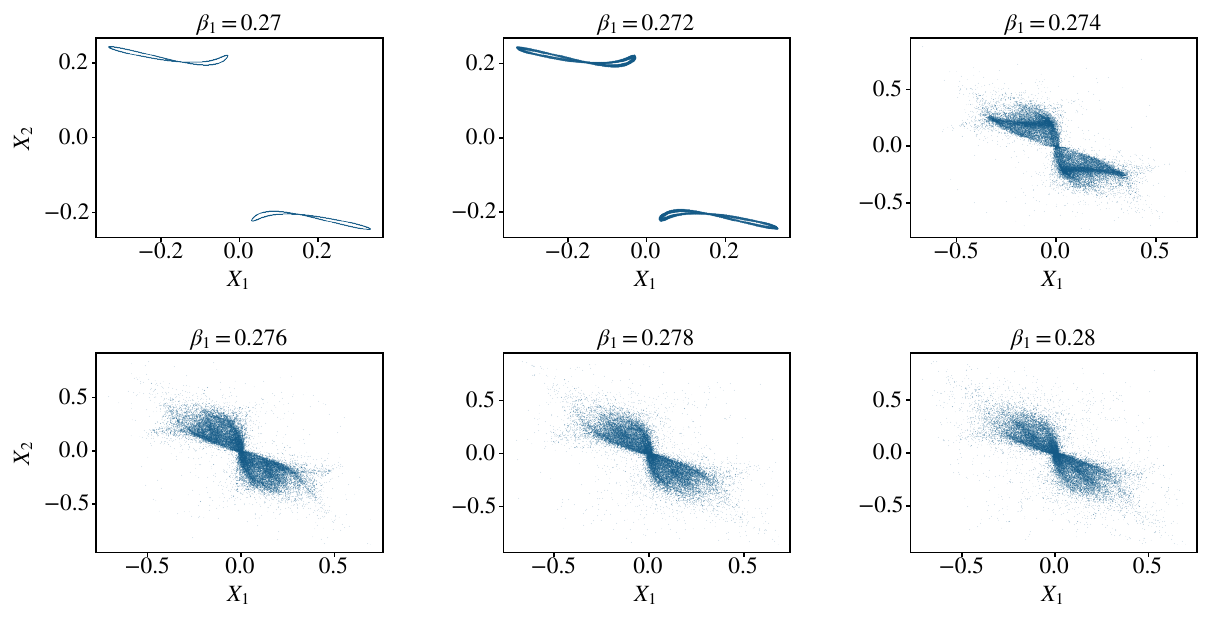}
\caption{Spatial projections of autonomous-Adam trajectories near $\beta_1=0.275$, at $h=1.05\times10^{-6}$ and $\beta_1\in\{0.270,0.272,0.274,0.276,0.278,0.280\}$.}
\label{fig:supp-adam_fixed_h1p05_beta1_near_0p3_attractors}
    \par\medskip
\includegraphics[width=\linewidth]{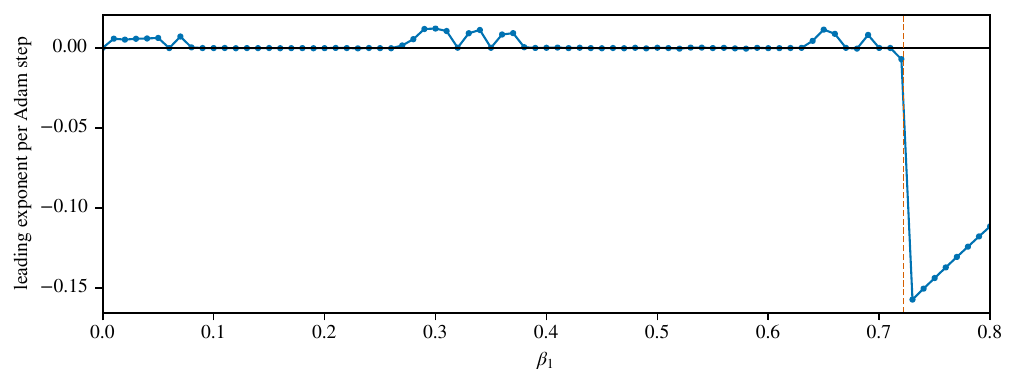}
\caption{Leading finite-time Lyapunov exponent versus $\beta_1$ at $h=1.05\times10^{-6}$, $\beta_2=0.98$, and $\epsilon=10^{-7}$. We sample $\beta_1=0,0.01,\ldots,0.80$, with 30,000 discarded and 20,000 measured iterates of $G_{\beta_1}$ per sample. We use the initial-state convention in Section~\ref{sec:numerical-parameters} and the first coordinate vector as the initial perturbation. The indicated stability boundary of the origin is $\beta_1\approx0.72203$.}
\label{fig:supp-adam_fixed_h1p05_beta1_ftle}
\end{figure}
\clearpage
\begin{figure}[H]
\centering
\includegraphics[width=\linewidth]{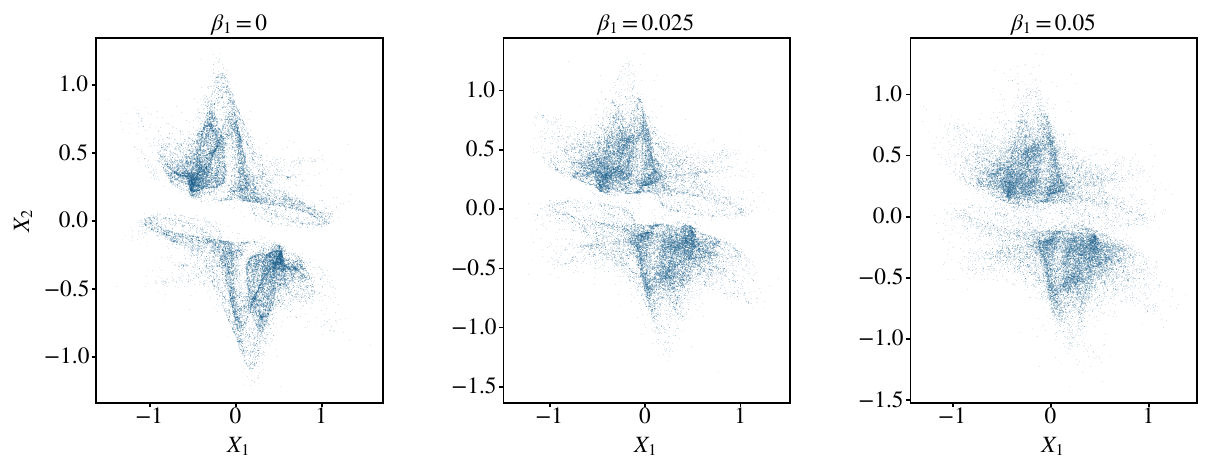}
\caption{Spatial projections at $h=1.92\times10^{-6}$ and $\beta_1=0,0.025,0.05$.}
\label{fig:supp-adam_fixed_h1p92_beta1_0_0p025_0p05_attractors}
\end{figure}

\begingroup
\let\originalthebibliography\thebibliography
\let\originalendthebibliography\endthebibliography
\renewenvironment{thebibliography}[1]{\originalthebibliography{#1}\setlength{\itemsep}{1pt plus 0.2pt}}{\originalendthebibliography}
\bibliographystyle{amsplain}
\bibliography{references}
\endgroup

\end{document}